\documentclass[twoside, 11pt]{article}   
\usepackage{jmlr2e}
\usepackage{amsmath,amssymb,color,array,multirow}
\usepackage{epsfig,amsfonts,latexsym, hyperref,  mathrsfs}
\usepackage{natbib}
\usepackage{centernot}
\usepackage{stmaryrd}
\usepackage{tikz}
\usetikzlibrary{shadows,trees}
\numberwithin{equation}{section}

\allowdisplaybreaks

\usepackage{lastpage}
\jmlrheading{21}{2020}{1-\pageref{LastPage}}{12/18; Revised
4/20}{4/20}{18-814}{M. Bhattacharjee, M. Banerjee and G. Michailidis}
\ShortHeadings{Change Point Estimation in DSBM}{Bhattacharjee, Banerjee and Michailidis}

\firstpageno{1}

\begin{document}

\title{Change Point Estimation \\  in a Dynamic Stochastic Block Model}

\author{\name Monika Bhattacharjee  \email m.bhattacharjee@ufl.edu \\
       \addr Informatics Institute\\ 
       University of Florida\\
        Gainesville, USA
       \AND
       \name Moulinath Banerjee \email moulib@umich.edu  \\
       \addr Department of Statistics\\
       University of Michigan\\
        Ann Arbor, USA
        \AND
        \name George Michailidis \email gmichail@ufl.edu\\
        \addr Informatics Institute\\ 
        Department of Statistics \\
       University of Florida\\
        Gainesville, USA}

\editor{Jie Peng}

\maketitle

\begin{abstract}%
We consider the problem of estimating the location of a single change point in a network generated by a dynamic stochastic block model mechanism. This model produces community structure in the network that exhibits change at a single time epoch. We propose two methods of estimating the change point, together with the model parameters,  before and after its occurrence. The first employs a least-squares criterion function and takes into consideration the full structure of the stochastic block model and is evaluated at each point in time. Hence, as an intermediate step, it requires estimating the community structure based on a clustering algorithm at every time point. The second method comprises the following two steps: in the first one, a least-squares function is used and evaluated at each time point, but {\em ignoring the community structure} and only considering a random graph generating mechanism exhibiting a change point. Once the change point is identified, in the second step, all network data before and after it are used together with a clustering algorithm to obtain the corresponding community structures and subsequently estimate the generating stochastic block model parameters. The first method, since it requires knowledge of the community structure and hence clustering at every point in time,  is significantly more computationally expensive than the second one. On the other hand, it requires a significantly less stringent identifiability condition for consistent estimation of the change point and the model
parameters than the second method; however, it also requires a condition on the misclassification rate of misallocating network nodes to their respective communities that may fail to hold in many realistic settings. Despite the apparent stringency of the identifiability condition for the second method, we show that networks generated by a stochastic block
mechanism exhibiting a change in their structure can easily satisfy this condition under a multitude of scenarios, including merging/splitting communities, nodes joining another community, etc. Further, for both methods under their respective identifiability and certain additional regularity conditions,  we establish rates of convergence and derive the asymptotic distributions of the change point estimators. The results are illustrated on synthetic data.
In summary, this work provides an in-depth investigation of the novel problem of change point analysis for networks generated by stochastic block models, identifies key conditions for the consistent estimation of the change point, and proposes a computationally fast algorithm that solves the problem in many settings that occur in applications. Finally, it discusses challenges posed by employing clustering algorithms in this problem, that require additional investigation for their full resolution.
\end{abstract}
   

\section{Introduction} \label{sec: intro}
The modeling and analysis of network data have attracted the attention of multiple scientific communities, due to their ubiquitous presence in many application domains; see  \cite{newman2006structure}, \cite{kolaczyk2014statistical}, \cite{crane2018probabilistic} and
references therein.  A popular and widely used statistical model for network data is the Stochastic Block Model (SBM) introduced in  \cite{HLL1983}.  It is a special case of a random graph model, where the nodes are partitioned into $K$ disjoint groups  (communities) and the edges between them are drawn independently with probabilities that only depend on the community membership of the nodes. This leads to a significant reduction in the dimension of the parameter space, from $\mathcal{O}(m^2)$ for the random graph model, with $m$ being the number of nodes in the network, to $\mathcal{O}(K^2)$ ($K<<m$).

There has been a lot of technical work on the SBM, including (i) estimation of the underlying community structure and the corresponding community  connection probabilities, for examples 
\cite{C2014,J2015,J2016,LR2015,RCY2011,SB2015,ZLZ2012};  (ii) establishing the minimax rate for estimating the SBM parameters---for examples \cite{G2015rate,K2017}---and the community structure---for examples \cite{ZZ2016,GMZZ2015}---under the assumption that the assignment problem of nodes to communities can be solved {\em exactly}. However, the latter problem is computationally NP-hard and hence estimates of the community structure and connection probabilities based on easy to compute procedures compromise the minimax rate---see \cite{ZLZ2015}.

There has also been some recent work on understanding the evolution of community structure over time, based on observing a sequence of network adjacency matrices---for examples
\cite{D2016a,D2016b,H2015,K2010,MM2017,M2015b,X2010,Xu2015,Y2011,bao2018core}. Various modeling formalism is employed including Markov random field models, low rank plus sparse decompositions  and dynamic versions of SBM (DSBM). These studies focus primarily on fast and scalable computational procedures for identifying the evolution of community structure over time. Some work that investigated theoretical properties of the DSBM and more generally graphon models assuming that the node assignment problem can be solved exactly includes \cite{P2016s}, while the theoretical performance of spectral clustering for the DSBM was examined in \cite{B2017} and \cite{P2017}. The last two studies estimate
the edge probability matrices  by either directly averaging  adjacency matrices observed at different time points, or by employing a kernel-type smoothing procedure  and extract the group memberships of the nodes by using spectral clustering.

The objective of this paper is to examine the {\em offline estimation} of a single change point under a DSBM network generating mechanism. Specifically, a sequence of networks is generated independently across time through the SBM mechanism, whose community connection probabilities exhibit a change at some time epoch. Then, the problem becomes one of identifying the change point epoch based on the observed sequence of network adjacency matrices, detecting the community structures before and after the change point, and estimating the corresponding SBM model parameters.

The existence of change points and their estimation has been well-studied for many univariate statistical models evolving independently over time and with shifts in mean and  in variance structures. A broad overview of the corresponding literature can be found in \cite{BD2013noncp} and \cite{CH1997}.  However, in many applications, multivariate (even high dimensional) signals are involved, while also exhibiting dependence across time---see the review article by \cite{AH2013} and references therein.

The emergence of data with network structure has accentuated the need to study change point problems in that context. For example, in political science, the study of political polarization has received increased attention (\cite{MM2013}) and especially the detection of regime changes (\cite{peel2015detecting}), including network-based approaches (\cite{bao2018core}). 
Recent work on change-point detection for network structured data includes \cite{wang2017hierarchical} and \cite{Wang2018}. Specifically, \cite{wang2017hierarchical} considers a generalized hierarchical random graph model, whereas \cite{Wang2018} focuses on change-point detection in sparse networks. The latter study deals with the Erd\H{o}s-R\'{e}nyi random graph model. However, to the best of our knowledge, a detailed and rigorous investigation of change-point detection for the dynamic SBM is largely lacking.

Therefore, the key contributions of this paper are threefold: first, the development of a computational strategy for solving the problem and establishing its theoretical properties under suitable regularity conditions, including (i) establishing the rate of convergence for the least-squares estimate of the change-point and (ii) the DSBM parameters, as well as (iii) deriving the asymptotic distribution of the change-point. An important step in the strategy for obtaining an estimate of the change-point involves clustering the nodes to communities, for which we employ a spectral clustering algorithm that exhibits cubic computational cost in the number of edges in the adjacency matrix. However, the theoretical analysis of the first method which involves clustering \emph{at every time point} requires imposing a rather stringent assumption on the rate of misclassifying nodes to communities. For these reasons, the second key contribution of this work is the introduction of a two-step computational strategy, wherein the first step,  the change-point is estimated based on a procedure that {\em ignores} the community structure, while in the second step the pre- and post-change-point model parameters are estimated using a spectral clustering algorithm, but at \emph{a single time point}. It is established that this strategy yields consistent estimates
for the change-point and the community connection probabilities, at a linear computational cost in the number of edges. However, the procedure requires a  stronger identifiability condition compared to the first strategy. Naturally, no additional condition of controlling the rate
of misclassifying nodes to communities during the first step is required. The third contribution of the paper is to show that the more stringent identifiability condition under the second strategy is easily satisfied in a number of scenarios by the DSBM, including splitting/merging communities and reallocating nodes to other communities before and after the change-point. 
Overall, this work provides valuable insights into the technical challenges of change-point analysis for DSBM and also  an efficient computational strategy that delivers consistent estimates of all the model parameters. Nevertheless, the challenges identified require further investigation for their complete resolution, as discussed in Section \ref{sec:discuss}.

The remainder of the paper is organized as follows. In Section \ref{sec: DSBM}, we introduce the DSBM model together with the necessary definitions and notation for technical development. Subsequently, we present the strategy to detect the change-point that involves a {\em community detection step at each time point}, followed by estimation of the DSBM model parameters together with the asymptotic properties of the estimators.
In Section \ref{sec: 2step}, we introduce a 2-step computational strategy for the DSBM change-point detection problem, which is computationally significantly less expensive and discuss the consistency of these estimators.  Section \ref{sec: compare} involves a comparative study between the two change-point detection strategies previously presented and also  provides many realistic settings where the computationally fast 2-step algorithm is provably consistent. The numerical performance of the two strategies based on synthetic data is illustrated in Section \ref{subsec: simulation}.  We briefly discuss  other community detection methods for DSBM involving a single change-point  in Section \ref{sec:discuss}. Finally, the asymptotic distribution of the change-point estimates along with a data-driven procedure for identifying the correct limiting regime is presented in Section \ref{sec: ADAP}.  Some concluding remarks are drawn in Section \ref{sec:concluding-remarks}. All proofs and additional technical material are presented  in Section \ref{sec: proofs}.

\section{The dynamic stochastic block model (DSBM)} \label{sec: DSBM}
The structure of the SBM is determined by the following parameters: (i) $m$ the number of nodes or vertices  in the network, (ii) a symmetric  $K \times K$ matrix $\Lambda = ((\lambda_{ij}))_{K \times K}$ containing the community connection probabilities and (iii) a partition of the node set $\{1,2,\ldots, m\}$ into $K$ communities, which is represented by a many-to-one onto map $z: \{1,2,\ldots, m\} \to \{1,2,\ldots, K\}$ for some $K \leq m$.
Hence, for each $1 \leq i \leq m$, the community of node $i$ is determined by $z(i)$, or equivalently
\begin{eqnarray}
\text{$l$-th community} = \mathcal{C}_l = \{i \in \{1,2,\ldots, m\}: z(i)=l\}\ \forall 1 \leq l \leq K. \nonumber
\end{eqnarray}
The map $z$ determines the community structure under the SBM. The observed edge set of the network is obtained as follows: any two nodes $i \in \mathcal{C}_l$ and $j \in \mathcal{C}_{l^\prime}$ are connected by an edge with probability $\lambda_{ll^\prime}I(i \neq j)$, independent of any other pair of nodes. Self-loops are not considered and hence the probability of having an edge between nodes $i$ and $j$ is $0$ whenever $i=j$. Henceforth, we use SBM$(z,\Lambda)$ for denoting an SBM with community structure $z$ and community connection probability matrix $\Lambda$. Next, let
\begin{eqnarray}
\text{Ed}_{z}(\Lambda) = ((\lambda_{z(i)z(j)}I(i \neq j)))_{m \times m}, \nonumber
\end{eqnarray}
which is the edge probability matrix whose $(i,j)$-th entry represents the probability of having an edge between nodes $i$ and $j$. Note that we are dealing with undirected networks and thus $\Lambda$ and Ed$_{z}(\Lambda)$ are symmetric matrices.  

The data come in the form of an observed square symmetric matrix $A = ((A_{ij}))_{m \times m}$ with entries
\begin{eqnarray} \nonumber
A_{ij} = \begin{cases}
1,\ \ \text{if an edge between nodes $i$ and $j$ is observed} \\
0,\ \ \text{otherwise}.
\end{cases}
\end{eqnarray}
An adjacency matrix $A$ is said to be generated according to SBM$(z,\Lambda)$, if $A_{ij} \sim$ \\ $\text{Bernoulli}(\Lambda_{z(i)z(j)}I(i \neq j))$ independently, and is
denoted by $A \sim \text{SBM}(z,\Lambda)$. It is easy to see that all diagonal entries of $A$ are $0$.   

In a DSBM, we consider a sequence of stochastic block models evolving {\em independently} over time, with  $A_{t,n} = ((A_{ij,(t,n)}))_{m \times m}$ denoting the adjacency matrix at time point $t$. Hence, $A_{t,n} \sim \text{SBM}(z_t,\Lambda_t)$ independently, with $n$ being the total number of time points available, and note $\{A_{t,n}: 1 \leq t \leq n, n \geq 1\}$ forms a triangular sequence of adjacency matrices. Further, in the technical analysis we assume that both the total number of time points and the number of nodes  are growing to infinity, that is $n,m \to \infty$. Moreover, the number of communities depends on both the number of nodes and time,  and grows to infinity, that is, $K = K_{m,n} \to \infty$ as $m,n \to \infty$.

We are interested in a DBSM exhibiting a single change-point and its estimation in an offline manner. For presenting the results, we embed all the time points in the $[0,1]$ interval, by dividing them by $n$.  Suppose $\tau_{m,n} \in (c^{*},1-c^{*})$ corresponds to the change-point epoch in the DSBM.  Hence,
\begin{eqnarray}
A_{t,n} \sim \begin{cases} \text{SBM}(z,\Lambda),\ \ \text{if $1 \leq t \leq \lfloor n\tau_{m,n} \rfloor$} \\
\text{SBM}(w,\Delta),\ \ \text{if $\lfloor n\tau_{m,n} \rfloor < t <n$}, \label{eqn: dsbmmodel}
\end{cases}
\end{eqnarray}
and $z \neq w$ and/or $\Lambda \neq \Delta$.  Note that $z$ and $w$ correspond to the pre- and post-change-point community structures, respectively. Similarly $\Lambda$ and $\Delta$ are the pre- and post-change-point community connection probability matrices, respectively. Further, note that $z, w, \Lambda$ and $\Delta$ may depend on $m$ and  $n$.   

Our objective is to estimate the model parameters $\tau_{m,n}, z,w,\Lambda$ and $\Delta$. Throughout this paper, we assume that $0 < c^{*} <\tau_{m,n} <1 - c^{*} <1$ with $c^{*}$ being known to avoid unnecessary technical complications if the true change-point is located arbitrarily close to boundary time points. We also assume that the total number of communities before and after the change-point are equal. \emph{Even if they are different, our results continue to hold after  replacing $K$ by the maximum of the number of communities to the left and the right.} 

To identify the change-point, we employ the following least-squares  criterion function
\begin{eqnarray}
\tilde{L} (b,z,w,\Lambda,\Delta) &=& \frac{1}{n}\sum_{i,j=1}^{m} \bigg[\sum_{t=1}^{nb} (A_{ij,(t,n)} - \lambda_{z(i)z(j)})^2 + \sum_{t=nb+1}^{n} (A_{ij,(t,n)} - \delta_{w(i)w(j)})^2 \bigg]. \hspace{1 cm}\label{eqn: estimateccc}
\end{eqnarray}

\begin{remark} \label{use-likelihood-criterion}
Note that the Bernoulli likelihood function criterion could also be used that will yield similar results for the change-point estimators, but it will need  stronger assumptions and will involve more technicalities compared to the least-squares criterion function adopted. More details on the likelihood criterion function for detecting a change-point in a random graph model can be found in \cite{BBM2017}. Results on the maximum likelihood estimator of the change-point in a random graph model are consequences of the results in \cite{BBM2017}. However, in DSBM one also needs to address the problem of assigning nodes to their respective communities, which makes the problem technically more involved as shown next. However, the main message of this paper will remain the same, irrespective of employing a likelihood or a least-squares criterion. 
\end{remark}

\noindent Let 
\begin{eqnarray}
S_{u,z} &=& \{i \in \{1,2\ldots, m\}: z(i) = u\},\ \ \ s_{u,z} = |S_{u,z}|, \nonumber  \\
S_{u,w} &=& \{i \in \{1,2\ldots, m\}: w(i) = u\},\ \ \ s_{u,w} = |S_{u,w}|  \label{eqn: blocksize}
\end{eqnarray}
denote the $u$-th block and block size under the community structures $z$ and $w$.  Also define,
\begin{eqnarray}
\tilde{\Lambda}_{z,(b, n),m} &=& ((\tilde{\lambda}_{uv,z,(b,n),m}))_{m \times m},\ \   \tilde{\lambda}_{uv,z,(b,n),m} = \frac{1}{nb} \frac{1}{s_{u,z} s_{v,z}} \sum_{t=1}^{nb} \sum_{\stackrel{i: z(i) =u}{j: z(j)=v}}  A_{ij,(t,n)},\ \ \ \ \nonumber \\
\tilde{\Delta}_{w,(b,n),m} &=& ((\tilde{\delta}_{uv,w,(b,n),m}))_{m \times m},\ \  \tilde{\delta}_{uv,w,(b,n),m} =  \frac{(n(1-b))^{-1}}{s_{u,w}s_{v,w}} \sum_{t=nb+1}^{n} \sum_{\stackrel{i: w(i)=u}{j: w(j)=v}} A_{ij,(t,n)}. \ \ \ \ \ \   \label{eqn: estimateb} 
\end{eqnarray}
Define the \textit{sparsity parameter} $\rho_{m,n} = \max_{1\leq u \leq v \leq K} \{\Lambda_{u,v}, \Delta_{u,v}\}$ which may depend on both $m$ and $n$. The DSBM in (\ref{eqn: dsbmmodel}) is dense if $\inf_{m,n}\rho_{m,n} > C>0$ and sparse if $\rho_{m,n} \to 0$ as $m,n \to \infty$. 

We start our analysis by assuming that the {\bf community structures} $z$ and $w$ are {\bf known}. In that case, an estimate of the change-point can  be obtained by solving
\begin{eqnarray}
\tilde{\tau}_{m,n} &=& \arg \min_{b \in (c^{*},1-c^{*})} \tilde{L}(b,z,w,\tilde{\Lambda}_{z,(b,n),m},\tilde{\Delta}_{w,(b,n),m}).   \label{eqn: estimateintro1}
\end{eqnarray}

The following signal-to-noise condition guarantees that the change-point is identifiable under a known community structure.
\vskip 3pt
\noindent \textbf{SNR-DSBM:} $\frac{n}{K^2 \rho_{m,n}} ||\text{Ed}_{z}(\Lambda) - \text{Ed}_{w}(\Delta)||_F^2 \to \infty$ 
\vskip 3pt
Intuitively, it implies that the  signal per connection probability parameter, after scaling by the sparsity parameter, needs to grow faster than $1/\sqrt{n}$, which is in accordance with
identifiability conditions for other change-point problems (for example see \cite{kosorok2008introduction} Section $14.5.1$).
\vskip 3pt
The following assumption on the sparsity parameter restricts its rate of decay and ensures  consistency of $\tilde{\tau}_{m,n}$ scaled by the sparsity parameter with details given in Remark \ref{rem: sparseknowncomm1}.
\vskip 3pt
\noindent \textbf{Sparse-DSBM:} $\rho_{m,n} > CK^{-2}$ for some $C>0$.
\vskip 3pt
 It is easy to see that Sparse-DSBM is always satisfied by a dense DSBM.  

\indent The following Theorem  establishes  asymptotic properties for  $\tilde{\tau}_{m,n}$.  Its proof is similar  (albeit much simpler in structure) to the proof of Theorem \ref{lem: b1}, where we deal with unknown community structures. Hence, it is omitted.

\begin{theorem} \label{thm: dsbmknown}  \textbf{(Convergence rate of ${\tilde{\tau}}_{m,n}$)} 
Suppose SNR-DSBM  and Sparse-DSBM  hold.  Then, 
$$n\rho_{m,n}^{-1} ||\text{Ed}_z( \Lambda) - \text{Ed}_w(\Delta) ||_{F}^2 ({\tilde{\tau}}_{m,n} - \tau_{m,n}) = O_{\text{P}}(1).$$ 
\end{theorem}

\begin{remark} \label{rem: sparseknowncomm1}
As stated in Theorem \ref{thm: dsbmknown},  the convergence rate of the appropriately centered and scaled change-point estimator $\rho_{m,n}^{-1}({\tilde{\tau}}_{m,n} - \tau_{m,n})$ is $n^{-1} ||\text{Ed}_z( \Lambda) - \text{Ed}_w(\Delta) ||_{F}^{-2}$.  In the presence of SNR-DSBM, Sparse-DSBM ensures that the convergence rate $n^{-1} ||\text{Ed}_z( \Lambda) - \text{Ed}_w(\Delta) ||_{F}^{-2}$ decays to $0$.   Sparse-DSBM can be weakened further under stronger signal-to-noise condition, for example see Remark \ref{rem: sparseunknowncomm2}. 
\end{remark}

\begin{remark} 
In the ensuing Theorem \ref{lem: b1}, we will establish that  the DSBM change-point estimator $\tilde{\tau}_{m,n}$  with an {\em unknown community 
structure} (that needs to be estimated from the available data) has exactly the same convergence rate as the one posited in Theorem \ref{thm: dsbmknown}. However, a much stronger identifiability condition than SNR-DSBM and Sparse-DSBM is needed, since more parameters are involved.
\end{remark}

Recall the estimates in (\ref{eqn: estimateb}) given by $\tilde{\Lambda} = (( \tilde{\lambda}_{ab,z,(\tilde{\tau}_{m,n}, n),m}))_{K \times K}$ and \\$\tilde{\Delta}  =(( \tilde{\delta}_{ab,w,(\tilde{\tau}_{m,n}, n),m}))_{K \times K}$. The edge probability matrices $\text{Ed}_{z}(\Lambda)$ and $\text{Ed}_{w}(\Delta)$ can also be estimated by $\text{Ed}_{z}(\tilde{\Lambda})$ and $\text{Ed}_{w}(\tilde{\Delta})$, respectively. The following Theorem provides the convergence rate of the corresponding estimators. Its proof is similar (and structurally simpler) to the proof of Theorem \ref{thm: c3} where we deal with unknown community structures and hence omitted.

\begin{theorem} \label{thm: b2} \textbf{(Convergence rate of edge probabilities when $z$ and $w$ are known)}
\\
Suppose SNR-DSBM holds.  Let $\mathcal{S}_{m,n} = \min(\min_{u} s_{u,z}, \min_{u} s_{u,w})$. Then
\begin{eqnarray}
&& \frac{1}{K^2}||\tilde{\Lambda} - \Lambda||_F^2,\ \ \frac{1}{K^2}||\tilde{\Delta} - \Delta||_F^2 = O_{\text{P}}\left( \frac{I(n>1)\rho_{m,n}^{2}}{n^2 ||\text{Ed}_z(\Lambda)-\text{Ed}_w(\Delta) ||_F^4} + \frac{\rho_{m,n}\log K}{n \mathcal{S}_{m,n}^2}\right),\ \  \nonumber \\
&&  \frac{1}{m^2} ||\text{Ed}_{z}(\tilde{\Lambda}) - \text{Ed}_{z}(\Lambda)||_F^2, \ \  \frac{1}{m^2}||\text{Ed}_{w}(\tilde{\Delta}) - \text{Ed}_{w}(\Delta)||_F^2 \nonumber \\
&& \hspace{4cm}= O_{\text{P}}\left(\frac{I(n>1)\rho_{m,n}^2}{n^2 ||\text{Ed}_z(\Lambda)-\text{Ed}_w(\Delta) ||_F^4}+\frac{\rho_{m,n}\log m}{n \mathcal{S}_{m,n}^2}\right).
\end{eqnarray}
\end{theorem}

Note that $\tilde{\Lambda} = (( \tilde{\lambda}_{ab,z,(\tilde{\tau}_{m,n}, n),m}))_{K \times K}$. To compute the rate for $\frac{1}{K^2}||\tilde{\Lambda} - \Lambda||_F^2$, we have  $(\tilde{\lambda}_{ab,z,(\tilde{\tau}_{m,n}, n),m} - \lambda_{ab})^2 \leq 2(\tilde{\lambda}_{ab,z,(\tilde{\tau}_{m,n}, n),m}-\tilde{\lambda}_{ab,z,({\tau}_{m,n}, n),m})^2 + 2(\tilde{\lambda}_{ab,z,({\tau}_{m,n}, n),m}-\lambda_{ab})^2 \equiv T_1 + T_2$. It is easy to see that the first term $T_1$ is dominated by $(\tilde{\tau}_{m,n} -\tau_{m,n})^2$ and thus by Theorem \ref{thm: dsbmknown}, $(\tilde{\tau}_{m,n} -\tau_{m,n})^2 = O_{\text{P}}\left( \frac{I(n>1)\rho_{m,n}^2}{n^2 ||\text{Ed}_z(\Lambda)-\text{Ed}_w(\Delta) ||_F^4}\right)$.  
Moreover, the rate of $T_2$ is $\frac{\rho_{m,n}\log K}{n \mathcal{S}_{m,n}^2}$. Details are given in Section \ref{subsec: c3}. Similar arguments work for the other matrices presented in Theorem \ref{thm: b2}. 

\begin{remark} \label{rem: parameteroracle}
\textbf{(Rate for $n=1$)}. If $n=1$, then there is no change-point and $T_1$ does not appear. In this case, we have only one community structure (say) $z$ and one community connection matrix (say) $\Lambda$. Moreover,  the number of communities $K=K_m$, sparsity parameter $\rho_m = \max_{1 \leq u, v \leq K} \lambda_{uv}$ and the minimum block size $\mathcal{S}_m = \min_{u} s_{u,z}$ depend only on $m$.  Estimation of $\Lambda$ for $n=1$ is studied in \cite{ZLZ2015} when the community structure $z$ is unknown. In this remark, we assume that $z$ is known. We estimate $\Lambda$ by $\tilde{\Lambda} = (( \tilde{\lambda}_{ab,z,(1/n, n),m}))_{K \times K}$. Then,   
\begin{eqnarray}
&& \frac{1}{K^2}||\tilde{\Lambda} - \Lambda||_F^2 
= O_{\text{P}}\left(  \frac{\rho_m\log K}{\mathcal{S}_m^2}\right),\ \  
\frac{1}{m^2} ||\text{Ed}_{z}(\tilde{\Lambda}) - \text{Ed}_{z}(\Lambda)||_F^2 
= O_{\text{P}}\left(\frac{\rho_m\log m}{ \mathcal{S}_m^2}\right). \nonumber
\end{eqnarray}
Similar results for the case of unknown communities are discussed in Remark \ref{rem: parameter} and Section \ref{sec:discuss}. 
\end{remark}
In Theorem \ref{thm: c3}, we establish the results for the same quantities in the case of unknown community structures. It will be seen that the convergence rate of $\tilde{\Lambda}$ and $\tilde{\Delta}$ given above, is much sharper compared to the case of unknown community structures, despite using repeated observations in the latter one; see also discussion in Remark \ref{rem: paraestimation}. 

The real problem of interest is when the community structure is {\bf unkown} and needs to be estimated from the observed sequence of adjacency matrices along with the change-point. A standard strategy in the change-point analysis is to optimize the least-squares criterion function $\tilde{L}(b,z,w,\Lambda,\Delta)$ posited above with respect to {\em all} the model parameters. This becomes challenging both computationally since one needs to find a good assignment of nodes to communities, and technically, since for any time point away from the true change-point the node assignment problem needs to be solved under a misspecified model; namely, the available adjacency matrices are generated according to both the pre- and post-change-point community connection probability matrices.

A natural estimator of $\tau_{m,n}$ can be obtained by solving 
\begin{eqnarray}
\tilde{\tilde{\tau}}_{m,n} &=& \arg \min_{b \in (c^{*},1-c^{*})} \tilde{L} (b,\tilde{z}_{b,n,m},\tilde{w}_{b,n,m},\tilde{\Lambda}_{\tilde{z}_{b,n,m},(b,n),m},\tilde{\Delta}_{\tilde{w}_{b,n,m},(b,n),m}),  \label{eqn: estimateintro}
\end{eqnarray}
where $\tilde{z}_{b,n,m}$ and $\tilde{w}_{b,n,m}$ are obtained using the clustering algorithm from \cite{B2017} (details below). While other clustering algorithms can also be employed, and are discussed in Section \ref{sec:discuss}, all clustering algorithms incur some degree of misclassification (while assigning nodes to communities) which must be suitably controlled by an appropriate assumption. The employed clustering algorithm requires a simpler and somewhat easier assumption on the misclassification rate, compared to other available clustering methods. 
\vskip 5pt
\noindent \textbf{Clustering Algorithm I}: (proposed in \cite{B2017})
\begin{enumerate}
\item
Obtain sums of the adjacency matrices before and after $b$ as $B_1 = \sum_{t=1}^{nb} A^{(t)}$ and $B_2 = \sum_{t=nb +1}^{n} A^{(t)}$ respectively. 
\item
Obtain $\hat{U}_{m \times K}$ and $\hat{V}_{m \times K}$ consisting of the leading $K$ eigenvectors of $B_1$ and $B_2$, respectively,  corresponding to its largest absolute eigenvalues.
\item
Use an $(1+\epsilon)$ approximate $K$-means clustering algorithm on the row
vectors of $\hat{U}$ and $\hat{V}$ to obtain $\tilde{z}_{b,n,m}$ and $\tilde{w}_{b,n,m}$ respectively. 
\end{enumerate}
Note that in Step $3$ above,   an $(1+\epsilon)$ approximate $K$-means clustering procedure is employed, instead of the  $K$-means. It is known that finding a global minimizer for the $K$-means clustering problem is NP-hard (for example see \cite{aloise2009np}). However,  efficient algorithms such as  $(1+\epsilon)$ approximate $K$-means clustering provide  an 
approximate solution, with the value of the objective function being minimized to within a constant fraction of the optimal value (see \cite{kumar2004simple} for more details).

\noindent
{\em Computational complexity of the procedure:}
Note that for each $b \in (c^{*},1-c^{*})$, the complexity of computing $\tilde{z}_{b,n,m}$ (or $\tilde{w}_{b,n,m}$) and $\tilde{\Lambda}_{\tilde{z}_{b,n,m},(b,n),m}$ (or $\tilde{\Delta}_{\tilde{w}_{b,n,m},(b,n),m}$) is $O(m^3)$ and  $O(m^2n)$, respectively.  Hence,  $\tilde{L}(b,\tilde{z}_{b,n,m},\tilde{w}_{b,n,m},\tilde{\Lambda}_{\tilde{z}_{b,n,m},(b,n),m},\tilde{\Delta}_{\tilde{w}_{b,n,m},(b,n),m})$ at $b$ has computational complexity  $O(m^3+m^2n)$. Some calculations show that only finitely many binary operations are needed to update  $\tilde{\Lambda}_{\tilde{z}_{b,n,m},(b,n),m}$ and $\tilde{\Delta}_{\tilde{w}_{b,n,m},(b,n),m}$ for the next available time point. However, computing $\tilde{z}_{b,n,m}$ and $\tilde{w}_{b,n,m}$ requires $O(m^3)$  operations for each time point. Therefore, the computational complexity for obtaining $\tilde{L}(b,\tilde{z}_{b,n,m},\tilde{w}_{b,n,m},\tilde{\Lambda}_{\tilde{z}_{b,n,m},(b,n),m},\tilde{\Delta}_{\tilde{w}_{b,n,m},(b,n),m})$ for $n$-many time points is $O(m^3n + m^2 n)=O(m^3n)$.

To establish consistency results for $\tilde{\tilde{\tau}}_{m,n}$, an additional assumption on the misclassification rate of $\tilde{z}_{b,n,m}$ and $\tilde{w}_{b,n,m}$
is needed, given next. We start with the following definition.
\begin{definition} \textbf{(Misclassification rate)} A node is considered as misclassified, if it is not allocated to its true community. The misclassification rate corresponds to the fraction of misclassified nodes. Let $\mathcal{M}_{(z,\tilde{z}_{b,n,m})}$ and $\mathcal{M}_{(w,\tilde{w}_{b,n,m})}$ be the misclassification rates of estimating $z$ and $w$ by $\tilde{z}_{b,n,m}$ and $\tilde{w}_{b,n,m}$, respectively. Then,
\begin{eqnarray}
\mathcal{M}_{(z,\tilde{z}_{b,n,m})} &=& 
\min_{\pi \in S_k} \sum_{i=1}^{m} \frac{I(\tilde{z}_{b,n,m}(i) \neq \pi(z(i)))}{s_{z,\pi(z(i))}},\nonumber \\
\mathcal{M}_{(w,\tilde{w}_{b,n,m})} &=& \min_{\pi \in S_k} \sum_{i=1}^{m} \frac{I(\tilde{w}_{b,n,m}(i) \neq \pi(w(i)))}{s_{w,\pi(w(i))}} 
\label{eqn: mdefine}
\end{eqnarray}
where $S_K$ denotes the set of all permutations of $\{1,2,\ldots,K\}$. 
\end{definition}

\noindent Define $$\mathcal{M}_{b,n,m} = \max(\mathcal{M}_{(z,\tilde{z}_{b,n,m})},\mathcal{M}_{(w,\tilde{w}_{b,n,m})}).$$
Consider the following assumption.
\vskip 5pt
\noindent \textbf{(NS)} $\Lambda$ and $\Delta$ are non-singular. 
\vskip 5pt
\noindent (NS) implies that there are exactly $K$ non-empty  communities in DSBM and hence we can use an $(1+\epsilon)$ approximate $K$-clustering algorithm. If (NS) does not hold, then we have $K^\prime$ $(<K)$ non-empty communities and an $(1+\epsilon)$ approximate $K^\prime$-clustering algorithm performs better.
\vskip 5pt
\noindent The following theorem provides the convergence rate of $\mathcal{M}_{b,n}$.  Its proof is given in Section \ref{subsec: mismiscluster}. Let $\nu_{m,n}$ denote the minimum between the smallest non-zero singular values of $\text{Ed}_{z}(\Lambda)$ and $\text{Ed}_{w}(\Delta)$.

\begin{theorem} \label{thm: mismiscluster}
Suppose (NS) holds.  Then, for all $b \in (c^*,1-c^*)$, we have
\begin{eqnarray}
\mathcal{M}_{b,n,m} = O_{\text{P}}\left( \frac{K}{n \nu_{m,n}^2}(\tau_{m,n} m + |\tau_{m,n}-b|\ ||\text{Ed}_{z}(\Lambda)-\text{Ed}_{w}(\Delta)||_F^2)\right). \nonumber
\end{eqnarray}
\end{theorem}

\begin{remark}
$\nu_{m,n}$ becomes closer to $0$ (that is we have less signal) as the regime becomes more sparse. Thus sparse regime has high misclassification rate for community detection. 
\end{remark}

\noindent 
\begin{remark}
To establish consistency of $\tilde{\tilde{\tau}}_{m,n}$,  we require that the misclassification rate $\mathcal{M}_{b,n,m}$   decays  faster than  $\rho_{m,n}^{-1}n^{-1}||\text{Ed}_{z}(\Lambda) -\text{Ed}_{w}(\Delta)||_F$;  see the proof of Theorem \ref{lem: b1}  and Remark \ref{rem: proof}  for technical details.  By the identifiability condition SNR-DSBM and Theorem \ref{thm: mismiscluster},  we have
\begin{eqnarray}
\mathcal{M}_{b,n,m}\rho_{m,n} n||\text{Ed}_{z}(\Lambda)-\text{Ed}_{w}(\Delta)||_F^{-1} \leq  C \frac{K\rho_{m,n}}{\nu_{m,n}^2} (\frac{m\sqrt{n}}{K}o(1) + m) \leq  C\rho_{m,n} ( \frac{m\sqrt{n}}{\nu_{m,n}^2}o(1) + \frac{Km}{\nu_{m,n}^2}) \nonumber 
\end{eqnarray}
holds with probability tending to $1$. Consistency of $\tilde{\tilde{\tau}}_{m,n}$ can be achieved under the SNR-DSBM condition and the following assumption (A1). 
\end{remark}

\medskip
\noindent \textbf{(A1)}  $\frac{Km}{ \nu_{m,n}^2} \rho_{m,n} \to 0$ and $\frac{m\sqrt{n}}{\nu_{m,n}^2} \rho_{m,n}= O(1)$

We note that (A1) is compatible with the clustering algorithm employed in our procedure. Other clustering algorithms may also be used which would lead to modifications of (A1), as discussed in Section \ref{sec:discuss}.

\subsection{Theoretical properties of $\tilde{\tilde{\tau}}_{m,n}$} \label{sec:theory-clustering}

Our first result establishes the convergence rate of the proposed estimate of the change-point. 

\begin{theorem} \label{lem: b1} \textbf{(Convergence rate of $\tilde{\tilde{\tau}}_{m,n}$)} \\
Suppose SNR-DSBM, Sparse-DSBM, (NS) and (A1)  hold.  Then, 
$$n\rho_{m,n}^{-1} ||\text{Ed}_z( \Lambda) - \text{Ed}_w(\Delta) ||_{F}^2 (\tilde{\tilde{\tau}}_{m,n} - \tau_{m,n}) = O_{\text{P}}(1).$$ 
\end{theorem}

\noindent
The proof of the theorem is given in Section \ref{subsec: b1}. 

The next result focuses on the misclassification rate for $\tilde{\tilde{z}} = \tilde{z}_{\tilde{\tilde{\tau}}_{m,n},n,m}$ and $\tilde{\tilde{w}} = \tilde{w}_{\tilde{\tilde{\tau}}_{m,n},n,m}$, respectively.

\begin{theorem} \label{thm: cluster} \textbf{(Rate of misclassification)}  \\
Suppose SNR-DSBM, Sparse-DSBM, (NS) and (A1)  hold.  Then,
\begin{eqnarray}
\mathcal{M}_{(z,\tilde{\tilde{z}})},\mathcal{M}_{(w,\tilde{\tilde{w}})} = O_{\text{P}}\left( \frac{Km}{n\nu_{m,n}^2}\right). \nonumber
\end{eqnarray}
\end{theorem}

\noindent
The proof of the Theorem is immediate from Theorems \ref{thm: mismiscluster} and \ref{lem: b1}.

Let $\tilde{\tilde{\Lambda}} = (( \tilde{\lambda}_{ab,\tilde{\tilde{z}},(\tilde{\tilde{\tau}}_{m,n}, n),m}))_{K \times K}$ and $\tilde{\tilde{\Delta}}  = (( \tilde{\delta}_{ab,\tilde{\tilde{w}},(\tilde{\tilde{\tau}}_{m,n}, n),m}))_{K \times K}$. The final result obtained is on the convergence rate of the community connection probability matrices
$\tilde{\tilde{\Lambda}}$ and $\tilde{\tilde{\Delta}}$, respectively. Let ${\mathcal{S}}_{m,n,\tilde{\tilde{z}}} = \min_{u} s_{u,\tilde{\tilde{z}}}$, $\mathcal{S}_{m,n,\tilde{\tilde{w}}} = \min_{u} s_{u,\tilde{\tilde{w}}}$ and $\tilde{\mathcal{S}}_{m,n} = \min(\mathcal{S}_{n,\tilde{\tilde{z}}},\mathcal{S}_{n,\tilde{\tilde{w}}})$.

\begin{theorem} \label{thm: c3} \textbf{(Convergence rate of edge probabilities $\tilde{\tilde{\Lambda}}$ and $\tilde{\tilde{\Delta}}$)} \\
Suppose SNR-DSBM, Sparse-DSBM, (NS) and (A1) hold. Further, for some positive sequence $\{\tilde{C}_{m,n}\}$, we have that $\tilde{\mathcal{S}}_{m,n} \geq \tilde{C}_{m,n}\ \forall m,n$ with probability $1$. Then,
\begin{eqnarray} 
\frac{1}{m^2} ||\text{Ed}_{\tilde{\tilde{z}}}(\tilde{\tilde{\Lambda}}) - \text{Ed}_{z}(\Lambda)||_F^2  &=& O_{\text{P}}\left( \left(\frac{Km}{n\nu_{m,n}^2}\right)^2 + \frac{I(n>1)\rho_{m,n}^2}{n^2 ||\text{Ed}_z(\Lambda)-\text{Ed}_w(\Delta) ||_F^4}+   \frac{\log m}{n \tilde{C}_{m,n}^2}\rho_{m,n} \right),\nonumber \\
\frac{1}{m^2} ||\text{Ed}_{\tilde{\tilde{w}}}(\tilde{\tilde{\Delta}}) - \text{Ed}_{w}(\Delta)||_F^2 
&=& O_{\text{P}}\left(\left(\frac{Km}{n\nu_{m,n}^2}\right)^2 +  \frac{I(n>1)\rho_{m,n}^2}{n^2 ||\text{Ed}_z(\Lambda)-\text{Ed}_w(\Delta) ||_F^4}+ \frac{\log m}{n \tilde{C}_{m,n}^2} \rho_{m,n}\right). \nonumber 
\end{eqnarray}
\end{theorem}
The proof of the Theorem is given in Section \ref{subsec: c3}.

\noindent
\begin{remark} \label{rem: paraestimation}
Note that the first term in the convergence rate of $\text{Ed}_{\tilde{\tilde{z}}}(\tilde{\tilde{\Lambda}})$, which is the square of the misclassification rate obtained in Theorem \ref{thm: cluster}, measures the closeness of $\text{Ed}_{\tilde{\tilde{z}}}(\tilde{\tilde{\Lambda}})$ to $\text{Ed}_{{z}}(\tilde{\tilde{\Lambda}})$. On the other hand, the second term is the convergence rate of $\text{Ed}_{{z}}(\tilde{\tilde{\Lambda}})$ for $\text{Ed}_{{z}}(\Lambda)$ and coincides with the convergence rate of the edge probability matrix estimator when the communities are known---see Theorem \ref{thm: b2} for details.  

As expected, the convergence rate of $\tilde{\tilde{\Lambda}}$ and $\tilde{\tilde{\Delta}}$, given in Theorem \ref{thm: c3}, is slower than the rate of $\tilde{\Lambda}$ and $\tilde{\Delta}$ when the communities are known. The reason is that the former estimates involve the misclassification rate of estimating $z$ and $w$ by $\tilde{\tilde{z}}$ and $\tilde{\tilde{w}}$, respectively. 
\end{remark}

\begin{remark} \label{rem: parameter}
\textbf{(Rate for $n=1$)}. For $n=1$, we go back to the setup of Remark \ref{rem: parameteroracle}. Suppose $z$ is unknown. We estimate $z$ and $\Lambda$ respectively by $\tilde{\tilde{z}}$ and $\tilde{\tilde{\Lambda}} = (( \tilde{\lambda}_{ab,\tilde{\tilde{z}},(1/n, n),m}))_{K \times K}$. Further, for some positive sequence $\{\tilde{C}_m\}$, suppose we have that $\tilde{\mathcal{S}}_m \geq \tilde{C}_m\ \forall m$ with probability $1$. Then 
\begin{eqnarray} 
\frac{1}{m^2} ||\text{Ed}_{\tilde{\tilde{z}}}(\tilde{\tilde{\Lambda}}) - \text{Ed}_{z}(\Lambda)||_F^2  &=& O_{\text{P}}\left( \left(\frac{Km}{\nu_{m,n}^2}\right)^2 +   \frac{\log m}{\tilde{C}_m^2}\rho_{m,n} \right). \nonumber 
\end{eqnarray}
The above rate of convergence is slower than the rate obtained in Remark \ref{rem: parameteroracle} where communities are known. This rate of convergence varies with different clustering methods employed for estimating $z$.  \cite{ZLZ2015} used a clustering algorithm for dense SBM (that is $\inf_{m,n} \rho_{m,n} >C>0$) so that the square of misclassification rate is $\sqrt{\frac{\log m}{m}}$  and $\tilde{C}_m^2 = \sqrt{m\log m}$. A detailed discussion on the impact of the clustering algorithm is provided in Section \ref{sec:discuss}. 
\end{remark}

\subsection{On the condition (A1)} \label{sec:A1}

As seen from the results in Section \ref{sec:theory-clustering}, condition (A1) plays a critical role. Next, we discuss examples where it holds---Examples \ref{example: misclassnew} and \ref{example: misclass1new}---and where it fails to do so---Example \ref{example: misclassnew2}.

\begin{example} \label{example: misclassnew}
Suppose we have $K$ balanced communities of size $m/K$.  Let $\Lambda = (p_1 - q_1)I_K + q_1J_K$ and $\Delta  = (p_2 - q_2)I_K + q_2J_K$, where $p_1,p_2,q_1,q_2$ may depend on both $m$ and $n$,  $I_K$ is the identity matrix of order $K$ and $J_K$ is the $K \times K$ matrix whose entries are all equal to $1$. Also assume $|p_1 - q_1|, |p_2-q_2| > \epsilon$ and $0<C<p_1,q_1,p_2,q_2<1-C<1$ for some $C, \epsilon >0$, which implies dense regime. Then, $\inf_{m,n}\rho_{m,n}>C$ for some $C>0$ and  the smallest non-zero singular value of $\text{Ed}_{z}(\Lambda)$ and $\text{Ed}_{w}(\Delta)$ are  $\frac{m}{K}|p_1-q_1|$ and $\frac{m}{K}|p_2-q_2|$, respectively. Therefore, $\nu_{m,n} = O(\frac{m}{K})$ and (A1)  reduces to 
\vskip 2pt
\begin{eqnarray} \label{eqn: A1reduce1}
\frac{K^3}{m} \to 0\ \ \text{and}\ \ \ \frac{K^2 \sqrt{n}}{m} = O(1).
\end{eqnarray} 
If $K$ is finite, then we need $n = O(m^2)$, which is a rather stringent requirement for most real applications.  

If $K = \sqrt{m}$, the condition does not hold as $m,n \to \infty$. If $K = Cm^{0.5 - \delta}$ for some $C, \delta>0$, then (\ref{eqn: A1reduce1}) holds if $m^{0.5-3\delta} \to 0$ and $n = O(m^{4\delta})$. In summary,  if $K = Cm^{0.5 - \delta}$,  $n = O(m^{4\delta})$ for some $C>0$ and $\delta> 1/6$, then (A1) holds. 

Next, define
\begin{eqnarray}
m_{\max,z}, m_{\max,w} &=& \text{largest community size of $z$ and $w$ respectively} \nonumber \\
m_{\min,z}, m_{\min,w} &=& \text{smallest community size of $z$ and $w$ respectively}. \nonumber 
\end{eqnarray}
The above conclusion also holds if we have  $\lim \frac{m_{\max,z}}{m_{\min,z}} = \lim \frac{m_{\max,w}}{m_{\min,w}}= 1$ instead of having balanced communities.

In the sparse regime, the assumption $|p_1 - q_1|, |p_2-q_2| > \epsilon$ and $0<C<p_1,q_1,p_2,q_2<1-C<1$ for some $C, \epsilon >0$ do not hold. In this case, (\ref{eqn: A1reduce1}) needs to hold after multiplying by $\sup_{m,n} \{(\min\{|p_1-q_1|,|p_2-q_2|\})^{-2} \max\{p_1,q_1,p_2,q_2\}\}$. 
\end{example}

\begin{example} \label{example: misclassnew2}
Consider the same model as in Example \ref{example: misclassnew} with $|p_1-q_1| = |p_2-q_2| = n^{-\delta}$ for some $\delta>0$. Suppose $K$ is finite. Then, (A1)  reduces to 
\vskip 2pt
\begin{eqnarray} \label{eqn: A1reduce}
\frac{n^{2\delta}}{m} \to 0\ \ \text{and}\ \ \ \frac{{n^{1/2+2\delta}}}{m} = O(1).
\end{eqnarray} 
In this case, (A1) does not hold if $m = C\sqrt{n}$ for some $C>0$.

We can also take $|p_1-q_1| = |p_2-q_2| = m^{-\delta}$ (instead of $n^{-\delta}$) for some $\delta>0$. Then also (A1) does not hold for $m=n$ and $\delta>1/4$.
\end{example}

\begin{example} \label{example: misclass1new}
Let $\Lambda = p_1 I_K$ and $\Delta = p_2 I_K$, where $p_1$ and $p_2$ may  depend on both $m$ and $n$. Assume that there is $\epsilon>0$ such that $p_1, p_2 > \epsilon$, that is dense regime. Then, the smallest non-zero singular values of $\text{Ed}_{z}(\Lambda)$ and $\text{Ed}_{w}(\Delta)$ are  $m_{\min,z}p_1$ and $m_{\min,w}p_2$ respectively. Let  $m_{\min} = \min \{m_{\min,z},m_{\min,w}\}$.  Therefore,  ${\nu}_{m,n} = O(m_{\min})$. Let $\tilde{\rho}_{m} = \frac{m_{\min}}{m}$.  Thus, (A1)  reduces to 
\begin{eqnarray}
\frac{K}{m \tilde{\rho}_m^2} \to 0\ \ \text{and}\ \ \ \frac{\sqrt{n}}{m\tilde{\rho}_m^2} = O(1). \label{eqn: 2} 
\end{eqnarray}
Suppose $K =Cm^{\lambda}$ and  $m_{\min} = Cm^{\delta}$ for some $\lambda, \delta \in [0,1]$. Then, $\tilde{\rho} _m = m^{\delta-1}$ and  (\ref{eqn: 2}) reduces to
\begin{eqnarray}
\frac{1}{m^{2\delta -\lambda-1}} \to 0 \ \ \text{and}\ \ \ \frac{\sqrt{n}}{m^{2\delta-1}} = O(1). \nonumber
\end{eqnarray}
Thus, (A1) holds if $K =Cm^{\lambda}$,  $m_{\min} = Cm^{\delta}$, $n = m^{4\delta-2}$ for some $\lambda, \delta \in [0,1]$ and $2\delta - \lambda-1 >0$. 

In the sparse regime, the assumption $p_1,p_2>\epsilon>0$ does not hold. In this case, (\ref{eqn: 2}) needs to hold after multiplying by $\sup_{m,n}\{(\min\{p_1,p_2\})^{-2}\max\{p_1,p_2\}\}$. 
\end{example}

\begin{remark}
Examples 2.1 and 2.3 provide sufficient conditions for Assumption (A1) in specific cases.  These conditions require a shorter time horizon (that means smaller value of $n$) in comparison with network sizes $m$ that are of interest in many real applications. Hence, these sufficient conditions are easily applicable to many real-world networks. Nevertheless, $\hat{\tau}_{m,n}$ (see Section \ref{sec: 2step}) proves useful in many practical settings for estimating the change-point
and avoids the above trade-off between $m$ and $n$.  
\end{remark}

\begin{remark} \label{rem: nc}
\textbf{A variant of the every point clustering algorithm with a weaker assumption on the misclassification rate:}
Note that computation of $\tilde{\tilde{\tau}}_{m,n}$ involves estimation of communities at \textit{every time point}. The necessity of clustering at every time-point leads us to consider
condition (A1). One may note though that since for theoretical considerations the change-point needs to be contained in the interval $(c^*,1-c^*)$, the following alternative approach can be employed: use Clustering Algorithm I for the first $[nc^{*}]$ and the last $[nc^{*}]$ time points for estimating $z$ and $w$, respectively. Denote the corresponding estimators by $z^{*}$ and $w^{*}$. Then, the corresponding change-point estimator $\tau^*_{m,n}$ can be obtained by
\begin{eqnarray}
\tau_{m,n}^* &=& \arg \min_{b \in (c^{*},1-c^{*})} \tilde{L} (b,{z}^{*},{w}^{*},\tilde{\Lambda}_{{z}^{*},(b,n),m},\tilde{\Delta}_{{w}^{*},(b,n),m}). \nonumber   
\end{eqnarray}
Since we are using order $n$-many time points in the clustering step and also for estimating the true change-point $\tau_{m,n} \in (c^*,1-c^*)$, the misclassification rates for $z^*$ and $w^*$ are similar to those of $\tilde{\tilde{z}}$ and $\tilde{\tilde{w}}$ obtained in Theorem \ref{thm: cluster}. As pointed out in the discussion preceding the statement of assumption (A1), clustering at every time point requires the misclassification rate $\mathcal{M}_{b,n,m}$ to decay faster than $n^{-1}\rho_{m,n}^{-1}||\text{Ed}_{z}(\Lambda) -\text{Ed}_{w}(\Delta)||_F$. However, when computing $\tau^*_{m,n}$, we use the same estimates $z^*$ and $w^*$ for \textit{all time points}. As will be seen later in Remark \ref{rem: proof}, if we use the same clustering solution (assignment of nodes to communities) through all the time points, we only require the misclassification rate to decay faster than $\rho_{m,n}^{-1}||\text{Ed}_{z}(\Lambda) -\text{Ed}_{w}(\Delta)||_F$ (instead of $n^{-1}\rho_{m,n}^{-1}||\text{Ed}_{z}(\Lambda) -\text{Ed}_{w}(\Delta)||_F$) for consistency of the change-point estimator. As a consequence, a weaker assumption on the misclassification rate
\vskip 2pt
\noindent \textbf{(A1*)} $\frac{m}{\sqrt{n}\nu_{m,n}^2}\rho_{m,n} = O(1)$ \vskip 2pt
\noindent together with the SNR-DSBM condition are needed to establish the consistency of $\tau^*_{m,n}$.  The upshot is that if node assignments $z^*$ and $w^*$ are employed, assumption (A1) becomes weaker.  

To further illustrate the latter point, note that in Example \ref{example: misclassnew},  (A1*) reduces to $K^2 = O(m\sqrt{n})$. Therefore, if $K = Cm^{\delta}$ and 
$n =Cm^{\lambda}$ for some $\delta \in (0,1)$ and $\lambda>0$, then (A1*) reduces to $1-2\delta+\lambda/2 \geq 0$.  For Example \ref{example: misclassnew2}, (A1*) boils down to  $n^{2\delta} = O(m)$.  Finally, in Example \ref{example: misclass1new},  (A1*) holds  if   $m_{\min} = Cm^{\delta}$, $n = m^{\lambda}$ for some $\lambda >0$,  $\delta \in [0,1]$ and $2\delta + \lambda/2 -1 \geq 0$.  
\end{remark}

\noindent Though the method in Remark \ref{rem: nc} needs a weaker assumption on the misclassification rate compared to (A1), this paper focuses on the ``every time-point clustering algorithm" due to the following two considerations. 

\begin{remark} \label{rem: limitation-nc} \textbf{Considerations for $\tau^*_{m,n}$:} Note that in practice the strategy in Remark \ref{rem: nc} requires that $c^*$ be known, which may not be the case in most applications. If $c^*$ is not known, a reasonable practical alternative is to use only the first and last time points to estimate $z$ and $w$, respectively. Further, $\frac{m\rho_{m,n}}{\nu_{m,n}^2} = O(1)$ is required for consistency of the change-point estimator. This is stronger than (A1*) but weaker than (A1).  One can argue that, in principle, the value of $c^*$ is needed to compute the change-point, since for establishing the theoretical results the search to identify it is restricted in the interval $(c^*,1-c^*)$. However, in practice, one always searches throughout the entire interval, and hence the practical alternative of using the first and last time points to estimate $z$ and $w$ is compatible with it. 

Finally, note that this alternative, that is known stretches of points that belong to only a single regime, is viable for estimating a single change-point, but no longer so when multiple change-points are involved. In the latter case, one would still assume that the first and last change-points are separated away from the boundary by some fixed amount, but no such restrictions on the locations of the intermediate change-points can be imposed, and hence a full search strategy (see for example the algorithm proposed in \cite{auger1989algorithms}) combined with clustering is unavoidable. This is the reason that our analysis focuses on the ``every time-point clustering algorithm"  since it provides insights on where challenges will arise for the case of multiple change-points,  appropriate treatment of which is nevertheless beyond the scope of this paper. 
\end{remark}

\begin{remark} \label{rem: otherclustering}
In this section, we used a specific clustering procedure proposed in \cite{B2017} to identify the communities and to locate the change-point. Nevertheless, other clustering algorithms proposed in the literature  [\cite{P2017, RCY2011}] could be employed. For any given clustering algorithms the following statements hold. 
\vskip 5pt
\noindent (a) The conclusions of Theorems  \ref{lem: b1},  \ref{thm: cluster} and \ref{thm: c3}  hold  once we replace (NS) and (A1)  by 
\vskip 2pt
\noindent \textbf{(A9)} $n^2\mathcal{M}_{b,n,m}^2 \rho_{m,n}^2 ||\text{Ed}_{z}(\Lambda) -\text{Ed}_{w}(\Delta)||_F^{-2}  \to 0, \ \forall \  b\in (c^*,1-c^*)$,
\vskip 2pt
\noindent where $\mathcal{M}_{b,n,m}$ is the maximum misclassification error that $\tilde{z}_{b,n,m}$ and $\tilde{w}_{b,n,m}$ in estimating $z$ and $w$, respectively, given in (\ref{eqn: mdefine}).
\vskip 5pt
\noindent (b) Suppose  (A9) and SNR-DSBM hold and in addition $\mathcal{M}_{\tilde{\tilde{\tau}}_{m,n},n,m}^2 = O_{\text{P}}(E_{m,n})$ for some sequence 
$E_{m,n} \to 0$. Moreover, assume that for some positive sequence $\{\tilde{C}_{m,n}\}$, $\hat{\mathcal{S}}_{m,n} \geq \tilde{C}_{m,n}\ \forall \ n$ with probability $1$ and $n\tilde{C}_{m,n}^{-2}\log m ||\text{Ed}_{z}(\Lambda) -\text{Ed}_{w}(\Delta)||_F^4 \geq I(n > 1)$. Then  
\begin{eqnarray} 
&& \frac{1}{m^2} ||\text{Ed}_{\tilde{\tilde{z}}}(\tilde{\tilde{\Lambda}}) - \text{Ed}_{z}(\Lambda)||_F^2,\ \  \frac{1}{m^2} ||\text{Ed}_{\tilde{\tilde{w}}}(\tilde{\tilde{\Delta}}) - \text{Ed}_{w}(\Delta)||_F^2  \nonumber \\
&& \hspace{3.5 cm}= O_{\text{P}}\left(  E_{m,n} + \frac{I(n>1)\rho_{m,n}^2}{n^2|\text{Ed}_{\tilde{\tilde{z}}}(\tilde{\tilde{\Lambda}}) - \text{Ed}_{z}(\Lambda)||_F^4} + \frac{\log m}{n \tilde{C}_{m,n}^2} \rho_{m,n}\right). \label{eqn: remremove}
\end{eqnarray}
\noindent The proofs of statements (a) and (b) follow immediately from those of   Theorems \ref{lem: b1}$-$\ref{thm: c3} and Remark \ref{rem: proof}.

The analogue of (A9) for existing clustering algorithms in the literature and their comparison with (A1) are discussed in Section \ref{sec:discuss} in more detail. 
\end{remark}

\section{A fast $2$-step procedure for change-point estimation in the DSBM} \label{sec: 2step}

The starting point of our exposition is the fact that the SBM is a special form of the Erd\H{o}s-R\'{e}nyi random graph model.  The latter is characterized by the following edge generating mechanism. Let $p_{ij}$ be the probability of having an edge between nodes $i$ and $j$ and let $P$ be the $m\times m$ corresponding connectivity probability matrix.  We denote this model by ER$(P)$. An adjacency matrix $A$ is said to be generated according to ER$(P)$, if $A_{ij} \sim \text{Bernoulli}(p_{ij})$ {\em independently} and we denote it by 
$A \sim \text{ER}(P)$. Clearly $A \sim \text{SBM}(z,\Lambda)$ implies $A \sim \text{ER}(\text{Ed}_{z}(\Lambda))$. 

The DSBM with single change-point in (\ref{eqn: dsbmmodel}) can be represented as a random graph model as follows: there is a sequence $\tau_{m,n} \in (0,1)$ such that for all $n \geq 1$,
\begin{eqnarray}
A_{t,n} \sim \begin{cases} \text{ER}(\text{Ed}_{z}(\Lambda)),\ \ \text{if $1 \leq t \leq \lfloor n\tau_{m,n} \rfloor$} \\
\text{ER}(\text{Ed}_{w}(\Delta)),\ \ \text{if $\lfloor n\tau_{m,n} \rfloor < t <n$} \label{eqn: dERmodel}
\end{cases}
\end{eqnarray}
and $\Lambda \neq \Delta$ and/or $z \neq w$. 

In general, without any structural assumptions,  a dynamic Erd\H{o}s-R\'{e}nyi model with a single change-point has $m(m+1)+1$ many unknown parameters,
the $0.5 m(m+1)$ pre- and post- change-point edge probabilities and $1$ change-point. An estimate of   $\tau_{m,n}$ can be obtained by optimizing the following least-squares criterion function.
\begin{eqnarray}
\hat{\tau}_{m,n} &=& {\arg \min}_{b \in (c^{*},1-c^{*})} L(b)\ \ \   \text{where} \nonumber \\
L(b) &=& \frac{1}{n}\sum_{i,j=1}^{m} \bigg[\sum_{t=1}^{nb} (A_{ij,(t,n)} - \hat{p}_{ij,(b,n),m})^2 + \sum_{t=nb+1}^{n} (A_{ij,(t,n)} - \hat{q}_{ij,(b,n),m})^2 \bigg], \nonumber \\
\hat{p}_{ij,(b,n),m} &=& \frac{1}{nb} \sum_{t=1}^{nb} A_{ij,(t,n)}\ \ \text{and}\ \ \hat{q}_{ij,(b,n),m} = \frac{1}{n(1-b)} \sum_{t=nb+1}^{n} A_{ij,(t,n)}. \label{eqn: estimatea1}
\end{eqnarray}

\noindent
Next, we present our 2-step algorithm. \\
\newline
\newline
\noindent \textbf{2-Step Algorithm}: 
\vskip 2pt
\noindent \textbf{Step 1:}  In this step, we ignore the community structures and assume $z(i) =  w(i) = i$ for all $1 \leq i \leq m$.  We compute  the least-squares criterion function $L(\cdot)$ given in (\ref{eqn: estimatea1}) and obtain the estimate ${\hat{\tau}}_{m,n} = \arg \min_{b \in (c^{*},1-c^{*})} L(b)$.
\vskip 2pt

\noindent \textbf{Step 2:} This step involves estimation of other parameters in DSBM. We estimate $z$ and $w$ by ${\hat{z}} = \tilde{z}_{\hat{\tau}_{m,n},n,m}$ and ${\hat{w}} = \tilde{w}_{\hat{\tau}_{m,n},n,m}$, respectively, and subsequently $\Lambda$ and $\Delta$ by $\hat{\hat{\Lambda}} = \tilde{\Lambda}_{{\hat{z}},(\hat{\tau}_{m,n},n),m}$ and $\hat{\hat{\Delta}} = \tilde{\Delta}_{{\hat{w}},(\hat{\tau}_{m,n},n),m}$, respectively.

\noindent
{\em Computational complexity of the 2-Step Algorithm} \\
It can easily be seen that Step $1$ requires $O(m^2n)$ operations, while Step 2 due to performing clustering requires $O(m^3)$ operations. Thus, the total computational complexity of the entire algorithm is $O(m^3 + m^2n) \sim O(m^2\max(m,n))$, which is significantly smaller than that of obtaining $\tilde{\tilde{\tau}}_{m,n}$  in (\ref{eqn: estimateintro}). 

\subsection{Theoretical Results for $\hat{\tau}_{m,n}$} \label{sec:theory-ER}

The following identifiability condition and the restriction on the sparse parameter are required.
\vskip 2pt
\noindent \textbf{SNR-ER:} $\frac{n}{m^2 \rho_{m,n}} ||\text{Ed}_{z}(\Lambda)-\text{Ed}_{w}(\Delta)||_F^2 \to \infty.$
\vskip 2pt
\noindent \textbf{Sparse-ER:} $\rho_{m,n} > Cm^{-2}$ for some $C>0$.
\vskip 2pt
SNR-ER requires that the signal per edge parameter and scaled by the sparsity parameter grows faster than $1/\sqrt{n}$. Clearly, SNR-ER is stronger than SNR-DSBM, as expected since the ER model involves $m^2$ parameters, as opposed to $K^2$ parameters for the DSBM. 

Similar to the discussed in Remark \ref{rem: sparseknowncomm1}, in the presence of SNR-ER, Sparse-ER ensures consistency of the change-point estimator $\hat{\tau}_{m,n}$ after scaling by the sparsity parameter.  Also, Sparse-ER is weaker than Sparse-DSBM.

The following Theorem provides asymptotic properties for the estimates of the DSBM parameters obtained from the 2-Step Algorithm.  Its proof is given in Section \ref{subsec: 2step}. 
\begin{theorem} \label{thm: 2step}
Suppose that SNR-ER and Sparse-ER hold. Then, the conclusion of Theorem \ref{lem: b1} holds for $\hat{\tau}_{m,n}$. Similarly, under SNR-ER, Sparse-ER and (NS),  the  conclusions of Theorems \ref{thm: cluster} and \ref{thm: c3} continue to hold for  $\hat{z}$, $\hat{w}$, $\hat{\hat{\Lambda}}$ and $\hat{\hat{\Delta}}$.
\end{theorem}

\begin{remark} \label{rem: sparseunknowncomm2}
As we can observe above, there is a trade-off between the signal-to-noise condition and the decay rate of the sparsity parameter. Though the convergence rate of $\tilde{\tau}_{m,n}$ and $\hat{\tau}_{m,n}$ are the same, the former one needs stronger Sparse-DSBM under weaker SNR-DSBM compared to weaker Sparse-ER under stronger SNR-ER required for the later estimator. 
\end{remark}

\begin{remark} \label{rem: sparsesnr}  
It is easy to see that  the average signal per edge $(m^2\rho_{m,n})^{-1}||\text{Ed}_{z}(\Lambda)-\text{Ed}_{w}(\Delta)||_F^2 \leq \rho_{m,n} \to 0$ for sparse graphs. On the other side, $\rho_{m,n}$ is bounded away from $0$ for dense graphs and  $(m^2\rho_{m,n})^{-1}||\text{Ed}_{z}(\Lambda)-\text{Ed}_{w}(\Delta)||_F^2 \geq C$ for some $C>0$ can be  often satisfied. Therefore dense regime generates more signal compared to the sparse regime and consequently the later one needs larger sample size (number of time points) to satisfy the signal-to-noise conditions SNR-DSBM and SNR-ER for detecting the change-point consistently and  has slower convergence rate of the change-point estimator. 
\end{remark}

\begin{remark} \label{rem: further}
One may wonder regarding dense settings (similar discussion is true for the sparse graphs as well) where SNR-DSBM holds, but neither (A1) nor SNR-ER do. Examples \ref{example: v10new1} and \ref{example: v10new2} introduce such settings in the context of changes in the connection probabilities and in the community structures, respectively. The methods discussed in Sections \ref{sec: DSBM} and \ref{sec: 2step} fail to detect the change-point under the above-presented settings. Therefore, alternative strategies need to be investigated. One possibility for the case of a  single change-point being present was discussed in Remark \ref{rem: nc} and more details are given in Example \ref{example: v10new3}. Another setting that does not require clustering is presented in Example \ref{example: v10new4} and builds on the model discussed in \cite{G2015rate}. However, the setting in Example \ref{example: v10new4} is very specific involving two parameters only.  Nevertheless, a generally applicable strategy is currently lacking for the regime where SNR-DSBM holds, but neither SNR-ER or (A1) does. This constitutes an interesting direction for future research.
\end{remark}

\section{Comparison of the ``Every time point clustering algorithm"  vs the 2-step algorithm} \label{sec: compare}

Our analysis up to this point has highlighted the following key findings. If the total signal is strong enough (that means if SNR-ER holds), then it is beneficial to use the $2$-step algorithm that provides consistent estimates of \textit{all} DSBM parameters at \textit{reduced computational cost}. On the other hand, if the signal is not adequately strong (that means if SNR-ER fails to hold, but SNR-DSBM holds) then the only option available is to use the computationally expensive ``every time point algorithm", provided that (A1) also holds. Our discussion in Section \ref{sec:A1} indicates that (A1) (and (A1*)) is not an innocuous condition and may fail to hold in real application settings.

For example, consider a DSBM with $m=60$ nodes, $K=2$ communities and $n=60$ time points. Suppose that there is a break at $n\tau_n = 30$, due to a change in community connection probabilities. Further, assume that the community connection probabilities before and after the change-point are given by $ \Lambda = \left(\begin{array}{cc}
0.6 & 0.3 \\ 
0.3 & 0.6
\end{array}  \right)$ and $\Delta = \Lambda + \frac{1}{n^{1/4}}J_2$ (or $= \Lambda + \frac{1}{m^{1/4}}J_2$), respectively. Finally, suppose that there is no change in community structures and $z(i) = w(i) = I(1 \leq i \leq m/2) + 2I(m/2 +1 \leq i \leq m)$.
In this case, one can check that $\inf_{m,n} \rho_{m,n} >C>0$,  $\frac{n}{m^2}||\text{Ed}_{z}(\Lambda)-\text{Ed}_{w}(\Delta)||_2^2 =7.75$, $\frac{Km}{\nu_{m,n}^2}= 1.48$, $\frac{m\sqrt{n}}{\nu_{m,n}^2} = 5.7$ and hence SNR-ER holds but (A1) fails. Figure \ref{fig: 2} plots the least-squares criterion function against time scaled by $1/n$, corresponding to the 2-step, known communities, and ``every time point" algorithms, respectively. The plots show that the trajectory of the least-squares criterion function is much smoother and the change-point is easily detectable when known community structures are assumed. It is also the case for the 2-step algorithm, albeit with more variability.  However, since (A1) fails to hold, the objective function depicted in Figure \ref{fig: 2} (bottom middle panel) clearly illustrates that the change-point is not detectable for the ``every time point" algorithm.

\begin{center}
\begin{figure}[htp] 
\includegraphics[height=70mm, width=74mm]{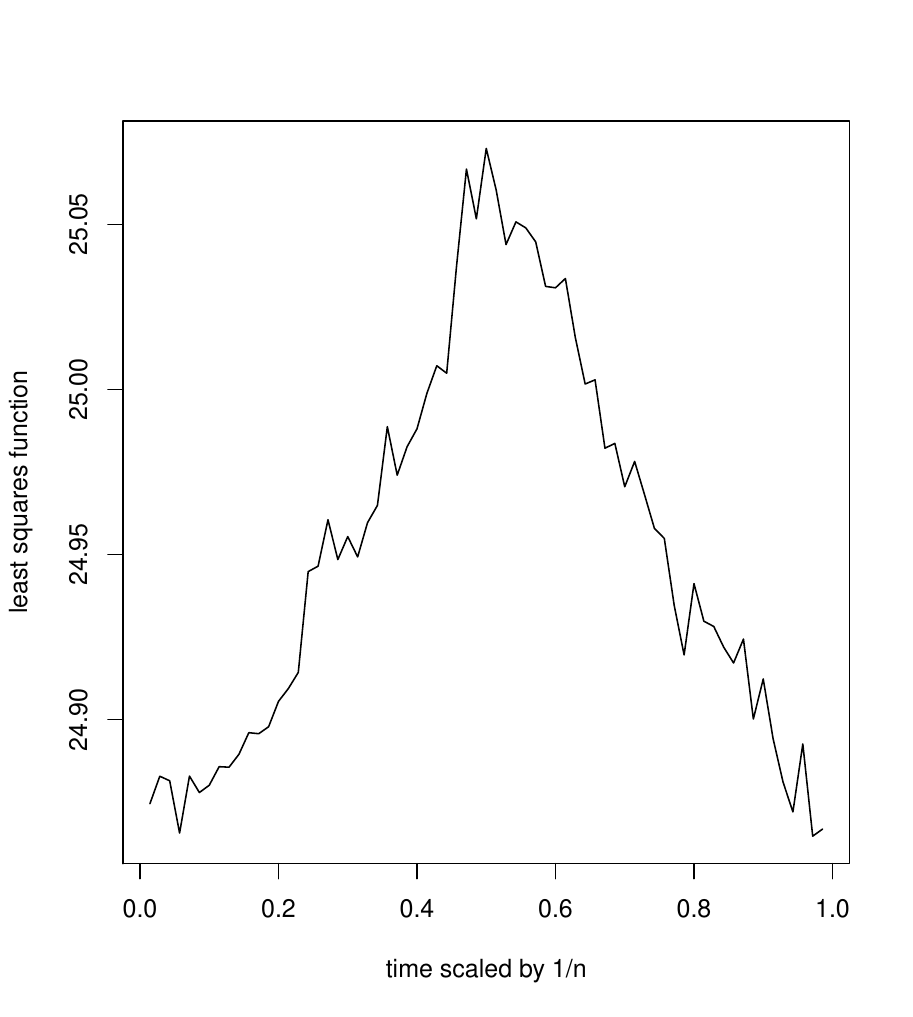} 
\includegraphics[height=70mm, width=74mm]{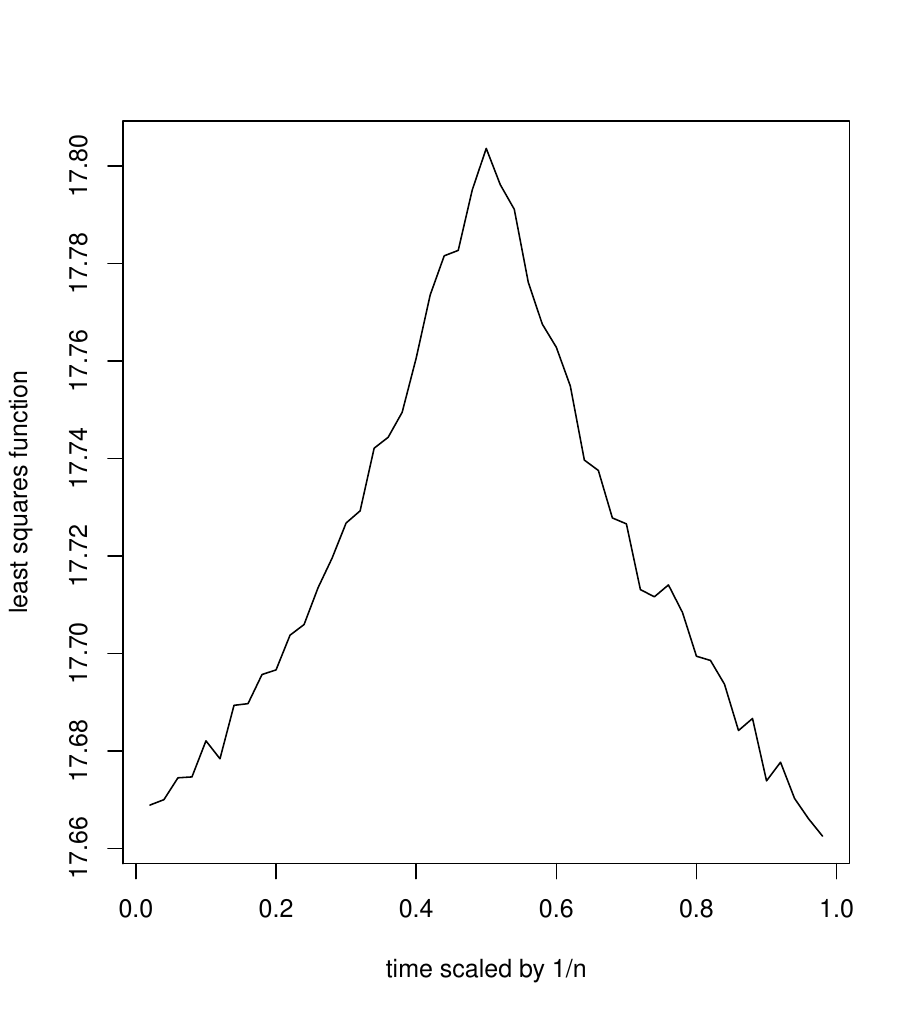}
\vskip 5pt
\center{\includegraphics[height=70mm, width=74mm]{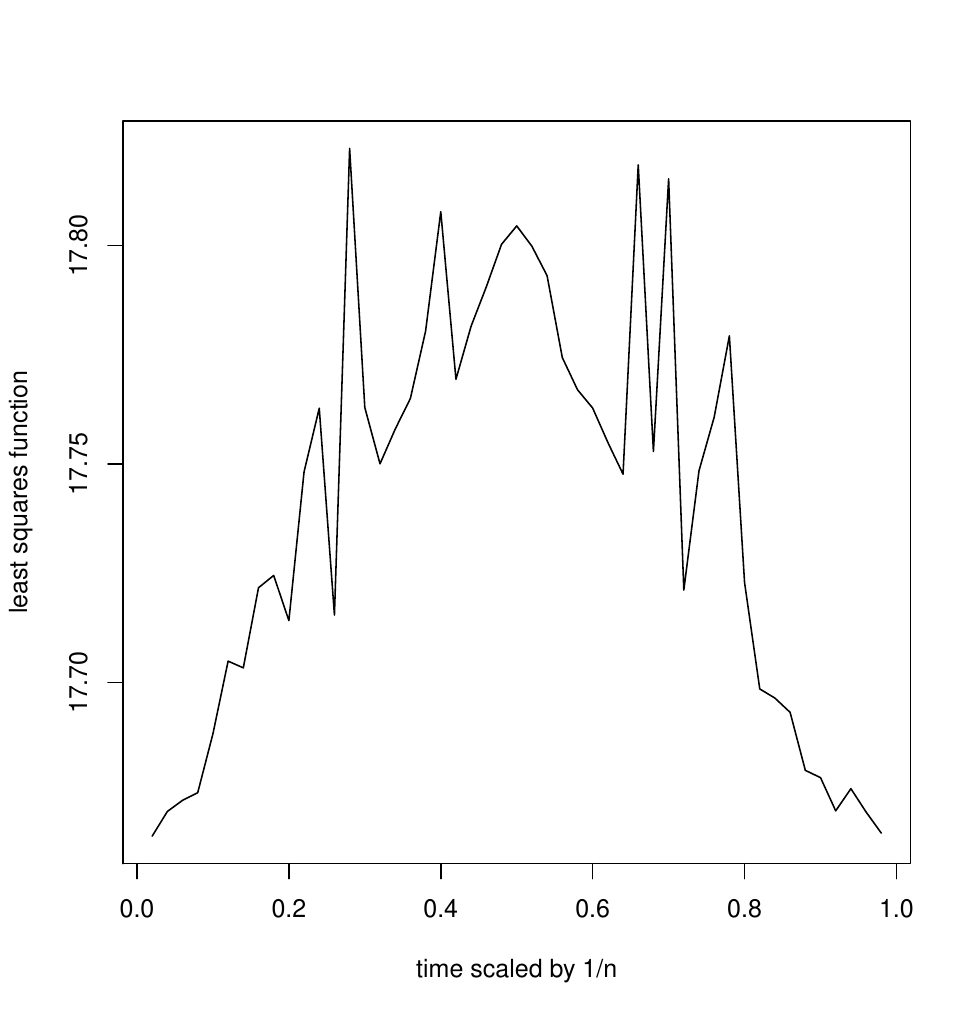}}
\caption{ \label{fig: 2} 
A plot of the least-squares criterion function against time scaled by $1/n$. Top left and right panels correspond to the 2-step and known communities algorithm, while the bottom middle depicts the ``every time point" algorithm, respectively.}
\end{figure} 
\end{center}

The next question to address is  {\em ``How stringent is SNR-ER"} under the DSBM model. As the following discussion shows, the reallocation of nodes to new communities generates strong enough signal, and therefore SNR-ER may be easier to satisfy in practice than one might suppose.

\bigskip
\noindent \textbf{Sufficient conditions for SNR-ER under the DSBM model}. \vskip 5pt
We examine a number of settings where SNR-ER holds under the DSBM network generating mechanisms and hence the 2-step algorithm can be employed.  Specifically, the following proposition provides sufficient conditions for SNR-ER to hold. Let $\mathbb{N}$  be  the set of all natural numbers. Define the classes of functions $\mathcal{F} := \{f\big| f: \mathbb{N} \times \mathbb{N} \to (0,1)\}$ and $\mathcal{G}_{m} := \{g=((m^2 f)\vee 1)\wedge 0.5m(m-1)  \big| f\in \mathcal{F}\}$ for all $m \geq 1$. 
For  any    $f \in \mathcal{F}$,  let 
$$\mathcal{A}(f) = \{(i,j):\ |\lambda_{z(i)z(j)} - \delta_{w(i)w(j)}| > f(m,n)\}.$$ Hence, $\mathcal{A}(f)$ corresponds to the set of all edges for which the connection probability changes at least by an $f(m,n)$ amount.

\begin{proposition}\label{example1}
 Suppose $|A(f)| \geq Cg(m,n)$ for some $f \in \mathcal{F}, g\in \mathcal{G}_m$ such that\\ $nm^{-2}g(m,n)(f(m,n))^2 \to \infty$. Then, SNR-ER holds. 
\end{proposition}

The above proposition follows from the fact that 
$$\frac{n}{m^2 \rho_{m,n}}||\text{Ed}_{z}(\Lambda) -\text{Ed}_{w}(\Delta)||_F^2  \geq nm^{-2}g(m,n)(f(m,n))^2 \to \infty.$$ 
This implies that at least $g(m,n)$-many edges need to change their connection probability by at least $f(m,n)$ amount for SNR-ER to be satisfied. This leads us to the following scenarios (A)-(D) that often arise in practice. 

The following example provides a choice for $f$ and $g$, respectively. 
\begin{example} \label{examplenewmn}
Let $A(\epsilon,\delta_1) = \{(i,j):\ |\lambda_{z(i)z(j)} - \delta_{w(i)w(j)}| > \epsilon m^{-\lambda_1/2}n^{-\delta_1 /2}\}$,  $|A(\epsilon,\delta_1)|\geq \max\{Cm^{2-\lambda_2}n^{-\delta_2},1\}$ and $m^{\lambda_1+\lambda_2} = o(n^{(1-\delta_1-\delta_2)})$ for some $C,\epsilon>0$ and $0\leq \delta_1+\delta_2<1, \lambda_1 \geq 0, 0 \leq \lambda_2 \leq 2$.   This implies that at least $m^{2-\lambda_2}n^{-\delta_2}$-many edges need to change their connection probability by at least $\epsilon m^{-\lambda_1/2}n^{-\delta_1 /2}$ amount.
Then Proposition \ref{example1}, by setting $f = \epsilon m^{-\lambda_1/2}n^{-\delta_1/2}$ and $g=\max\{Cm^{2-\lambda_2}n^{-\delta_2},1\}$, establishes that SNR-ER holds. As $\lambda_1, \lambda_2, \delta_1,\delta_2$ increase, the signal $||\text{Ed}_{z}(\Lambda) -\text{Ed}_{w}(\Delta)||_F^2$ due to the change decreases and therefore a large number of time points $n$ is required to accumulate adequate signal (that is to satisfy the SNR condition) for detecting the change-point. 
\end{example}

Next, we discuss settings motivated by real-world applications, wherein the SNR-ER condition holds for DSBM.

\vskip 2pt
\noindent (A) \textbf{Reallocation of nodes}: Suppose that the pre- and post-community connection probabilities are the same; that is $\Lambda = \Delta$. This also implies that the total number of communities before and after the change-point are equal.  Suppose that some of the nodes are reallocated to new communities after the change-point epoch. \\
\noindent
A motivating example for this scenario comes from voting patterns of legislative bodies as analyzed in \cite{bao2018core}. In this setting, one is interested in
identifying when voting patterns of legislators change significantly.  By considering the legislators as the nodes of the network, an edge between two of them
indicates voting similarly on a legislative measure (for examples bill, resolution), while the communities reflect their political affiliations, it can be seen that after an election the composition of the communities may be altered---reassignment of nodes.

In this situation, SNR-ER holds if the entries of $\Lambda$ (or $\Delta$) are adequately separated and enough nodes are reallocated.  Specifically, 
for some $\epsilon, C>0$ and $0 \leq \delta_1 + \delta_2 <1, \lambda_1 \geq 0, 0 \leq \lambda_2 \leq 2$, suppose we have $|\Lambda_{ij} - \Lambda_{i^\prime j^\prime}| > \epsilon m^{-\lambda_1/2}n^{-\delta_1 /2}\ \forall (i,j) \neq (i^\prime,j^\prime)$ and $\max\{Cm^{2-\lambda_2} n^{-\delta_2},1\}$-many nodes change their community after time $n\tau_n$. Then, by Proposition \ref{example1} and Example \ref{examplenewmn},  SNR-ER holds. 
\vskip 2pt

\noindent (B) \textbf{Change in connectivity}: Suppose that the community structures remain the same before and after the change-point (that implies $z = w$), but their community connection probabilities change (therefore $\Lambda \neq \Delta$). This scenario is motivated by the following examples: in transportation networks, when service is reduced or even halted between two service locations, in social media platforms (for example Facebook) when a new online game launches, or in collaboration networks when a large scale project is completed.

Then, SNR-ER holds if entries of $\Lambda$ are adequately separated from those of $\Delta$. Specifically, for some $\epsilon>0$ and $0 \leq \delta <1, \lambda\geq 0$, 
suppose we have $|\lambda_{ij} -\delta_{ij}| > \epsilon m^{-\lambda/2} n^{-\delta/2}\ \forall i,j = 1,2,\ldots,K$ and $m^{2+\lambda} = o(n^{1-\delta})$.  Then, by Proposition \ref{example1} and Examples \ref{examplenewmn},  SNR-ER holds.
\vskip 2pt

\noindent (C) \textbf{Merging Communities}: Sometimes, when two user communities cover the same subject matter and share similar contributors, they may wish to merge their communities to push their efforts forward in a desired direction. Suppose that the $1$st and the $K$th communities in $z$ merge into the $1$st community in $w$. In this situation, SNR-ER holds if the pre-connection probability between the $1$st and the $K$-th communities and the post-connection probability within the $1$st community are adequately separated and if the sizes of the $1$st and  the $K$-th communities are large before the change. Precisely, suppose $|\lambda_{1K} - \delta_{11}| > Cm^{-\lambda_1/2}n^{-\delta_1/2}$, $s_{1,z} s_{K,z} \geq Cm^{2-\lambda_2}n^{-\delta_2}$ and $m^{\lambda_1+\lambda_2} = o(n^{1-\delta_1-\delta_2})$ for some $C>0$ and $0 \leq \delta_1 + \delta_2 < 1, \lambda_1 \geq 0, 0 \leq \lambda_2\leq 2$. Then, by Proposition \ref{example1} and Example \ref{examplenewmn},  SNR-ER holds. 
\vskip 2pt

\noindent (D) \textbf{Splitting communities}: One community often splits into two communities when conflicts and disagreements arise among its members.
Suppose that the $1$st community in $z$ splits into the $1$st and $K$th communities in $w$.  In this case, SNR-ER holds if the pre-connection probability within 
the $1$st community and the post-connection probability between the $1$st and the $K$th communities are adequately separated and the size of the $1$st and the $K$th communities are large after the change. Suppose $|\lambda_{11} - \delta_{1K}| > Cm^{-\lambda_1/2}n^{-\delta_1/2}$,  $s_{1,w} s_{K,w} \geq Cm^{2-\lambda_2}n^{-\delta_2}$ and $m^{\lambda_1+\lambda_2} = o(n^{1-\delta_1-\delta_2})$ for some $C>0$ and $0 \leq \delta_1 + \delta_2 < 1, \lambda_1 \geq 0, 0 \leq \lambda_2 \leq 2$. Then, by Proposition  \ref{example1} and Example \ref{examplenewmn},  SNR-ER holds. 

\begin{remark} \label{rem: sparsecomp}
Examples (A)-(D) above and Proposition \ref{example1} hold for both dense and sparse networks. However, as discussed in Remark \ref{rem: sparsesnr},  for sparse networks, a large enough number of time points $n$ is required compared to the total number of nodes $m$. This is because that in a sparse network, there are relatively few edges
to contribute to the total signal in Proposition \ref{example1}. 
\end{remark}

Next, we discuss two examples for a dense network regime wherein the SNR-ER condition fails to hold, but SNR-DSBM does.

\noindent (E) If most edges change their connection probabilities by an amount of  $C_1/\sqrt{n}$ for some $C_1>0$, then SNR-ER does not hold, but SNR-DSBM does. Specifically, let $A(C_1) = \{(i,j):\ |\lambda_{z(i)z(j)} - \delta_{w(i)w(j)}| = C_1/\sqrt{n}\}$. Suppose  $|A(C_1)| = C_2m^2$ for some $C_1,C_2>0$,  $|\lambda_{z(i)z(j)} - \delta_{w(i)w(j)}| = 0\ \forall (i,j) \in A^c$ and $\displaystyle \min (\min_{u} s_{u,z},\min_{u} s_{u,w})  \to \infty$. Then $\frac{n}{m^2} ||\text{Ed}_{z}(\Lambda) - \text{Ed}_{w}(\Delta)||_F^2 = C_1^2 C_2 \centernot\longrightarrow \infty$ but $\frac{n}{K^2} ||\text{Ed}_{z}(\Lambda) - \text{Ed}_{w}(\Delta)||_F^2 = C_1^2C_2\frac{m^2}{K^2} \to \infty$. 
\vskip 2pt

\noindent (F) If the connection probabilities between the smallest community and the remaining ones change by $C/\sqrt{n}$ for some $C>0$, then for an
appropriate choice of $K$ and smallest community size, SNR-ER does not hold, but SNR-DSBM does. Specifically, suppose $z=w$, $K = C_1 m^{\delta_1/2}$, $\displaystyle \min_{u} s_{u,z} = s_{1,z} = C_2m^{\delta_2 /2}$, $|\lambda_{1j} - \delta_{1j}| = C_3/\sqrt{n}\ \forall j$ and $|\lambda_{ij} - \delta_{ij}| = 0\ \forall i \neq 1, j \neq 1$ for some $C_1,C_2,C_3 >0$ and $0 < \delta_1 + \delta_2 \leq 2$, $\delta_1<\delta_2$. Then   $\frac{n}{m^2} ||\text{Ed}_{z}(\Lambda) - \text{Ed}_{w}(\Delta)||_F^2 = C_3^2C_2m^{-(2-\delta_2)} \to 0$ but $\frac{n}{K^2} ||\text{Ed}_{z}(\Lambda) - \text{Ed}_{w}(\Delta)||_F^2 = C_3^2 C_2 C_1^{-2} m^{\delta_2 - \delta_1} \to \infty$. 

\medskip
Note that examples (E)-(F) only deal with the SNR-DBSM condition and do not address the equally important (A1) condition for the ``every time point clustering
algorithm" to work. The next example provides a dense setting where SNR-ER does not hold, but both SNR-DSBM and (A1) hold.

\vskip 2pt
\noindent (G) Consider the model and assumptions in Example \ref{example: misclass1new}. Suppose $p_2 = p_1 + \frac{1}{\sqrt{n}}$. Then, $mn^{-1} \leq ||\text{Ed}_{z}(\Lambda) - \text{Ed}_{w}(\Delta)||_F^2 \leq m^2 n^{-1}$. Hence, SNR-ER does not hold. Further, if $K^2 = o(m)$, then SNR-DSBM holds. 
Thus, SNR-DSBM and (A1) hold if $K =Cm^{\lambda}$,  $m_{\min} = Cm^{\delta}$ and $n = m^{4\delta-2}$ for some $\lambda \in [0,0.5), \delta \in [0,1]$ and $2\delta - \lambda-1 >0$.

The upshot of the above examples is that due to the structure of the DSBM, there are many instances arising in real settings where SNR-ER holds.
On the other hand, as an example (G) illustrates, some rather special settings are required for SNR-ER to fail, while both SNR-DSBM and (A1) hold.
Thus, it is relatively safe to assume that the 2-step algorithm is applicable across a wide range of network settings,  making it a very attractive option to practitioners.

\subsection{Numerical Illustration}  \label{subsec: simulation}

Next, we discuss the performance of the three change-point estimates $\tilde{\tau}_{m,n}$,  $\hat{\tau}_{m,n}$ and $\tilde{\tilde{\tau}}_{m,n}$ based on synthetic data generated according to the following mechanism, focusing on the impact of the parameters $m$, $n$ and small community connection probabilities on their performance.

\noindent \textbf{Effect of $m$ and $n$:} We simulate from the following DSBMs (1), (2), (3) for three choices of $(m,n, n\tau_n) = (60,60,30), (500,20,10), (500,100,50)$ and two choices of $\lambda = 0, 1/20$.  These results are presented  in Tables \ref{table: 1}-\ref{table: 6}. Although the following DSBMs satisfy the assumptions in Proposition \ref{example1}, SNR-ER may be small for dealing with finite samples.\footnote{Note that by Proposition \ref{example1} and Example \ref{examplenewmn},   for finite number of communities (K is finite), $\frac{n}{m^2}||\text{Ed}_{z}(\Lambda) - \text{Ed}_{w}(\Delta)||_F^2 = O(m^{-\lambda
_1-\lambda_2}n^{1-\delta_1-\delta_2})$  and $\frac{n}{K^2}||\text{Ed}_{z}(\Lambda) - \text{Ed}_{w}(\Delta)||_F^2 = O(m^{2-\lambda_1-\lambda_2} n^{1-\delta_1-\delta_2})$ and for balanced community with $\min\{|\lambda_{ij}-\lambda_{i^\prime j^\prime}|, |\delta_{ij}-\delta_{i^\prime j^\prime}|: (i,j) \neq (i^\prime, j^\prime), (i,j) \neq (j^\prime,i^\prime)\} \geq Cn^{-\delta}m^{-\lambda}$ ($C>0$, $\delta \geq 0$, $\lambda\geq 0$), we have $\nu_{m,n} = O(\frac{m^{1-\lambda}n^{-\delta}}{K})$,  we have $\frac{Km}{\nu_{m,n}^2} = O(n^{2\delta}/m^{1-2\lambda})$ and $\frac{m\sqrt{n}}{\nu_{m,n}^2} = O(\frac{n^{0.5+2\delta}}{m^{1-2\lambda}})$.  Thus, for small $n$  and large $m$, SNR-ER becomes small, but SNR-DSBM and (A1) hold. Moreover,  (A1) is not satisfied for large $\delta$ and $\lambda$  and  small $m$.}
\vskip 2pt
\noindent (1) \textbf{Reallocation of nodes}:  $K=2$, $z(i) = I(1 \leq i \leq m/2) + 2I(m/2 +1 \leq i \leq m)$, $w(2i-1) = 1,\ w(2i)=2\ \forall 1 \leq i \leq m/2$. 
$ \Lambda = \Delta = \left(\begin{array}{cc}
0.6 & 0.6 - \frac{1}{n^\delta m^\lambda} \\ 
0.6 - \frac{1}{n^\delta m^\lambda} & 0.6
\end{array}  \right)$ for $\delta = 1/20,1/10,1/4$. 
\vskip 2pt
\noindent (2) \textbf{Change in connectivity}:  $K=2$, $z(i) = w(i) = I(1 \leq i \leq m/2) + 2I(m/2 +1 \leq i \leq m)$, $ \Lambda = \left(\begin{array}{cc}
0.6 & 0.3 \\ 
0.3 & 0.6
\end{array}  \right)$, $\Delta = \Lambda + \frac{1}{n^{1/4}m^\lambda}J_2$.
\vskip 2pt
\noindent (3) \textbf{Merging communities}:  $K=3$, $z(i) = I(1 \leq i \leq 20) + 2I(21 \leq i \leq 40) + 3I(41 \leq i \leq 60)$, $w(i) = I(1 \leq i \leq 20,\ 41 \leq i \leq 60) + 2I(21 \leq i \leq 40)$, $\Lambda = \left(\begin{array}{ccc}
 0.6 & 0.3 & 0.6 - \frac{1}{n^{1/20}m^\lambda} \\ 
0.3 & 0.6 & 0.3 \\ 
0.6 - \frac{1}{n^{1/20}m^\lambda} & 0.3 & 0.6
\end{array}  \right)$, $\Delta = \left(\begin{array}{ccc}
 0.6 & 0.3 & 0 \\ 
0.3 & 0.6 & 0 \\ 
0 & 0 & 0
\end{array}  \right)$. \\
Splitting communities and merging communities are similar once we interchange $z$, $w$, and $\Lambda$, $\Delta$. 
\vskip 2pt
\begin{center}
\begin{table} [h!] 
\caption{\label{table: 1}
Illustrating the performance of  the change-point estimators with $m=60, n=60, n\tau_{m,n}=30$, $\lambda=0$ based on $100$ replicates and 
DSBMs  (1), (2) and (3). Figures in brackets are frequencies of the number of change-points the corresponding change-point is observed. 
Further, $F_n := ||\text{Ed}_{z}(\Lambda) -\text{Ed}_{w}(\Delta)||_F^2$.}
\vskip 2pt
\begin{tabular}{| m{1.9cm} | m{1.9cm} | m{1.9cm} | m{2.4cm} | m{2.4cm}| m{2.4cm}| }
\hline
 &\multicolumn{3}{|c|}{Reallocation of nodes,in (1)} & Change  \text{in connectivity}, in (2) &  Merging \text{communities}, in (3) \\ 
 \hline
  & $\delta = 1/20$ & $\delta = 1/10$ & $\delta = 1/4$ & &  \\
\hline
$F_n$ & $1195.246$ & $793.6742$ & $232.379$ & $2390.49$ &   $531.2205$ \\ 
\hline 
$\frac{n}{m^2}F_n$  & $19.92077$ & $13.2279$ & $3.873$ & $39.84$ &  $8.8537$ \\ 
\hline 
$\frac{n}{K^2}F_n$ & $17928.69$ & $11905.11$ & $3485.685$ & $35837.39$ &  $3541.47$   \\ 
\hline 
$\frac{Km}{\nu_{m,n}^2}$ & $0.198$ & $0.3$ & $1.2$ & $1.03$ &  $3.11$  \\ 
\hline 
$\frac{m\sqrt{n}}{\nu_{m,n}^2}$ & $0.7777$ & $1.1712$ & $4$ & $5.738$ &  $8.03$  \\ 
\hline 
$\tilde{\tau}_{m,n}$ & $30(90)$, $29(4)$, $28(2)$, $31(4)$ & 30(85), 29(6), 28(4), 31(5) & 30(88), 29(7), 28(5) &  30(85), 29(5), 28(6), 31(4) & 30(88), 29(3), 28(4), 31(5)\\
\hline 
$\hat{\tau}_{m,n}$ & $30$($88$),  $28(5)$, $31(3)$, $34(4)$  & $30$($83$), $29(3)$,  $28(7)$, $31(7)$ & $30(80)$, $29(9)$, $28(7)$, $31(4)$ & $30(83)$, $28(10)$, $31(7)$ & $30(88)$, $29(8)$,  $28(4)$  \\ 
\hline 
$\tilde{\tilde{\tau}}_{m,n}$  & $30(85)$, $28(5)$, $31(6)$, $32(4)$ & $30(80)$, $28(8)$, $31(6)$, $32(4)$, $33(2)$ & $30(34)$, $22(42)$, $25(10)$,  $33(14)$ & $30(40)$, $21(30)$, $28(18)$, $26(12)$ & $30(21)$, $19(10)$, $23 (48)$, $26(14)$, $38(7)$  \\ 
\hline 
\end{tabular} 
\end{table}
\end{center}

\begin{center}
\begin{table} [h!] 
\caption{\label{table: 2}
Illustrating the performance of  the change-point estimators with $m=500, n=20, n\tau_{m,n}=10$, $\lambda=0$ based on $100$ replicates and 
DSBMs  (1), (2) and (3). Figures in brackets are frequencies of the number of times the corresponding change-point is observed. Further, $F_n := ||\text{Ed}_{z}(\Lambda) -\text{Ed}_{w}(\Delta)||_F^2$.}
\vskip 2pt
\begin{tabular}{| m{1.9cm} | m{1.9cm} | m{1.9cm} | m{2.4cm} | m{2.4cm}| m{2.4cm}|  }
\hline
 &\multicolumn{3}{|c|}{Reallocation of nodes, in (1)} & Change  \text{in connectivity}, in (2) &  Merging \text{communities}, in (3) \\ 
 \hline
  & $\delta = 1/20$ & $\delta = 1/10$ & $\delta = 1/4$ & &  \\
\hline
$F_n$ & $92641.81$ & $68660.03$ & $27950.85$ & $185283.6$ &   $41174.14$ \\ 
\hline 
$\frac{n}{m^2}F_n$  & $7.411$ & $5.49$ & $2.24$ & $14.82$ &  $3.294$ \\ 
\hline 
$\frac{n}{K^2}F_n$ & $463209$ & $343300.2$ & $139754.2$ & $926418.1$ &  $91498.08$  \\ 
\hline 
$\frac{Km}{\nu_{m,n}^2}$ & $0.0216$ & $0.0291$ & $0.0716$ & $0.072$ &  $0.3732$  \\ 
\hline 
$\frac{m\sqrt{n}}{\nu_{m,n}^2}$ & $0.0483$ & $0.0651$ & $0.16$ & $0.16$ &  $0.576$  \\ 
\hline 
$\tilde{\tau}_{m,n}$ & 10(88), 9(6), 11(6) & 10 (85), 9(5), 8(3), 11(6), 12(1) & 10(90), 9(5), 8(1), 12(4) & 10(89), 9(5), 11(4), 12(2) & 10(88), 9(6), 8(4), 12(2) \\
\hline
$\hat{\tau}_{m,n}$ & $10$($90$),  $8(6)$, $11(4)$ & $10$($88$), $8(5)$, $11(7)$  & $10(39)$, $3(23)$, $7(30)$,  $13(8)$ & $10(85)$, $9(7)$, $8(8)$ & $10(83)$, $9(7)$, $8(4)$  $11(4)$, $12(2)$  \\ 
\hline 
$\tilde{\tilde{\tau}}_{m,n}$  & $10(85)$, $9(7)$, $11(5)$, $12(3)$ & $10(82)$,  $8(6)$, $11(5)$, $12(7)$ & $10(77)$, $8(11)$, $9(4)$,  $11(8)$ & $10(83)$, $9(9)$, $8(4)$, $11(4)$  & $10(80)$, $8(9)$, $9(7)$ $11(4)$  \\ 
\hline 
\end{tabular} 
\end{table}
\end{center}

\begin{center}
\begin{table} [h!] 
\caption{\label{table: 3}Illustrating the performance of  the change-point estimators with $m=500, n=100, n\tau_{m,n}=50$, $\lambda=0$, based on $100$ replicates and 
DSBMs  (1), (2) and (3). Figures in brackets are frequencies of the number of times the corresponding change-point is observed. Further, $F_n := ||\text{Ed}_{z}(\Lambda) -\text{Ed}_{w}(\Delta)||_F^2$.}
\vskip 2pt
\begin{tabular}{| m{1.9cm} | m{1.9cm} | m{1.9cm} | m{2.4cm} | m{2.4cm}| m{2.4cm}|  }
\hline
 &\multicolumn{3}{|c|}{Reallocation of nodes, in (1)} & Change  \text{in connectivity}, in (2) &  Merging \text{communities}, in (3) \\ 
 \hline
  & $\delta = 1/20$ & $\delta = 1/10$ & $\delta = 1/4$ & &  \\
\hline
$F_n$ & $78869.67$ & $49763.4$ & $12500$ & $157739.3$ &   $35053.19$ \\ 
\hline 
$\frac{n}{m^2}F_n$  & $31.548$ & $19.905$ & $5$ & $63.1$ &  $14.0213$ \\ 
\hline 
$\frac{n}{K^2}F_n$ & $1971742$ & $1244085$ & $312500$ & $3943483$ &  $1389479.8$  \\ 
\hline 
$\frac{Km}{\nu_{m,n}^2}$ & $0.02536$ & $0.04019$ & $0.16$ & $0.1778$ &  $0.3732$  \\ 
\hline 
$\frac{m\sqrt{n}}{\nu_{m,n}^2}$ & $0.1268$ & $0.201$ & $0.8$ & $0.889$ &  $1.244$  \\
\hline 
$\tilde{\tau}_{m,n}$ & 50(90), 49(4), 48(4), 51(2) & 50(89), 49(8), 51(3) & 50(93), 49(5), 48(1), 51(1) & 50(91), 49(6), 48(3) & 50(89), 49(7), 48(1), 51(3) \\ 
\hline 
$\hat{\tau}_{m,n}$ & $50$($92$),  $48(3)$, $51(5)$ & $50$($88$), $49(7)$, $48(3)$, $47(1)$, $51(1)$ & $50(82)$, $49(5)$, $47(3)$, $52(7)$, $53(3)$ & $50(84)$, $49(4)$, $47(2)$, $51(6)$, $52(4)$ & $50(87)$, $49(6)$, $51(3)$,  $52(4)$  \\ 
\hline 
$\tilde{\tilde{\tau}}_{m,n}$  & $50(87)$, $49(7)$, $48(5)$, $91(1)$ & $50(88)$, $48(6)$, $47(2)$, $51(4)$  & $50(82)$,  $49(7)$, $48(5)$, $51(4)$, $52(2)$ & $50(82)$, $49(7)$, $48(5)$, $52(6)$ & $50(87)$, $49(4)$, $47(3)$, $51(6)$  \\ 
\hline 
\end{tabular} 
\end{table}
\end{center}

\vskip 2pt
\begin{center}
\begin{table} [h!] 
\caption{\label{table: 4}
Illustrating the performance of  the change-point estimators with $m=60, n=60, n\tau_{m,n}=30$, $\lambda=1/20$ based on $100$ replicates and 
DSBMs  (1), (2) and (3). Figures in brackets are frequencies of the number of change-points the corresponding change-point is observed. 
Further, $F_n := ||\text{Ed}_{z}(\Lambda) -\text{Ed}_{w}(\Delta)||_F^2$.}
\vskip 2pt
\begin{tabular}{| m{1.9cm} | m{1.9cm} | m{1.9cm} | m{2.4cm} | m{2.4cm}| m{2.4cm}| }
\hline
 &\multicolumn{3}{|c|}{Reallocation of nodes,in (1)} & Change  \text{in connectivity}, in (2) &  Merging \text{communities}, in (3) \\ 
 \hline
  & $\delta = 1/20$ & $\delta = 1/10$ & $\delta = 1/4$ & &  \\
\hline
$F_n$ & $793.67$ & $527.82$ & $154.31$ & $187.35$ &   $231.02$ \\ 
\hline 
$\frac{n}{m^2}F_n$  & $13.23$ & $8.784$ & $2.572$ & $26.456$ &  $3.85$ \\ 
\hline 
$\frac{n}{K^2}F_n$ & $11905.11$ & $7905.3$ & $2314.58$ & $23810.26$ &  $1540.12$   \\ 
\hline 
$\frac{Km}{\nu_{m,n}^2}$ & $0.3$ & $0.455$ & $1.56$ & $6.024$ &  $1.021$  \\ 
\hline 
$\frac{m\sqrt{n}}{\nu_{m,n}^2}$ & $1.17$ & $1.764$ & $6.024$ & $5.738$ &  $2.64$  \\ 
\hline 
$\tilde{\tau}_{m,n}$ & 30(88), 29(7), 28(3), 32(2) & 30(91), 29(5), 28(1), 31(3) & 30(92), 28(2), 31(6) & 30(88), 29(8), 28(2), 31(2) & 30(89), 29(3), 31(7), 32(1) \\
\hline
$\hat{\tau}_{m,n}$ & $30$($84$),  $29(7)$, $32(5)$, $28(4)$  & $30$($87$), $29(6)$,  $31(7)$ & $30(78)$, $29(11)$, $28(8)$, $32(3)$ & $30(90)$, $28(2)$, $29(8)$ & $30(88)$, $31(7)$,  $28(5)$  \\ 
\hline 
$\tilde{\tilde{\tau}}_{m,n}$  & $30(83)$, $29(5)$, $28(7)$, $31(5)$ & $30(23)$, $26(23)$, $41(34)$, $22(14)$, $23(6)$ & $30(29)$, $22(45)$, $19(10)$,  $41(16)$ & $30(35)$, $22(21)$, $24(6)$, $19(18)$,  $42(20)$ & $30(21)$, $22(20)$, $17 (48)$,  $37(11)$  \\ 
\hline 
\end{tabular} 
\end{table}
\end{center}

\begin{center}
\begin{table} [h!]
\caption{\label{table: 5}
Illustrating the performance of  the change-point estimators with $m=500, n=20, n\tau_{m,n}=10$, $\lambda=1/20$ based on $100$ replicates and 
DSBMs  (1), (2) and (3). Figures in brackets are frequencies of the number of times the corresponding change-point is observed. Further, $F_n := ||\text{Ed}_{z}(\Lambda) -\text{Ed}_{w}(\Delta)||_F^2$.}
\vskip 2pt
\begin{tabular}{| m{1.9cm} | m{1.9cm} | m{1.9cm} | m{2.4cm} | m{2.4cm}| m{2.4cm}|  }
\hline
 &\multicolumn{3}{|c|}{Reallocation of nodes, in (1)} & Change  \text{in connectivity}, in (2) &  Merging \text{communities}, in (3) \\ 
 \hline
  & $\delta = 1/20$ & $\delta = 1/10$ & $\delta = 1/4$ & &  \\
\hline
$F_n$ & $49763.4$ & $36881.4$ & $15014.7$ & $99526.8$ &   $14501.2$ \\ 
\hline 
$\frac{n}{m^2}F_n$  & $3.98$ & $2.95$ & $1.2$ & $7.96$ &  $1.13$ \\ 
\hline 
$\frac{n}{K^2}F_n$ & $24881.7$ & $184406.8$ & $75070.2$ & $497634$ &  $32224.9$  \\ 
\hline 
$\frac{Km}{\nu_{m,n}^2}$ & $0.04$ & $0.054$ & $0.13$ & $0.13$ &  $0.13$  \\ 
\hline 
$\frac{m\sqrt{n}}{\nu_{m,n}^2}$ & $0.089$ & $0.121$ & $0.298$ & $0.298$ &  $0.1997$  \\ 
\hline
$\tilde{\tau}_{m,n}$ & 10(89), 9(8), 8(3) & 10(87), 9(5), 8(3), 11(5) & 10(88), 9(3), 8(2), 11(6), 12(1) & 10(90), 9(8), 8(2) & 10(89), 9(4), 8(1), 11(6) \\
\hline 
$\hat{\tau}_{m,n}$ & $10$($87$),  $8(4)$, $11(9)$ & $10$($85$), $8(8)$, $11(7)$  & $10(29)$, $3(23)$, $6(35)$,  $15(13)$ & $10(85)$, $9(9)$, $8(6)$ & $10(28)$, $7(33)$, $4(14)$,  $15(25)$  \\ 
\hline 
$\tilde{\tilde{\tau}}_{m,n}$  & $10(88)$, $9(8)$, $11(4)$ & $10(89)$,  $8(6)$, $11(5)$ & $10(79)$, $8(9)$, $9(5)$,  $11(7)$ & $10(83)$, $9(10)$, $8(7)$  & $10(85)$, $8(7)$, $9(5)$ $11(3)$  \\ 
\hline 
\end{tabular} 
\end{table}
\end{center}

\begin{center}
\begin{table} [h!] 
\caption{\label{table: 6} Illustrating the performance of  the change-point estimators with $m=500, n=100, n\tau_{m,n}=50$, $\lambda=1/20$, based on $100$ replicates and 
DSBMs  (1), (2) and (3). Figures in brackets are frequencies of the number of times the corresponding change-point is observed. Further, $F_n := ||\text{Ed}_{z}(\Lambda) -\text{Ed}_{w}(\Delta)||_F^2$.}
\vskip 2pt
\begin{tabular}{| m{1.9cm} | m{1.9cm} | m{1.9cm} | m{2.4cm} | m{2.4cm}| m{2.4cm}|  }
\hline
 &\multicolumn{3}{|c|}{Reallocation of nodes, in (1)} & Change  \text{in connectivity}, in (2) &  Merging \text{communities}, in (3) \\ 
 \hline
  & $\delta = 1/20$ & $\delta = 1/10$ & $\delta = 1/4$ & &  \\
\hline
$F_n$ & $42365.37$ & $26730.8$ & $67145$ & $84731.13$ &   $11943.6$ \\ 
\hline 
$\frac{n}{m^2}F_n$  & $16.94$ & $10.69$ & $2.686$ & $33.89$ &  $4.65$ \\ 
\hline 
$\frac{n}{K^2}F_n$ & $1059139$ & $668271.6$ & $167862.2$ & $2118278$ &  $132706.8$  \\ 
\hline 
$\frac{Km}{\nu_{m,n}^2}$ & $0.047$ & $0.0748$ & $0.298$ & $0.298$ &  $0.157$  \\ 
\hline 
$\frac{m\sqrt{n}}{\nu_{m,n}^2}$ & $0.236$ & $0.374$ & $0.49$ & $1.49$ &  $0.524$  \\ 
\hline
$\tilde{\tau}_{m,n}$ & 50(88), 49(6), 48(3), 51(3) & 50(90), 49(3), 51(5), 52(2) & 50(88), 49(4), 48(2), 51(4), 52(2) & 50(92), 49(3), 51(5) & 50(87), 49(7), 48(2), 51(4) \\
\hline 
$\hat{\tau}_{m,n}$ & $50$($88$),  $48(5)$, $49(7)$ & $50$($91$), $49(6)$, $48(3)$ & $50(78)$, $49(7)$, $48(5)$, $52(8)$, $53(2)$ & $50(90)$, $49(6)$, $48(2)$, $51(2)$ & $50(87)$, $49(8)$, $51(5)$ \\ 
\hline 
$\tilde{\tilde{\tau}}_{m,n}$  & $50(83)$, $49(8)$, $48(7)$, $53(2)$ & $50(88)$, $48(9)$, $47(2)$, $51(1)$  & $50(85)$,  $49(5)$, $48(5)$, $51(5)$ & $50(82)$, $49(5)$, $48(8)$, $51(2)$, $52(3)$ & $50(85)$, $49(8)$, $48(2)$, $51(5)$  \\ 
\hline 
\end{tabular} 
\end{table}
\end{center}

The following conclusions are in accordance with the results presented in Tables \ref{table: 1} through \ref{table: 6}. 
\vskip 2pt
\noindent (a) SNR-ER holds for  large $n$, small $\delta$, $\lambda$  and  large  signal $||\text{Ed}_{z}(\Lambda) - \text{Ed}_{w}(\Delta)||_F^2$. We  observe large SNR-ER and consequently good performance of $\hat{\tau}_{m,n}$, throughout  Tables \ref{table: 1}-\ref{table: 6} except Column $3$ in Table \ref{table: 2} and Column $3$ and $5$ of Table \ref{table: 5}, which involve a small $n$ and large $\delta$, $\lambda$, leading to poor performance of $\hat{\tau}_{m,n}$.
\vskip 2pt
\noindent (b) SNR-ER implies SNR-DSBM and thus a large SNR-DSBM is observed throughout Tables \ref{table: 1}-\ref{table: 6}.  
Moreover, if $\nu_{m,n} = O(\frac{m^{1-\lambda}n^{-\delta}}{K})$ for some $\delta >0$, then (A1) holds for small $\delta$, $\lambda$, $n$, small $K$ and large $m$. Thus, (A1) holds and $\tilde{\tilde{\tau}}_{m,n}$ exhibits good performance throughout Tables \ref{table: 1}-\ref{table: 6} except Columns $3-5$  in Table \ref{table: 1} and Columns $2-5$ in Table \ref{table: 4} where $\delta$ and $\lambda$ are large and $m$ small. 
\vskip 2pt
\noindent (c) Throughout Tables \ref{table: 1}-\ref{table: 6}, SNR-DSBM holds and $\tilde{\tau}_{m,n}$ exhibits good performance, as expected. The estimator $\tilde{\tau}_{m,n}$ performs equally well to $\hat{\tau}_{m,n}$ and $\tilde{\tilde{\tau}}_{m,n}$, whenever SNR-ER and (A1) are satisfied. In all other settings, $\tilde{\tau}_{m,n}$ clearly outperforms them.
For example, $\tilde{\tau}_{m,n}$ performs better than $\tilde{\tilde{\tau}}_{m,n}$ in Columns 3-5 of Table \ref{table: 1} and Columns 2-5 of Table \ref{table: 4} and better than $\hat{\tau}_{m,n}$ in Column 3 of Table \ref{table: 2} and Columns 3 and 5 of Table \ref{table: 5}.
\vskip 2pt

The above numerical results amply demonstrate the competitive nature of the computationally inexpensive 2-step algorithm under the settings  posited. However, note that the connection probabilities assumed are in general strong that leads to a large $F_n$ signal. Next, we illustrate the performance for the case of excessively small connection probabilities.

\noindent \textbf{Effect of excessively small connection probabilities:}  In this paper, we assume that the entries of $\Lambda$ and $\Delta$ are bounded away from $0$ and $1$,  to establish results on the asymptotic distribution of the change-point estimators (see Section \ref{sec: ADAP}). This assumption is, however,   not needed for establishing consistency and the convergence rate of the estimators.  Here  we consider DSBMs with small entries in $\Lambda$ and $\Delta$ and  illustrate their effect on the performance of the change-point estimators based on simulated results. For DSBMs (4) and (5),  we consider $(m,n,n\tau_{m,n}) =(60,60,30)$.

\vskip 2pt
\noindent (4) \textbf{Reallocation of nodes}:  Let $K=2$, $z(i) = I(1 \leq i \leq m/2) + 2I(m/2 +1 \leq i \leq m)$, $w(2i-1) = 1,\ w(2i)=2\ \forall 1 \leq i \leq m/2$. \\ Further,
$ \Lambda = \Delta =\left(\begin{array}{cc}
\frac{1}{n^{\lambda_1}m^{\lambda_2}} & \frac{1}{n^{\lambda_1}m^{\lambda_2}} - \frac{1}{n^{\delta_1}m^{\delta_2}} \\ 
\frac{1}{n^{\lambda_1}m^{\lambda_2}} - \frac{1}{n^{\delta_1}m^{\delta_2}} & \frac{1}{n^{\lambda_1}m^{\lambda_2}}
\end{array}  \right)$\\ for $(\delta_1, \delta_2,\lambda_1,\lambda_2) = (3/8,3/8,1/4,1/4), (5/8,1/4, 1/4,3/8)$. 
\vskip 2pt
\noindent (5) \textbf{Change in connectivity}:  Let $K=2$, $z(i) = w(i) = I(1 \leq i \leq m/2) + 2I(m/2 +1 \leq i \leq m)$, $ \Lambda = \left(\begin{array}{cc}
\frac{2}{n^{\lambda_1}m^{\lambda_2}} & \frac{1}{n^{\lambda_1}m^{\lambda_2}}  \\ 
\frac{1}{n^{\lambda_1}m^{\lambda_2}}  & \frac{2}{n^{\lambda_1}m^{\lambda_2}} 
\end{array}  \right)$, $\Delta = \Lambda + \frac{1}{n^{1/8}m^{1/8}}J_2$, $(\lambda_1,\lambda_2) = (1/4,1/4),(1/4,3/8)$.
\vskip 2pt
The results are presented in Table \ref{table: 7}. For models (4) and (5),  SNR-ER is proportional to $n^{-2\delta_1}m^{-2\delta_2}$ and $n^{-1/4}m^{-1/4}$ respectively. 
The choices of $\delta_1$ and $\delta_2$ taken in (4) suffice to make the connection probabilities in $\Lambda$ and $\Delta$ small enough, so that the 
 resulting SNR-ER is small. Consequently,  $\hat{\tau}_{m,n}$ does not perform well in Columns $1$ and $2$ of Table \ref{table: 7}.  On the other hand, $\delta_1 = 1/8$ and $\delta_2=1/8$ in (5) are adequate to induce a large SNR-ER, as reflected in the improved performance of $\hat{\tau}_{m,n}$ for $\tau_{m,n}$ in Columns $3$ and $4$ of Table \ref{table: 7}. On the other hand,  while $m/K$ in Table \ref{table: 7} is large enough to satisfy SNR-DSBM,  $\nu_{m,n}$ is proportional to $n^{-\lambda_1}m^{-\lambda_2}$ and by the choices of $\lambda_1$, $\lambda_2$ in models (4) and (5), quite small, as a consequence of which (A1) does not hold for the settings depicted in Table \ref{table: 7}, and  the performance of $\tilde{\tilde{\tau}}_n$ suffers. The estimator $\tilde{\tau}_{m,n}$ performs very well throughout Table \ref{table: 7}, since SNR-DSBM holds.

\begin{center}
\begin{table} [h!] 
\caption{\label{table: 7}
Illustrating the performance of  the change-point estimators with $m=60, n=60, n\tau_{m,n}=30$, based on $100$ replicates and 
DSBMs  (4) and (5). Figures in brackets are frequencies of the number of times the corresponding change-point is observed. Further, $F_n := ||\text{Ed}_{z}(\Lambda) -\text{Ed}_{w}(\Delta)||_F^2$.}
\vskip 2pt
\begin{tabular}{| m{1cm} | m{3cm} | m{3cm} | m{3cm} | m{3cm} |  }
\hline
  &\multicolumn{2}{|c|}{Reallocation of nodes, in (4)} &  \multicolumn{2}{|c|}{Change in connectivity, in (5)}   \\ 
 \hline
  & $(\delta_1,\delta_2,\lambda_1,\lambda_2) = (3/8,3/8,1/4,1/4)$ & $(\delta_1,\delta_2,\lambda_1,\lambda_2) = (5/8,1/4,1/4,3/8)$ & $(\lambda_1,\lambda_2) = (1/4,1/4)$ & $(\lambda_1,\lambda_2)=(1/4,3/8)$  \\
\hline
$F_n$ & $3.873$ & $1.3915$ & $464.758$ & $464.758$  \\ 
\hline 
$\frac{n}{m^2}F_n$  & $0.0645$ & $0.0232$ & $7.746$ & $7.746$  \\ 
\hline 
$\frac{n}{K^2}F_n$ & $58.095$ & $20.874$ & $6971.37$ & $6971.37$   \\ 
\hline 
$\frac{Km}{\nu_{m,n}^2}$ & $61.968$ & $172.466$ & $8$ & $22.265$   \\ 
\hline 
$\frac{m\sqrt{n}}{\nu_{m,n}^2}$ & $240$ & $667.96$ & $30.984$ & $86.232$   \\ 
\hline 
$\tilde{\tau}_{m,n}$ & 30(85), 29(6), 28(3), 31(5), 32(1) & 30(88), 29(4), 31(8) & 30(91), 29(8), 28(1) & 30(89), 29(5), 28(2), 31(4) \\
\hline
$\hat{\tau}_{m,n}$ & $30$($15$), $27(18)$,  $24(38)$, $22(19)$, $41(10)$ & $30$($21$), $28(4)$, $18(31)$, $39(14)$, $47(19)$, $49(11)$  & $30(88)$, $29(9)$,  $32(3)$ & $30(88)$, $28(6)$, $31(6)$   \\ 
\hline 
$\tilde{\tilde{\tau}}_{m,n}$  & $30(7)$, $28(23)$, $22(13)$, $18(18)$, $38(20)$,$43(19)$ & $26(15)$, $22(23)$,  $32(27)$, $39(25)$, $43(10)$  & $30(14)$, $25(8)$, $23(28)$, $38(34)$, $43(16)$  & $30(15)$, $26(10)$, $21(25)$, $21(28)$, $39(14)$,  $44(8)$   \\ 
\hline 
\end{tabular} 
\end{table}
\end{center}

\noindent \textbf{Simulation on setting (G):}  Consider the setup in setting (G) previously presented and let $n=20$, $n\tau_{m,n} = 10$, $m = 20$, $K=2$, $z(i) = w(i) = I(1 \leq i \leq 9) + 2I(10 \leq i \leq 20)$, $p_1^2=0.8$, $p_2 = p_1 + 1/\sqrt{n}$, $\Lambda = p_1I_2$, $\Delta=p_2I_2$. Simulation results are given in Table \ref{table: 8}. In this case, both SNR-DSBM and (A1) hold. Hence, both $\tilde{\tau}_{m,n}$ and $\tilde{\tilde{\tau}}_{m,n}$ perform well as expected. However, due to the failure of SNR-ER to hold, the performance of $\hat{\tau}_{m,n}$ suffers.

\begin{center}
\begin{table} [h!] 
\caption{\label{table: 8} Illustrating the performance of change-point estimators with $m=20, n=20, n\tau_{m,n}=10$ based on $100$ replicates for  DSBMs in setting (G). Figures in brackets are frequencies of the number of times the corresponding change-point is observed. Further, $F_n := ||\text{Ed}_{z}(\Lambda) -\text{Ed}_{w}(\Delta)||_F^2$.}
\vskip 2pt
\begin{tabular}{| m{1cm} | m{1cm} | m{1cm} | m{1cm} | m{1cm} | m{1cm} | m{2.0cm}| m{2.0cm} | m{2.1 cm}|}
\hline 
  & $F_n$ & $\frac{n}{m^2}F_n$ & $\frac{n}{K^2}F_n$ & $\frac{Km}{\nu_{m,n}^2}$ & $\frac{m\sqrt{n}}{\nu_{m,n}^2}$ & $\tilde{\tau}_{m,n}$ & $\hat{\tau}_{m,n}$ & $\tilde{\tilde{\tau}}_{m,n}$ \\ 
\hline 
\text{On (G)} & $10.1$ & $0.51$ & $50.5$ & $0.62$ & $1.38$  & \text{10(90), 9(5)}, \text{8(2), 11(4)} & \text{$4(42)$, $5(33)$}, \text{$8(12)$,} \text{$10(5)$,  $14(8)$} & \text{$10(78)$, $9(15)$,} \text{$8(5)$,} \text{$12(2)$} 
\\ 
\hline 
\end{tabular} 
\end{table}
\end{center}

\section{Asymptotic distribution of change-point estimators and adaptive inference}  \label{sec: ADAP}

Up to this point, the analysis focused on establishing consistency results for the derived change-point estimators and the corresponding convergence rates. Nevertheless, it is also of interest to provide confidence intervals, primarily for the change-point estimates. This issue is addressed next for $\tilde{\tau}_{m,n}$,  $\tilde{\tilde{\tau}}_{m,n}$, $\tau_{m,n}^{*}$ and $\hat{\tau}_{m,n}$, and as will be shortly seen the distributions are different depending on the behavior of the norm difference of the parameters before and after the change-point. Since this norm difference is not usually known a priori, we solve this problem through a data-based adaptive procedure to determine the quantiles of the asymptotic distribution, irrespective of the specific regime pertaining to the data at hand.

\subsection{Form of asymptotic distribution} \label{subsec: asympdist}
For ease of presentation, we focus on $\hat{\tau}_{m,n}$, but analogous results hold for $\tau_{m,n}^{*}$, $\tilde{\tau}_{m,n}$  and $\tilde{\tilde{\tau}}_{m,n}$.  As previously mentioned, there are three different regimes for its asymptotic distribution depending on:---(I) $||\text{Ed}_z(\Lambda) - \text{Ed}_w(\Delta)||_F^2 \to \infty$, (II) $||\text{Ed}_z(\Lambda) - \text{Ed}_w(\Delta)||_F^2 \to 0$ and (III) $||\text{Ed}_z(\Lambda) - \text{Ed}_w(\Delta)||_F \to c>0$. 

We need additional regularity assumptions (A2)-(A7) for the other regimes. Assumption (A2) stated below ensures that the connection probabilities are bounded away from $0$ and $1$, which  gives rise to a dense graph and ensures the positive asymptotic variance of the change-point estimators. 
 \vskip 2pt
\noindent \textbf{(A2)} For some $c>0$,  $0< c < \inf_{u,v} \lambda_{uv}, \inf_{u,v} \delta_{uv} \leq \sup_{u,v} \lambda_{uv}, \sup_{u,v} \delta_{uv} < 1-c<1$. 
\vskip 2pt
\noindent The  precise statements of (A3)-(A7) are given in Section \ref{subsec: assumption}, but a brief discussion of their roles is presented below.

Assumption (A3) is required in Regime II and  guarantees the existence  of the asymptotic variance of the change-point estimator. In Theorem \ref{lem: b2lse}(b), this variance is denoted by $\gamma^2$. 

In Regime III, we consider the following set of edges
\begin{eqnarray}
\mathcal{K}_n = \{(i,j): 1 \leq i,j \leq m,\ \  
|\lambda_{z(i)z(j)} - \delta_{w(i)w(j)}| \to 0\}
\end{eqnarray} 
and treat edges in $\mathcal{K}_n$ and $\mathcal{K}_0 = \mathcal{K}_n^c$ separately. Note that in Regime II, $\mathcal{K}_n = \{(i,j):\ 1\leq i,j \leq m\}$  is the set of all edges.  Hence, we can treat $\mathcal{K}_n$ in a similar way as in Regime II.  The role of (A4) in Regime III is analogous to that of (A3) in Regime II.  In the limit, $\mathcal{K}_n$ contributes a Gaussian process with a triangular drift term. (A4) ensures the existence of the asymptotic variance $\tilde{\gamma}^2$  of the limiting Gaussian process as well as the drift $c_1^2$.   (A5) is a technical assumption and is required for establishing asymptotic normality on $\mathcal{K}_n$.   Moreover,  $\mathcal{K}_0$ is a finite set. (A6) guarantees that $\mathcal{K}_0$ does not vary with $n$.  (A7) guarantees that  $\tau_{m,n} \to \tau^{*}$  for some $\tau^{*} \in (c^{*},1-c^{*})$,  $\lambda_{z(i)z(j)} \to a_{ij,1}^{*}$ and $\delta_{w(i)w(j)} \to a_{ij,2}^{*}$ for all $(i,j) \in \mathcal{K}_0$.  Consider the collection of independent Bernoulli random variables $\{A_{ij,l}^{*}: (i,j) \in \mathcal{K}_0, l=1,2\}$ with $E(A_{ij,l}^{*}) = a_{ij,l}^{*}$.  Then, (A7) implies $A_{ij,(\lfloor nf \rfloor, n)} \stackrel{\mathcal{D}}{\to} A_{ij,1}^{*}I(f<\tau^{*}) + A_{ij,2}^{*}I(f > \tau^{*})\ \forall (i,j) \in \mathcal{K}_0$.

The following Theorem summarizes the asymptotic distribution results.

\begin{theorem} \label{lem: b2lse}
Suppose  SNR-ER holds for $\hat{\tau}_{m,n}$,  SNR-DSBM, (NS) and (A1) hold for $\tilde{\tilde{\tau}}_{m,n}$, SNR-DSBM, (NS) and (A1*) hold for ${\tau}_{m,n}^{*}$ and SNR-DSBM holds for $\tilde{\tau}_{m,n}$. Then, the following statements are true. Let $\hat{\eta}_{m,n}$ denote generically any of the following change-point estimators: $\hat{\tau}_{m,n}$,  $\tilde{\tau}_{m,n}$, $\tilde{\tilde{\tau}}_{m,n}$ and $\tau_{m,n}^{*}$.

\noindent
{\bf (a)} If $||\text{Ed}_{z}(\Lambda)-\text{Ed}_{w}(\Delta)||_F^2 \to \infty$, then $\lim_{n \to \infty} P(\hat{\eta}_{m,n}=\tau_{m,n}) =1.$

\noindent
{\bf (b)} If (A2) and (A3) hold and  $||\text{Ed}_{z}(\Lambda)-\text{Ed}_{w}(\Delta)||_F^2 \to 0$, then
\begin{eqnarray}
n||\text{Ed}_{z}(\Lambda)-\text{Ed}_{w}(\Delta)||_F^2  (\hat{\eta}_{m,n}-\tau_{m,n}) \stackrel{\mathcal{D}}{\to}  \gamma^2 \arg \max_{h \in \mathbb{R}} (-0.5|h| + B_h),\ \ 
\end{eqnarray}
where $B_h$ denotes the standard Brownian motion.

\noindent
{\bf (c)} Suppose (A2), (A4)-(A7)  hold and $||\text{Ed}_{z}(\Lambda) -\text{Ed}_{w}(\Delta)||_F \to c >0$, then
\begin{eqnarray}
n (\hat{{\eta}}_{m,n} -\tau_{m,n}) &\stackrel{\mathcal{D}}{\to}& \arg \max_{h \in \mathbb{Z}} (D(h) + C(h) + A(h)) \nonumber 
\end{eqnarray}
where  for each ${h} \in \mathbb{Z}$,
\begin{eqnarray}
D (h+1)-D(h) &=& 0.5 {\rm{Sign}}(-h) c_1^2, \label{eqn: msethm2d}\\
C(h+1) - C(h) &=& \tilde{\gamma} W_{{h}},\ \ W_{{h}} \stackrel{\text{i.i.d.}}{\sim} \mathcal{N}(0,1), \ \ \ \ \ \ \  \ \label{eqn: msethm2c}\\
A(h+1) - A(h) &=& \sum_{k \in \mathcal{K}_0} \bigg[(Z_{ij, {h}} -a_{ij,1}^{*})^2 -(Z_{ij,{h}} -a_{ij,2}^{*})^2 \bigg],\  
  \label{eqn: msethm2a}
\end{eqnarray}
 $\{Z_{ij, {h}}\}$ are independently distributed with $Z_{ij, {h}} \stackrel{d}{=} A_{ij,1}^{*}I({h} < 0) + A_{ij,2}^{*}I({h} \geq  0)$ for all $(i,j) \in \mathcal{K}_0$. 
\end{theorem}

\begin{remark} \label{rem: sparsesmall}
As we have already noted, consistency of the change-point estimators holds for both dense and sparse graphs. The same conclusion holds for the asymptotic distribution under Regime I.  However,  (A2) is a crucial assumption for establishing the asymptotic distribution of the change-point estimator under Regimes II and III. (A2) implies that the random graph is dense. The different statistical and probabilistic aspects of sparse random graphs constitute a growing area of research in the recent literature. Most of the results in the sparse setting do not follow from the dense case and different tools and techniques are needed for their analysis; see Remark \ref{rem: sparse} for examples.  Though the convergence rate results established in Sections \ref{sec: DSBM} and \ref{sec: 2step} hold for the sparse setting,  deriving the asymptotic distribution of the change-point estimator under Regimes II and III in sparse random graphs requires further investigation. 
\end{remark}

\subsection{Adaptive Inference} \label{subsec: ADAP}

Next, we present a data adaptive procedure that does \textit{not} require a priori knowledge of the limiting regime. Recall the estimators ${\hat{\tau}}_{m,n}$, $\hat{\hat{\Lambda}}$, $\hat{\hat{\Delta}}$, $\hat{z}$ and $\hat{w}$ of the parameters in the DSBM model  given in (\ref{eqn: dsbmmodel}).  We generate independent $m \times m$ adjacency matrices $A_{t,n,\text{DSBM}}$, $1 \leq t \leq n$, where
\begin{eqnarray}
A_{t,n,\text{DSBM}} = ((A_{ij,(t,n),\text{DSBM}} )) \sim \begin{cases}\text{SBM}(\hat{z},\hat{\hat{\Lambda}}),\ \ \text{if $1 \leq t \leq \lfloor n{\hat{\tau}}_{m,n} \rfloor$} \\
\text{SBM}(\hat{w},\hat{\hat{\Delta}}),\ \ \text{if $\lfloor n{\hat{\tau}}_{m,n} \rfloor < t <n$}. \label{eqn: dsbmmodeladap}
\end{cases}
\end{eqnarray}
\noindent  Obtain
\begin{eqnarray}
\hat{h}_{\text{DSBM}} = \arg \max_{h \in (n(c^{*}- {\hat{\tau}}_{m,n}),n(1-c^{*}-{\hat{\tau}}_{m,n}))} \tilde{L}^{*} ({\hat{\tau}}_{m,n}+h/n,\hat{z},\hat{w},\hat{\hat{\Lambda}},\hat{\hat{\Delta}})
\end{eqnarray}
where
\begin{eqnarray}
\tilde{L}^* ({\hat{\tau}}_{m,n}+h/n,\hat{z},\hat{w},\hat{\hat{\Lambda}},\hat{\hat{\Delta}}) &=&  \frac{1}{n}\sum_{i,j=1}^{m} \bigg[\sum_{t=1}^{n{\hat{\tau}}_{m,n}+h}  (A_{ij,(t,n),\text{DSBM}} - \hat{\hat{\lambda}}_{\hat{z}(i),\hat{z}(j)})^2  \nonumber \\
&& \hspace{-0 cm}+  \sum_{t=n{\hat{\tau}}_{m,n}+h+1}^{n} (A_{ij,(t,n),\text{DSBM}} - \hat{\hat{\delta}}_{\hat{w}(i),\hat{w}(j)})^2  \bigg]. \hspace{1 cm}\label{eqn: estimatecccadap}
\end{eqnarray}

Theorem \ref{thm: adapdsbm} states the asymptotic distribution of $\hat{h}_{\text{DSBM}}$ under a stronger identifiability condition. Specifically,
\vskip 2pt
\noindent \textbf{SNR-ER-ADAP}:  $\frac{\sqrt{n}}{m^2 \sqrt{\log m}} ||\text{Ed}_{z}(\Lambda) - \text{Ed}_{w}(\Delta)||_F^2 \to \infty$
\vskip 2pt
\noindent It is easy to show that SNR-ER-ADAP holds if all  assumptions in Proposition \ref{example1} hold and 
\vskip 2pt 
\noindent \textbf{(AD)} $m = e^{n^{\delta_3}}$  for some $\delta_1,\delta_2,\delta_3>0$ and $0 <\delta_1 + \delta_2 + \delta_3 /2 < 1/2$ ($\delta_1, \delta_2$ are as in Proposition \ref{example1})
\newline
is satisfied. 
\vskip 2pt
\noindent Specifically, Examples (A)-(D) in Section \ref{sec: compare} satisfy SNR-ER-ADAP in the presence of condition (AD). 
\vskip 2pt
\noindent We also need the following condition to ensure that $\hat{z}$ and $\hat{w}$ are consistent estimates for $z$ and $w$, respectively.
\vskip 2pt
\noindent \textbf{(A1-ADAP)} $\frac{Km}{n\nu_{m,n}^2} ||\text{Ed}_{z}(\Lambda) - \text{Ed}_{w}(\Delta)||_2^{-1} \to 0$. 
\vskip 2pt
\noindent Under SNR-ER-ADAP,  one can reduce A1-ADAP to $\frac{K}{(n^3\log m)^{1/4}\nu_{m,n}^2 } \to O(1)$. This  holds  whenever  within and between community connection probabilities are equal (that is $\lambda_{ij} = q_1, \delta_{ij}=q_2\  \forall\  i \neq j$ and $\lambda_{ii}=p_1, \delta_{ii}=p_2\ \forall i$),  balanced communities of size $O(m/K)$ are present,  and their number is $K = O(m^{2/3})$. This is because the first two conditions implies $\nu_{m,n} = O(m/k)$ (for example see Example \ref{example: misclassnew}). 

We also require $\log m = o(\sqrt{n})$, so that the entries of $\text{Ed}_{\hat{z}}(\hat{\hat{\Lambda}})$ and $\text{Ed}_{\hat{w}}(\hat{\hat{\Delta}})$ are bounded away from $0$ and $1$.  Note that this assumption implies $0<\delta_3 < 1/2$ in (AD). 
\vskip 2pt

\begin{theorem} \label{thm: adapdsbm} \textbf{(Asymptotic distribution of $\hat{h}_{\text{DSBM}}$)}
 Suppose (A2), SNR-ER-ADAP and A1-ADAP  hold and $\log m = o(\sqrt{n})$. Then, the following results are true.

\noindent $(a)$ If $||\text{Ed}_{z}(\Lambda)-\text{Ed}_{w}(\Delta)||_F \to \infty$, then $\lim_{n \to \infty} P(\hat{h}_{\text{DSBM}}=0) =1$

\noindent $(b)$ If (A3) holds  and $||\text{Ed}_{z}(\Lambda) -\text{Ed}_{w}(\Delta)||_F \to 0$, then
\begin{eqnarray}
||\text{Ed}_{z}(\Lambda) - \text{Ed}_{w}(\Delta)||_F^2 \hat{h}_{\text{DSBM}} \stackrel{\mathcal{D}}{\to}  \gamma^{-2}\arg \max_{h \in \mathbb{R}} (-0.5|h| + B_h)\ \ \ \ 
\end{eqnarray}
where $B_h$ corresponds to a standard Brownian motion.

\noindent $(c)$ If  (A4)-(A7)  hold  and $||\text{Ed}_{z}(\Lambda) - \text{Ed}_{w}(\Delta)||_F \to c >0$, then
\begin{eqnarray}
\hat{h}_{\text{DSBM}} &\stackrel{\mathcal{D}}{\to}& \arg \max_{h \in \mathbb{Z}} (D(h) + C(h) + A(h)), \nonumber  
\end{eqnarray}
where $D(\cdot)$, $C(\cdot)$ and $A(\cdot)$  are same as (\ref{eqn: msethm2d})-(\ref{eqn: msethm2a}). 
 \end{theorem}

\noindent The proof of the Theorem is given in Section \ref{subsec: adapdsbm}.
  
\begin{remark} \label{rem: adap_prob}
Though SNR-ER-ADAP is stronger than SNR-ER, we have not been able to relax it. 
Data-driven methods usually require stronger assumptions even in simple models like the Erd\H{o}s-R\'{e}nyi random graph one,
with a single change-point---see for example \cite{BBM2017}.  To establish the form of the asymptotic distribution of $\hat{h}_{\text{DSBM}}$, we need to establish convergence in probability as stated in Lemma \ref{lem: adaplem} which do not occur without SNR-ER-ADAP.  
\end{remark}

\begin{remark} \label{rem: adapcompare} Similar conclusions as in Theorem \ref{thm: adapdsbm} hold for $\tilde{\tilde{\tau}}_{m,n}$ and $\tau_{m,n}^*$, but under the stronger assumptions (A1) and (A1*), respectively, instead of A1-ADAP. Details are omitted to avoid repetition. Moreover, as discussed in Sections \ref{sec: DSBM} and \ref{sec: compare},  (A1) and (A1*) are difficult to satisfy whereas SNR-ER-ADAP may often hold (see discussion after stating the assumption). Also computation of $\tilde{\tilde{\tau}}_{m,n}$ is itself  expensive and performing adaptive inference for it will make it even more so. 
For these reasons,  we have mainly focused on adaptive inference for $\hat{\tau}_{m,n}$. 
\end{remark}  
  
\begin{remark} \label{rem: adapdiscuss} Note that the asymptotic distribution of $\hat{h}_{\text{DSBM}}$ is  identical to the asymptotic distribution of ${\hat{\tau}}_{m,n}$. Therefore, in practice we can simulate $\hat{h}_{\text{DSBM}}$  for a large number of replicates and use their empirical quantiles as estimates of the quantiles of the limiting distribution under the (unknown) true regime. 
Moreover, the adaptive inference is a computationally expensive procedure and  comes at a certain cost, namely the  stronger assumption SNR-ER-ADAP. 
\end{remark}

\section{Concluding Remarks} \label{sec:concluding-remarks}

In this paper, we have addressed the change-point problem in the context of DSBM. We establish consistency of the change-point estimator under a suitable identifiability condition and a second condition that controls the misclassification rate arising from using clustering for assigning nodes to communities and discuss the stringency of the latter condition. Further, we propose a fast computational strategy that ignores the underlying community structure but provides a consistent estimate of the change-point. Further, for both methods under their respective identifiability and certain additional regularity conditions,  we establish rates of convergence and derive the asymptotic distributions of the change-point estimators.

In addition, this work identifies an interesting issue that requires further research; namely, a range of models where the SNR-DSBM identifiability condition holds, but the misclassification rate condition (A1) needed for the ``every time point clustering algorithm" and the identifiability condition (SNR-ER) of the alternative strategy fails to hold. In that range, no general strategy for solving the change-point problem for DSBM seems to be currently available.

Note that even adopting a very simple (almost toy-like) structure for the DSBM (for example the framework in \cite{GMZZ2015}), detecting the change-point remains a hard problem if neither (A1) or (SNR-ER) hold. The modified algorithm proposed in Remark \ref{rem: nc} is a good possibility, but the limitations discussed in Remark \ref{rem: limitation-nc} remain a concern.  At present, we are not aware of any detection algorithm that works by imposing only SNR-DSBM type assumptions for the change-point problem in DSBM.

\vskip 5pt
\noindent \textit{\textbf{Acknowledgments}}. We are thankful to Dr. Daniel Sussman for valuable comments and suggestions. We also thank the two anonymous referees for many helpful comments and suggestions that led to improvements in the structure of the paper and presentation of the material.

The work of Monika Bhattacharjee was entirely supported by a fellowship of Informatics Institute, University of Florida, USA. Currently, Monika Bhattacharjee is at the Department of Mathematics, Indian Institute of Technology Bombay.

The work of Moulinath Banerjee was supported  by NSF DMS-1712962.

The work of George Michailidis was supported in part by NSF grants DMS 1830175, IIS 1632730 and DMS 1821220.


\section{Proofs and Other Technical Material} \label{sec: proofs}
Throughout this section, $C$ is a generic  positive constant. Often we write $\tau_{m,n}$ as $\tau$. 

\subsection{Proof of Theorem \ref{thm: mismiscluster}} \label{subsec: mismiscluster}
Without loss of generality, assume $\tau<b$. By Lemmas $5.1$ and $5.3$ of \cite{LR2015}, with probability tending to $1$, we have
\begin{eqnarray}
\mathcal{M}_{b,n,m} &\leq & C\frac{K}{n\nu_{m,n}^2} \bigg[ ||\frac{1}{n}\sum_{t=1}^{n\tau}(A_{t,n} - \text{Ed}_{z}(\Lambda)) ||_F^2  +  ||\frac{1}{n}\sum_{t=n\tau +1}^{nb}(A_{t,n} -\text{Ed}_{w}(\Delta)) ||_F^2 \nonumber \\
&& \hspace{2cm}+ ||\frac{1}{n}\sum_{t=n\tau +1}^{nb} (\text{Ed}_{z}(\Lambda)-\text{Ed}_{w}(\Delta))||_F^2 \bigg] \nonumber \\
&=& C\frac{K}{n\nu_{m,n}^2} (A_1+A_2 + |\tau-b|\ ||\text{Ed}_{z}(\Lambda)-\text{Ed}_{w}(\Delta)||_F^2),\ \text{say}. \nonumber
\end{eqnarray}
Now by Theorem $5.2$ of \cite{LR2015},  $A_1, A_2 = O_{\text{P}}(m)$.  Thus,
\begin{eqnarray}
\mathcal{M}_{b,n,m} = O_{\text{P}}\left(\frac{K}{n\nu_{m,n}^2} ( m + |\tau-b|||\text{Ed}_{z}(\Lambda)-\text{Ed}_{w}(\Delta)||_F^2)\right). \nonumber
\end{eqnarray}
This completes the proof of Theorem \ref{thm: mismiscluster}. $\blacksquare$

\subsection{Selected useful lemmas}

The following two lemmas directly quoted from \cite{Wellner1996empirical} are needed to establish Theorems \ref{lem: b1} and \ref{lem: b2lse}.
\begin{lemma} \label{lem: wvan1}
For each $n$, let $\mathbb{M}_n$ and $\tilde{\mathbb{M}}_n$ be stochastic processes indexed by a set $\mathcal{T}$. Let $\tau_n\ \text{(possibly random)} \in \mathcal{T}_n \subset \mathcal{T}$ and 
$d_n(b,\tau_n)$ be a map (possibly random) from $\mathcal{T}$ to $[0,\infty)$. Suppose that for every large $n$ and $\delta \in (0,\infty)$
\begin{eqnarray}
&& \sup_{\delta/2 < d_n(b,\tau_n) < \delta,\  b \in \mathcal{T}} (\tilde{\mathbb{M}}_n(b) - \tilde{\mathbb{M}}_n(\tau_n)) \leq -C\delta^2, \label{eqn: lemcon1} \\
&& E\sup_{\delta/2 < d_n(b,\tau_n) < \delta,\  b \in \mathcal{T}}  \sqrt{n} |\mathbb{M}_n(b) - \mathbb{M}_n(\tau_n) - (\tilde{\mathbb{M}}_n(b) - \tilde{\mathbb{M}}_n(\tau_n))| \leq  C\phi_{n}(\delta), \label{eqn: lemcon2}
\end{eqnarray}
for some $C>0$ and for function $\phi_n$ such that $\delta^{-\alpha}\phi_n(\delta)$ is decreasing in $\delta$ on $(0,\infty)$ for some $\alpha <2$. Let $r_n$ satisfy
\begin{eqnarray} \label{eqn: lemrn}
r_n^2  \phi(r_n^{-1}) \leq \sqrt{n}\ \ \text{for every  $n$}.
\end{eqnarray}
Further, suppose that the sequence $\{\hat{\tau}_{m,n}\}$ takes its values in $\mathcal{T}_n$ and 
satisfies $\mathbb{M}_n(\hat{\tau}_n) \geq \mathbb{M}_n(\tau_n) - O_P (r_n^{-2})$ for large enough $n$. Then,
$r_n d_{n}(\hat{\tau}_n,\tau_n) = O_P (1)$. 
\end{lemma}

\begin{lemma}  \label{lem: wvandis1}
Let $\mathbb{M}_n$ and $\mathbb{M}$ be two stochastic processes indexed by a metric space $\mathcal{T}$,
such that $\mathbb{M}_n \Rightarrow \mathbb{M}$ in $l^{\infty}(\mathcal{C})$ for every compact set $\mathcal{C} \subset \mathcal{T}$, that is,
\begin{align}
\sup_{h \in \mathcal{C}} |\mathbb{M}_n(h) - \mathbb{M}(h)| \stackrel{P}{\to} 0.
\end{align} 
Suppose that almost all sample paths $h \to \mathbb{M}(h)$ are upper semi-continuous and possess a unique maximum at a (random) point $\hat{h}$, which as a random map in $\mathcal{T}$ is tight. If the sequence $\hat{h}_n$ is uniformly tight and satisfies $\mathbb{M}_n(\hat{h}_n) \geq \sup_{n} \mathbb{M}_n(h) - o_{P}(1)$, then $\hat{h}_n \stackrel{\mathcal{D}}{\to} \hat{h}$ in $\mathcal{T}$.
\end{lemma}

The following lemma is needed in the proof of Theorem \ref{thm: adapdsbm}.

\begin{lemma} \label{lem: adaplem}
Suppose SNR-ER-ADAP, A1-ADAP holds  and $\log m = o(\sqrt{n})$. Then, the following statements hold. 
\vskip 2pt 
\noindent(a)  $\frac{||\text{Ed}_{\hat{z}}(\hat{\hat{\Lambda}})-\text{Ed}_{\hat{w}}(\hat{\hat{\Delta}}) ||_F^2}{||\text{Ed}_{z}(\Lambda)-\text{Ed}_{w}(\Delta)  ||_F^2} \stackrel{\text{P}}{\to} 1$.
\vskip 2pt
\noindent (b) If $||\text{Ed}_{z}(\Lambda)-\text{Ed}_{w}(\Delta)  ||_F^2 \to 0$, then
\begin{eqnarray}
\frac{\sum_{i,j=1}^{m}(\lambda_{z(i)z(j)}-\delta_{w(i)w(j)})^2 \lambda_{z(i)z(j)}(1-\lambda_{z(i)z(j)})}{||\text{Ed}_{z}(\Lambda)-\text{Ed}_{w}(\Delta)  ||_F^2}  & \stackrel{\text{P}}{\to} & \gamma^2, \nonumber \\
\frac{\sum_{i,j=1}^{m}(\lambda_{z(i)z(j)}-\delta_{w(i)w(j)})^2 \delta_{w(i)w(j)}(1-\delta_{w(i)w(j)})}{||\text{Ed}_{z}(\Lambda)-\text{Ed}_{w}(\Delta)  ||_F^2} \stackrel{\text{P}}{\to} \gamma^2. \nonumber
\end{eqnarray}
\vskip 2pt
\noindent (c) If $||\text{Ed}_{z}(\Lambda)-\text{Ed}_{w}(\Delta)  ||_F^2 \to c^2>0$, then
\begin{eqnarray}
\sum_{i,j\in \mathcal{K}_n}(\lambda_{z(i)z(j)}-\delta_{w(i)w(j)})^2 \lambda_{z(i)z(j)}(1-\lambda_{z(i)z(j)}) \stackrel{\text{P}}{\to} \tilde{\gamma}^2, \nonumber \\
\sum_{i,j\in \mathcal{K}_n}(\lambda_{z(i)z(j)}-\delta_{w(i)w(j)})^2 \delta_{w(i)w(j)}(1-\delta_{w(i)w(j)}) \stackrel{\text{P}}{\to} \tilde{\gamma}^2.\nonumber
\end{eqnarray} 
\end{lemma}

\begin{proof}
We only show the proof of part (a), since parts (b) and (c) follow employing similar arguments. 
\begin{eqnarray}
\bigg|\frac{||\text{Ed}_{\hat{z}}(\hat{\hat{\Lambda}})-\text{Ed}_{\hat{w}}(\hat{\hat{\Delta}}) ||_F^2}{||\text{Ed}_{z}(\Lambda)-\text{Ed}_{w}(\Delta)  ||_F^2}-1 \bigg| &=& \frac{|||\text{Ed}_{\hat{z}}(\hat{\hat{\Lambda}})-\text{Ed}_{\hat{w}}(\hat{\hat{\Delta}}) ||_F^2-||\text{Ed}_{z}(\Lambda)-\text{Ed}_{w}(\Delta)  ||_F^2|}{||\text{Ed}_{z}(\Lambda)-\text{Ed}_{w}(\Delta)  ||_F^2} \nonumber \\
& \leq & \frac{||\text{Ed}_{\hat{z}}(\hat{\hat{\Lambda}}) -\text{Ed}_{z}(\Lambda) - \text{Ed}_{\hat{w}}(\hat{\hat{\Delta}}) + \text{Ed}_{w}(\Delta) ||_F^2}{||\text{Ed}_{z}(\Lambda)-\text{Ed}_{w}(\Delta)  ||_F^2} \nonumber \\
& \leq & \frac{||\text{Ed}_{\hat{z}}(\hat{\hat{\Lambda}}) -\text{Ed}_{z}(\Lambda)||_F^2 +||\text{Ed}_{\hat{w}}(\hat{\hat{\Delta}}) - \text{Ed}_{w}(\Delta) ||_F^2 }{||\text{Ed}_{z}(\Lambda)-\text{Ed}_{w}(\Delta)  ||_F^2}\nonumber 
\end{eqnarray}
Therefore, part (a) follows from Theorem \ref{thm: c3}, SNR-ER-ADAP, A1-ADAP and $\log m = o(\sqrt{n})$. 
\end{proof}

\subsection{Proof of Theorem \ref{lem: b1}} \label{subsec: b1}
Throughout this proof, we use the following simplified notation for ease of exposition: $A_{ijt} = A_{ij,(t,n)}, z_1 = \tilde{z}_{b,n,m}, z_2 = \tilde{z}_{\tau_{m,n},n,m}, w_1 = \tilde{w}_{b,n,m}, w_2 = \tilde{w}_{\tau_{m,n},n,m}$, $\Lambda_1 = \tilde{\Lambda}_{\tilde{z}_{b,n,m},(b,n),m}$, $\Lambda_2 = \tilde{\Lambda}_{\tilde{z}_{\tau_{m,n},n,m},(\tau_{m,n},n),m}$, $\Lambda_3 = \tilde{\Lambda}_{\tilde{z}_{\tau_{m,n},n},(b,n),m}$, $\Delta_1 = \tilde{\Delta}_{\tilde{w}_{b,n,m},(b,n),m}$, $\Delta_2 = \tilde{\Delta}_{\tilde{w}_{\tau_{m,n},n,m},(\tau_n,n),m}$, $\Delta_w = \tilde{\Delta}_{w,(b,n),m}$, $\lambda_{uv,1} = \tilde{\lambda}_{uv,\tilde{z}_{b,n,m},(b,n),m}$, $\lambda_{uv,2} = \tilde{\lambda}_{uv,\tilde{z}_{\tau_{m,n},n,m},(\tau_{m,n},n),m}$,\\ $\lambda_{uv,3} = \tilde{\lambda}_{uv,\tilde{z}_{\tau_{m,n},n},(b,n),m}$,  $\delta_{uv,1} = \tilde{\delta}_{uv,\tilde{w}_{b,n,m},(b,n),m}$, $\delta_{uv,2} = \tilde{\delta}_{uv,\tilde{w}_{\tau_{m,n},n,m},(\tau_{m,n},n),m}$, $\delta_{uv,w} = \tilde{\delta}_{uv,w,(b,n),m}$.  Suppose $b<\tau_{m,n}$. Similar arguments work when $b>\tau_{m,n}$.  Note that 
\begin{eqnarray}
\tilde{\tilde{\tau}}_{m,n}  = \arg \min_{b \in (c^*,1-c^*)} \tilde{L}(b,z_1,w_1,\Lambda_1,\Delta_1)  \nonumber 
\end{eqnarray}
where
\begin{eqnarray}
\tilde{L} (b,z_1,w_1,\Lambda_1,\Delta_1)  = \frac{1}{n} \sum_{i,j=1}^{m} \bigg[  \sum_{t=1}^{nb}(A_{ijt}-{\lambda}_{z_1(i)z_1(j),1})^2  + \sum_{t=nb+1}^{n} (A_{ijt}-\hat{\delta}_{w_1(i) w_1(i),1})^2 \bigg].
\end{eqnarray}

\noindent To prove Theorem \ref{lem: b1}, we need  Lemma \ref{lem: wvan1}  quoted from \cite{Wellner1996empirical}.  For our purpose, we make use of the above lemma with $\mathbb{M}_n (\cdot) = \tilde{L} (\cdot,\tilde{z}_{\cdot,n,m},\tilde{w}_{\cdot,n,m},$\\ $\tilde{\Lambda}_{\tilde{z}_{\cdot,n,m},(\cdot,n),m},\tilde{\Delta}_{\tilde{w}_{\cdot,n,m},(\cdot,n),m})$, $\tilde{\mathbb{M}}_n(\cdot) = E\tilde{L} (\cdot,\tilde{z}_{\cdot,n,m},\tilde{w}_{\cdot,n,m},\tilde{\Lambda}_{\tilde{z}_{\cdot,n,m},(\cdot,n),m},\tilde{\Delta}_{\tilde{w}_{\cdot,n,m},(\cdot,n),m})$,
$\mathcal{T} = [0,1]$, $\mathcal{T}_n = \{1/n,2/n,\ldots, (n-1)/n,1\} \cap [c^{*},1-c^{*}]$,  $d_n(b,\tau_{m,n}) = ||\text{Ed}_{z}(\Lambda)-\text{Ed}_{w}(\Delta)||_F$\\ $\sqrt{|b - \tau_{m,n}|}$, $\phi_n(\delta) = \delta$, $\alpha = 1.5$, $r_n = \sqrt{n}(\rho_{m,n})^{-1/2}$ and $\hat{\tau}_{n} = \tilde{\tilde{\tau}}_{m,n}$.  Thus, to prove Theorem \ref{lem: b1}, it suffices to establish that for some $C>0$,
\begin{eqnarray}
&&  E(\mathbb{M}_n (b) - \mathbb{M}_n (\tau_{m,n})) \leq -C||\text{Ed}_{z}(\Lambda) - \text{Ed}_{w}(\Delta)||_F^2 |b - \tau_{m,n}|\ \ \text{and} \label{eqn: lsecon1} \\
&& E\sup_{\stackrel{\delta/2 < d_n(b,\tau_{m,n}) < \delta,} { b \in \mathcal{T}}}   |\mathbb{M}_n (b) - \mathbb{M}_n (\tau_{m,n}) - E(\mathbb{M}_n (b) - \mathbb{M}_n (\tau_{m,n}))| \leq  C\frac{\delta \sqrt{\rho_{m,n}}}{\sqrt{n}}. 
\label{eqn: lsecon2}
\end{eqnarray}
As the right side of (\ref{eqn: lsecon1}) and (\ref{eqn: lsecon2}) are independent of $z_1,z_2,w_1,w_2$, it suffices to show 
\begin{eqnarray}
&&  E^{*}(\mathbb{M}_n (b) - \mathbb{M}_n (\tau_{m,n})) \leq -C||\text{Ed}_{z}(\Lambda) - \text{Ed}_{w}(\Delta)||_F^2 |b - \tau_{m,n}|\ \ \text{and} \label{eqn: lsecon1f} \\
&& E^{*}\sup_{\stackrel{\delta/2 < d_n(b,\tau_{m,n}) < \delta,}{  b \in \mathcal{T}}}   |\mathbb{M}_n (b) - \mathbb{M}_n (\tau_{m,n}) - E(\mathbb{M}_n (b) - \mathbb{M}_n (\tau_{m,n}))| \leq  C\frac{\delta \sqrt{\rho_{m,n}}}{\sqrt{n}}. 
\label{eqn: lsecon2f}
\end{eqnarray}
where $E^{*}(\cdot) = E(\cdot| z_1,w_1,z_2,w_2)$.   Similarly denote $V^* = V(\cdot|z_1,z_2,w_1,w_2)$ and $\text{Cov}^{*}(\cdot) = \text{Cov}(\cdot|z_1,z_2,w_1,w_2)$. 
\vskip 2pt
Note that the left hand side of (\ref{eqn: lsecon2f}) is dominated by
\begin{eqnarray}
\left(E^{*}\sup_{\delta/2 < d_n(b,\tau_{m,n}) < \delta,\  b \in \mathcal{T}}   (\mathbb{M}_n (b) - \mathbb{M}_n (\tau_{m,n}) - E(\mathbb{M}_n (b) - \mathbb{M}_n (\tau_{m,n})))^2\right)^{1/2}. \label{eqn: doobprelse}
\end{eqnarray}
By Doob's martingale inequality, (\ref{eqn: doobprelse}) is further dominated by
\begin{eqnarray}
(\text{V}^{*}(\mathbb{M}_n (b) - \mathbb{M}_n (\tau_{m,n})))^{1/2}\ \ \text{where $d_n(b,\tau_{m,n}) = \delta$}.
\end{eqnarray}
Thus, to prove Theorem \ref{lem: b1}, it suffices to show that for some $C>0$,
\begin{eqnarray}
\text{V}^{*}(\mathbb{M}_n (b) - \mathbb{M}_n (\tau_{m,n})) \leq Cn^{-1} d_n^2(b,\tau_{m,n}) \rho_{m,n}. \label{eqn: lsecon2final}
\end{eqnarray}
Hence, it sufficies to prove (\ref{eqn: lsecon1f}) and (\ref{eqn: lsecon2final}) to establish Theorem \ref{lem: b1}. We shall prove these for $b<\tau_{m,n}$. Similar arguments work when $b \geq \tau_{m,n}$. 

Denote by $L_1 = \tilde{L} (b,z_1,w_1,\Lambda_1,\Delta_1)$ and $L_2 = \tilde{L} (\tau,z_2,w_2,\Lambda_2,\Delta_2)$. Hence,
\begin{eqnarray}
L_1 - L_2 = A(b) + B(b) + D(b), \label{eqn: abd1}
\end{eqnarray}
where
\begin{eqnarray}
A(b) &=& \frac{1}{n} \sum_{i,j=1}^{m} \sum_{t=1}^{nb} \bigg[(A_{ijt}-\lambda_{z_1(i)z_1(j),1})^2-(A_{ijt} -\lambda_{z_2(i)z_2(j),2})^2 \bigg], \nonumber \\
B(b) &=& \frac{1}{n} \sum_{i,j=1}^{m} \sum_{t=nb+1}^{n\tau} \bigg[(A_{ijt}-\delta_{w_1(i)w_1(j),1})^2 - (A_{ijt}-\delta_{z_2(i)z_2(j),2})^2  \bigg],\nonumber \\
D(b) &=&   \frac{1}{n} \sum_{i,j=1}^{m} \sum_{t=n\tau+1}^{n} \bigg[ (A_{ijt}-\delta_{w_1(i)w_1(j),1})^2-(A_{ijt}-\delta_{w_2(i)w_2(j),2})^2\bigg]. \nonumber 
\end{eqnarray}
Consider the first term of $A(b)$ as follows. 
\begin{eqnarray}
&& \frac{1}{nb}\sum_{t=1}^{nb}\sum_{i,j=1}^{m}(A_{ijt}-\lambda_{z_1(i)z_1(j),1})^2 \nonumber \\
&=&  \frac{1}{nb}\sum_{t=1}^{nb}\sum_{i,j=1}^{m} A_{ijt}^2  + \sum_{u,v=1}^{K} s_{u,z_1}s_{v,z_1}(\lambda_{uv,1})^2 - \sum_{u,v=1}^{K}s_{u,z_1}s_{v,z_1} (\lambda_{z_1(i)z_1(j),1})  \sum_{\stackrel{i: z_1(i)=u}{j: z_1(j)=v}} \sum_{t=1}^{nb} \frac{A_{ijt}}{nb} \nonumber \\
&=&  \frac{1}{nb}\sum_{t=1}^{nb}\sum_{i,j=1}^{m} A_{ijt}^2 - \sum_{u,v=1}^{K} s_{u,z_1}s_{v,z_1}(\lambda_{uv,1})^2.  \nonumber
\end{eqnarray}
Similarly, the second term of $A(b)$ is
\begin{eqnarray}
\frac{1}{nb}\sum_{t=1}^{nb} \sum_{i,j=1}^{m}(A_{ijt} -\lambda_{z_2(i)z_2(j),2})^2 
&=& \frac{1}{nb}\sum_{t=1}^{nb}\sum_{i,j=1}^{m} A_{ijt}^2 + \sum_{u,v=1}^{K} s_{u,z_2}s_{v,z_2}(\lambda_{uv,2})^2 \nonumber \\
&& \hspace{3 cm}- 2\sum_{u,v=1}^{K} s_{u,z_2}s_{v,z_2}\lambda_{uv,2}\lambda_{uv,3}. \nonumber 
\end{eqnarray}
\noindent Therefore, 
\begin{eqnarray}
bE^* (A(b)) &=& - \sum_{u,v=1}^{K} s_{u,z_2}s_{v,z_2}E^* (\lambda_{uv,2})^2 - b \sum_{u,v=1}^{K} s_{u,z_1}s_{v,z_1}E^* (\lambda_{uv,1})^2 \nonumber \\
&& \hspace{3 cm}+ 2\sum_{u,v=1}^{K} s_{u,z_2}s_{v,z_2}E^{*}(\lambda_{uv,2}\lambda_{uv,3}). \nonumber 
\end{eqnarray}
Let $S((u,v,f),(a,b,g))$ be the total number of edges which connect communities $u$ and $v$ under community structure $f$,  and also communities
$a$ and $b$ under community structure $g$. Therefore,
\begin{eqnarray}
E^* (\lambda_{uv,1})^2 &=& V^{*} (\lambda_{uv,1}) + (E(\lambda_{uv,1}))^2 \nonumber \\
&=& \frac{1}{nb} \frac{1}{(s_{u,z_1}s_{v,z_1})^2} \sum_{a,b=1}^{K} S((u,v,z_1),(a,b,z))\lambda_{ab}(1-\lambda_{ab}) \nonumber \\
&&\hspace{1 cm} + \left(\frac{1}{s_{u,z_1}s_{v,z_1}} \sum_{a,b=1}^{K} S((u,v,z_1),(a,b,z))\lambda_{ab} \right)^2, \nonumber \\
E^* (\lambda_{uv,2})^2 &=& V^{*} (\lambda_{uv,2}) + (E(\lambda_{uv,2}))^2 \label{eqn: e2} \\
&=& \frac{1}{n\tau} \frac{1}{(s_{u,z_2}s_{v,z_2})^2} \sum_{a,b=1}^{K} S((u,v,z_2),(a,b,z))\lambda_{ab}(1-\lambda_{ab}) \nonumber \\
&&\hspace{1 cm} + \left(\frac{1}{s_{u,z_2}s_{v,z_2}} \sum_{a,b=1}^{K} S((u,v,z_2),(a,b,z))\lambda_{ab} \right)^2, \nonumber \\
E^* (\lambda_{uv,2}\lambda_{uv,3}) &=& \text{Cov}^* (\lambda_{uv,2},\lambda_{uv,3}) + (E(\lambda_{uv,2}))(E(\lambda_{uv,3})) \nonumber \\
&=& \frac{1}{n\tau} \frac{1}{(s_{u,z_2}s_{v,z_2})^2} \sum_{a,b=1}^{K} S((u,v,z_2),(a,b,z))\lambda_{ab}(1-\lambda_{ab}) \nonumber \\
&&\hspace{1 cm} + \left(\frac{1}{s_{u,z_2}s_{v,z_2}} \sum_{a,b=1}^{K} S((u,v,z_2),(a,b,z))\lambda_{ab} \right)^2. \nonumber 
\end{eqnarray}
Hence,
\begin{eqnarray}
bE^*(A(b)) &=& b(A_1(b) + A_2(b)) \nonumber 
\end{eqnarray}
where
\begin{eqnarray}
A_1(b) &=& -\sum_{u,v=1}^{K} \frac{1}{nb} \frac{1}{s_{u,z_1}s_{v,z_1}} \sum_{a,b=1}^{K} S((u,v,z_1),(a,b,z))\lambda_{ab}(1-\lambda_{ab}) \nonumber \\
&& + \sum_{u,v=1}^{K} \frac{1}{n\tau} \frac{1}{s_{u,z_2}s_{v,z_2}} \sum_{a,b=1}^{K} S((u,v,z_2),(a,b,z))\lambda_{ab}(1-\lambda_{ab}), \nonumber \\
A_2(b) &=& -\sum_{u,v=1}^{K} \frac{1}{s_{u,z_1}s_{v,z_1}} \left(\sum_{a,b=1}^{K} S((u,v,z_1),(a,b,z))\lambda_{ab}\right)^2 \nonumber \\
&& + \sum_{u,v=1}^{K} \frac{1}{s_{u,z_2}s_{v,z_2}} \left(\sum_{a,b=1}^{K} S((u,v,z_2),(a,b,z))\lambda_{ab}\right)^2. \label{eqn: v10a2}
\end{eqnarray}
Note that
\begin{eqnarray}
A_1(b) &\geq & -\sum_{u,v=1}^{K} \left(\frac{1}{nb} - \frac{1}{n\tau}\right)\frac{1}{s_{u,z_1}s_{v,z_1}} \sum_{a,b=1}^{K} S((u,v,z_1),(a,b,z))\lambda_{ab}(1-\lambda_{ab}) \label{eqn: com1} \\
&& -\sum_{u,v=1}^{K} \frac{1}{n\tau}\frac{1}{s_{u,z_1}s_{v,z_1}} \sum_{a,b=1}^{K} S((u,v,z_1),(a,b,z))\lambda_{ab}(1-\lambda_{ab}) \nonumber \\
&& + \sum_{u,v=1}^{K} \frac{1}{n\tau} \frac{1}{s_{u,z_2}s_{v,z_2}} \sum_{a,b=1}^{K} S((u,v,z_2),(a,b,z))\lambda_{ab}(1-\lambda_{ab}) \nonumber \\
&\geq & -C \left(\frac{1}{nb} - \frac{1}{n\tau}\right)K^2 \rho_{m,n}    \nonumber \\
&&-C\sum_{u,v=1}^{K} \frac{1}{n} \frac{1}{s_{u,z_1}s_{v,z_1}} \sum_{(a,b)\neq (u,v)} S((u,v,z_1),(a,b,z)) \nonumber \\
&& - C\sum_{u,v=1}^{K} \frac{1}{n} \frac{1}{s_{u,z_2}s_{v,z_2}} \sum_{(a,b)\neq (u,v)} S((u,v,z_2),(a,b,z))\nonumber \\
&& -  C\frac{1}{n}\sum_{u,v=1}^{K} (S((u,v,z_1),(u,v,z))) |(s_{u,z_1}s_{v,z_1})^{-1} - (s_{u,z}s_{v,z})^{-1}| \nonumber \\
&& - C\frac{1}{n}\sum_{u,v=1}^{K} (S((u,v,z_2),(u,v,z))) |(s_{u,z_2}s_{v,z_2})^{-1} - (s_{u,z}s_{v,z})^{-1}| \nonumber \\
&& -C\frac{1}{n}\sum_{u,v=1}^{K} (s_{u,z}s_{v,z})^{-1} |(S((u,v,z_1),(u,v,z))) - (S((u,v,z_2),(u,v,z))| \nonumber \\
& \geq &  -C(\tau-b)\frac{K^2}{n}\rho_{m,n} - C(\tau-b)\mathcal{M}_{b,n,m}^2 \rho_{m,n}^2. 
\end{eqnarray}
Further,
\begin{eqnarray}
A_2(b) &\geq & -C\sum_{u,v=1}^{K} \frac{1}{s_{u,z_1}s_{v,z_1}} \left(\sum_{(a,b) \neq (u,v)} S((u,v,z_1),(a,b,z))\right)^2 \rho_{m,n}^2 \nonumber \\
&& - C\sum_{u,v=1}^{K} \frac{1}{s_{u,z_1}s_{v,z_1}} \left( S((u,v,z_1),(u,v,z))\right)^2 \rho_{m,n}^2 \nonumber \\
&& -C\sum_{u,v=1}^{K} \frac{1}{s_{u,z_2}s_{v,z_2}} \left(\sum_{(a,b) \neq (u,v)} S((u,v,z_2),(a,b,z))\right)^2\rho_{m,n}^2   \nonumber \\
&& - C\sum_{u,v=1}^{K} \frac{1}{s_{u,z_2}s_{v,z_2}} \left( S((u,v,z_2),(u,v,z))\right)^2 \rho_{m,n}^2\nonumber \\
& \geq & -C(\tau-b)n^2 \mathcal{M}_{b,n,m}^2 \rho_{m,n}^2 \nonumber \\
&& \ \ \ \ - C\sum_{u,v=1}^{K} (S((u,v,z_1),(u,v,z)))^2 |(s_{u,z_1}s_{v,z_1})^{-1} - (s_{u,z}s_{v,z})^{-1}| \rho_{m,n}^2 \nonumber \\
&& - C\sum_{u,v=1}^{K} (S((u,v,z_2),(u,v,z)))^2 |(s_{u,z_2}s_{v,z_2})^{-1} - (s_{u,z}s_{v,z})^{-1}|\rho_{m,n}^2  \nonumber \\
&& -C\sum_{u,v=1}^{K} (s_{u,z}s_{v,z})^{-1} |(S((u,v,z_1),(u,v,z)))^2 - (S((u,v,z_2),(u,v,z))^2| \rho_{m,n}^2 \nonumber \\
&  \geq & -C(\tau-b)n^2 \mathcal{M}_{b,n,m}^2\rho_{m,n}^2 . \label{eqn: v10a21}
\end{eqnarray}
This proves
\begin{eqnarray}
E^*(A(b)) \geq -C(\tau-b) (\frac{K^2}{n}+n^2\mathcal{M}_{b,n,m}^2\rho_{m,n}^2 ). \label{eqn: compa}
\end{eqnarray}
Next, consider $B(b)$. Define 
\begin{eqnarray}
\mu_{uv,1} &=& \frac{1}{n(\tau-b)}\sum_{t=nb+1}^{n\tau} \frac{1}{s_{u,z_2}s_{v,z_2}} \sum_{\stackrel{i: z_2(i)=u}{j: z_2(j)=v}} A_{ijt}, \nonumber \\
\mu_{uv,2} &=& \frac{1}{n(\tau-b)}\sum_{t=nb+1}^{n\tau} \frac{1}{s_{u,w_1}s_{v,w_1}} \sum_{\stackrel{i: w_1(i)=u}{j: w_1(j)=v}} A_{ijt}. \nonumber 
\end{eqnarray}
Note that
\begin{eqnarray}
B(b) &=& \frac{1}{n} \sum_{t=nb+1}^{n\tau}\sum_{i,j=1}^{m}\bigg[(A_{ijt}-\delta_{w_1(i)w_1(j),1})^2 - (A_{ijt}-\lambda_{z_2(i)z_2(j),2})^2 \bigg] \nonumber \\
&=& \frac{1}{n}  \sum_{t=nb+1}^{n\tau} \sum_{i,j=1}^{m} \bigg[  (\delta_{w_1(i)w_1(j),1})^2 - (\lambda_{z_2(i)z_2(j),2})^2 - 2 A_{ijt} (\delta_{w_1(i)w_1(j),1})  \nonumber \\
&& \hspace{9 cm}+ 2A_{ijt}(\lambda_{z_2(i)z_2(j),2})) \bigg] \nonumber \\
&=& (\tau-b) \sum_{u,v=1}^{K} (s_{u,w_1}s_{v,w_1} (\delta_{uv,1})^2 - s_{u,z_2}s_{v,z_2}(\lambda_{uv,2})^2 -2\mu_{uv,2}\delta_{uv,1} + 2\mu_{uv,1}\lambda_{uv,2} ).  \nonumber 
\end{eqnarray}
Therefore,
\begin{eqnarray}
E^{*}(B(b)) &=& B_1(b) + B_2(b)  \label{eqn: compb4}
\end{eqnarray}
where
\begin{eqnarray}
B_1(b) &=& (\tau-b) \sum_{u,v=1}^{K} \bigg[ s_{u,w_1}s_{v,w_1} V^{*}(\delta_{uv,1}) - s_{u,z_2}s_{v,z_2}V^{*}(\lambda_{uv,2})  \nonumber \\
&& \hspace{2 cm}- 2s_{u,w_1}s_{v,w_1} \text{Cov}^{*}(\mu_{uv,2},\delta_{uv,1})+ 2s_{u,z_2}s_{v,z_2}\text{Cov}^{*}(\mu_{uv,1},\lambda_{uv,2}) \bigg], \nonumber \\
B_2(b) &=& (\tau-b) \sum_{u,v=1}^{K} \bigg[s_{u,w_1}s_{v,w_1} (E^{*}(\delta_{uv,1}))^2 - s_{u,z_2}s_{v,z_2}(E^{*}(\lambda_{uv,2}))^2 \nonumber \\
&& \hspace{2 cm}- 2s_{u,w_1}s_{v,w_1} E^{*}(\mu_{uv,2})E^{*}(\delta_{uv,1}) + 2s_{u,z_2}s_{v,z_2}E^{*}(\mu_{uv,1})E^{*}(\lambda_{uv,2}) \bigg] \nonumber \\
 &=& (\tau-b) (B_{21} + B_{22} + B_{23} + B_{24}). \label{eqn: compb3}
\end{eqnarray}
Now, $V^{*}(\lambda_{uv,2})$ is given in (\ref{eqn: e2}) and 
\begin{eqnarray}
 V^{*}(\delta_{uv,1}) &=& \frac{1}{(n(1-b))^2} \bigg[ \frac{n(\tau-b)}{s_{u,w_1}s_{v,w_1}} \sum_{a,b=1}^{K} S((u,v,w_1),(a,b,z))\lambda_{ab}(1-\lambda_{ab})  \nonumber \\
 && \hspace{2 cm}+ \frac{n(1-\tau)}{s_{u,w_1}s_{v,w_1}} \sum_{a,b=1}^{K} S((u,v,w_1),(a,b,w))\delta_{ab}(1-\delta_{ab})  \bigg], \nonumber \\
 \text{Cov}^{*}(\mu_{uv,2},\delta_{uv,1}) &=& \frac{1}{n^2(\tau-b)(1-b)} \bigg[ \frac{n(\tau-b)}{s_{u,w_1}s_{v,w_1}} \sum_{a,b=1}^{K} S((u,v,w_1),(a,b,z))\lambda_{ab}(1-\lambda_{ab})\bigg], \nonumber \\
 \text{Cov}^{*}(\mu_{uv,1},\lambda_{uv,2}) &=& \frac{1}{n^2(\tau-b)\tau} \bigg[ \frac{n(\tau-b)}{s_{u,z_2}s_{v,z_2}} \sum_{a,b=1}^{K} S((u,v,z_2),(a,b,z))\lambda_{ab}(1-\lambda_{ab})\bigg]. \nonumber  
\end{eqnarray}
Using similar calculations as in (\ref{eqn: com1}), we obtain
\begin{eqnarray}
B_1(b) \geq -C(\tau-b)\left( \frac{K^2}{n}\rho_{m,n} + \mathcal{M}_{b,n,m}^2\rho_{m,n}^2 \right). \label{eqn: compb5}
\end{eqnarray}
Next,  consider $B_2(b)$. 
\begin{eqnarray}
B_{21} & = & \sum_{u,v=1}^{K} (s_{u,w_1}s_{v,w_1}-s_{u,w}s_{v,w}) (E^{*}(\delta_{uv,1}))^2   \nonumber \\
&& \hspace{1 cm} +  \sum_{u,v=1}^{K} s_{u,w}s_{v,w} ((E^*(\delta_{uv,1}))^2 - (E^*(\delta_{uv,w}))^2) + \sum_{u,v=1}^{K} s_{u,w}s_{v,w} (E^*(\delta_{uv,w}))^2 \nonumber \\
& \geq & -C\mathcal{M}_{b,n,m}^2 - C\sum_{u,v=1}^{K} s_{u,w}s_{v,w} |(E^*(\delta_{uv,1})) - (E^*(\delta_{uv,w}))| + \sum_{u,v=1}^{K}s_{u,w}s_{v,w} \delta_{uv}^2 \nonumber \\
& \geq & -C\mathcal{M}_{b,n,m}^2\rho_{m,n}^2  - B_{211} + \sum_{u,v=1}^{K}s_{u,w}s_{v,w} \delta_{uv}^2. \label{eqn: compb1}
\end{eqnarray}
We then get
\begin{eqnarray}
B_{211} = && \sum_{u,v=1}^{K} s_{u,w}s_{v,w} |(E^*(\delta_{uv,1})) - (E^*(\delta_{uv,w}))| \nonumber \\
& \leq & \sum_{u,v=1}^{K} s_{u,w}s_{v,w} \bigg | \frac{1}{n(1-b)} \bigg[ n(\tau-b) \frac{1}{s_{u,w_1}s_{v,w_1}} \sum_{a,b=1}^{K} S((u,v,w_1),(a,b,z))\lambda_{ab} \nonumber \\
&& + n(1-\tau) \frac{1}{s_{u,w_1}s_{v,w_1}}\sum_{a,b=1}^{K} S((u,v,w_1),(a,b,w))\delta_{ab} \bigg] -  \delta_{uv} \bigg | \nonumber \\
& \leq & C\frac{\tau-b}{1-b} \bigg| \sum_{u,v=1}^{K} \bigg[\frac{s_{u,w}s_{v,w}}{s_{u,w_1}s_{v,w_1}} \sum_{a,b=1}^{K} S((u,v,w_1),(a,b,z))\lambda_{ab} - s_{u,w}s_{v,w} \delta_{uv}\bigg] \bigg| \nonumber \\
&& +  C\frac{1-\tau}{1-b} \bigg| \sum_{u,v=1}^{K} \bigg[\frac{s_{u,w}s_{v,w}}{s_{u,w_1}s_{v,w_1}} \sum_{a,b=1}^{K} S((u,v,w_1),(a,b,w))\delta_{ab} - s_{u,w}s_{v,w} \delta_{uv}\bigg] \bigg| \nonumber \\
& = & B_{211a} + B_{211b},\ \ \text{say}. \label{eqn: compb2}
\end{eqnarray}
It is easy to see that 
\begin{eqnarray}
B_{211a} &\leq & C\rho_{m,n}^2 \sum_{u,v=1}^{K} \bigg|\frac{s_{u,w}s_{v,w}}{s_{u,w_1}s_{v,w_1}}-1\bigg| s_{u,w_1}s_{v,w_1} \nonumber \\
&& \hspace{0.5 cm}+ C\rho_{m,n}^2 \sum_{u,v=1}^{K}\sum_{a,b=1}^{K} |S((u,v,w_1),(a,b,z))-S((u,v,w),(a,b,z))|    \nonumber \\
&\leq & C\mathcal{M}_{b,n,m}^2\rho_{m,n}^2. \nonumber 
\end{eqnarray}
Similarly, $B_{211b} \leq C\mathcal{M}_{b,n,m}^2\rho_{m,n}^2$. Thus, by (\ref{eqn: compb1}) and (\ref{eqn: compb2}), we get
\begin{eqnarray}
B_{21} \geq -C\mathcal{M}_{b,n,m}^2\rho_{m,n}^2  +\sum_{u,v=1}^{K}s_{u,w}s_{v,w}\delta_{uv}^2.  \nonumber 
\end{eqnarray}
Using similar arguments as above, we also have
\begin{eqnarray}
B_{22} &\geq & -C\mathcal{M}_{b,n,m}^2 \rho_{m,n}^2 -\sum_{u,v=1}^{K}s_{u,z}s_{v,z}\lambda_{uv}^2,  \nonumber \\
B_{23} &\geq & -C\mathcal{M}_{b,n,m}^2 \rho_{m,n}^2 -  2\sum_{u,v=1}^{K} s_{u,w}s_{v,w}\delta_{uv} \sum_{a,b=1}^{K}S((u,v,w), (a,b,z))\lambda_{ab} \nonumber \\
& \geq & -C\mathcal{M}_{b,n,m}^2\rho_{m,n}^2  - 2 \sum_{i,j=1}^{m} \lambda_{z(i)z(j)}\delta_{w(i)w(j)}, \nonumber \\
B_{24} & \geq & -C\mathcal{M}_{b,n,m}^2\rho_{m,n}^2 + 2\sum_{u,v=1}^{K}s_{u,z}s_{v,z}\lambda_{uv}^2. \nonumber
\end{eqnarray}
Hence, by (\ref{eqn: compb3})
\begin{eqnarray}
B_2(b) \geq -C(\tau-b)\mathcal{M}_{b,n,m}^2\rho_{m,n}^2  + C(\tau-b)||\text{Ed}_{z}(\Lambda) -\text{Ed}_{w}(\Delta) ||_F^2. \nonumber 
\end{eqnarray}
Consequently, by (\ref{eqn: compb4}) and (\ref{eqn: compb5}), we have
\begin{eqnarray}
E^{*}(B(b)) &\geq & -C(\tau-b)\frac{K^2}{n}\rho_{m,n} -C(\tau-b)\mathcal{M}_{b,n,m}^2 \rho_{m,n}^2 \nonumber \\
&& \hspace{1.5 cm} + C(\tau-b)||\text{Ed}_{z}(\Lambda) -\text{Ed}_{w}(\Delta) ||_F^2. \label{eqn: compbb}
\end{eqnarray}
Recall $D(b)$ in (\ref{eqn: abd1}). Similar arguments as above also lead us to conclude
\begin{eqnarray}
E^*(D(b)) \geq -C(\tau-b)\frac{K^2}{n}\rho_{m,n} -C(\tau-b)n^2\mathcal{M}_{b,n,m}^2\rho_{m,n}^2 . \label{eqn: compd}
\end{eqnarray}
Hence by (\ref{eqn: abd1}), (\ref{eqn: compa}), (\ref{eqn: compbb}) and (\ref{eqn: compd}), we have 
\begin{eqnarray}
E^* (L_1-L_2) &\geq & -C(\tau-b)\frac{K^2}{n}\rho_{m,n} -C(\tau-b)n^2\mathcal{M}_{b,n,m}^2 \rho_{m,n}^2 \nonumber \\
&& \hspace{4 cm} + C(\tau-b)||\text{Ed}_{z}(\Lambda) -\text{Ed}_{w}(\Delta) ||_F^2 \nonumber \\
&\geq & -C(\tau-b)||\text{Ed}_{z}(\Lambda) -\text{Ed}_{w}(\Delta) ||_F^2, \ \ \text{by SNR-DSBM and (A1)}. \label{eqn: compe}
\end{eqnarray}
\noindent This proves (\ref{eqn: lsecon1f}). 
\vskip 2pt
Next, we compute variances. By (\ref{eqn: abd1}), 
\begin{eqnarray}
V^{*}(L_1-L_2) = V^{*}(A(b)) + V^{*}(B(b)) + V^{*}(D(b)). \nonumber
\end{eqnarray}
We only show the computation for $V^{*}(A(b))$. Other terms can be handled similarly.
\begin{eqnarray}
V^* (A(b)) & \leq & C \sum_{u,v=1}^{K} s_{u,z_2}s_{v,z_2}V^* (\lambda_{uv,2})^2 + C \sum_{u,v=1}^{K} s_{u,z_1}s_{v,z_1}V^* (\lambda_{uv,1})^2 \nonumber \\
&& \hspace{3 cm}+ 2C\sum_{u,v=1}^{K} s_{u,z_2}s_{v,z_2}V^{*}(\lambda_{uv,2}\lambda_{uv,3}). \label{eqn: vara1}
\end{eqnarray}
Let $\text{Cum}_r(X)$ denote the $r$-th order cumulant of $X$.   Then, 
\begin{eqnarray}
V^* (\lambda_{uv,2})^2 & \leq & \frac{C}{n^4 (s_{u,z_2}s_{v,z_2})^4}E \bigg[\sum_{t=1}^{n\tau} \sum_{\stackrel{i: z_2(i)=u}{j: z_2(j)=v}} (A_{ijt} -EA_{ijt}) \bigg]^4  \nonumber \\
& \leq & \frac{C}{n^4 (s_{u,z_2}s_{v,z_2})^4} \bigg[ \sum_{t=1}^{n\tau} \sum_{\stackrel{i: z_2(i)=u}{j: z_2(j)=v}} \text{Cum}_4(A_{ijt} -EA_{ijt})  \nonumber \\
&& + \bigg( \sum_{t=1}^{n\tau} \sum_{\stackrel{i: z_2(i)=u}{j: z_2(j)=v}} \text{Cum}_4(A_{ijt} -EA_{ijt})\bigg)\bigg( \sum_{t=1}^{n\tau} \sum_{\stackrel{i: z_2(i)=u}{j: z_2(j)=v}} \text{Cum}_4(A_{ijt} -EA_{ijt})\bigg)\bigg] \nonumber \\
& \leq & \frac{C \rho_{m,n}}{n^2(s_{u,z_2}s_{v,z_2})^2}. \nonumber 
\end{eqnarray}
Similarly,  $V^* (\lambda_{uv,1}) \leq  \frac{C \rho_{m,n}}{n^2(s_{u,z_1}s_{v,z_1})^2}$ and  $V^{*}(\lambda_{uv,2}\lambda_{uv,3}) \leq  \frac{C \rho_{m,n}}{n^2(s_{u,z_2}s_{v,z_2})^2}$. Hence,
\begin{eqnarray}
V^{*}(A(b)) \leq C\rho_{m,n}\frac{K^2}{n^2} \leq C(\tau-b)\rho_{m,n}\frac{K^2}{n}.\nonumber 
\end{eqnarray}
Using similar arguments as above, we also have
\begin{eqnarray}
V^{*}(B(b)), V^{*}(D(b)) \leq C\rho_{m,n}\frac{K^2}{n^2} \leq C(\tau-b)\rho_{m,n}\frac{K^2}{n}.\nonumber 
\end{eqnarray}
Hence,
\begin{eqnarray}
V^{*}(L_1-L_2)  \leq C\frac{K^2}{n^2}\rho_{m,n} \leq C(\tau-b)\rho_{m,n}\frac{||\text{Ed}_{z}(\Lambda)-\text{Ed}_{w}(\Delta)||_F^2}{n}.\label{eqn: Varcal1}
\end{eqnarray}
\noindent This proves (\ref{eqn: lsecon2final}). 
\vskip 2pt
\noindent Therefore, by Lemma \ref{lem: wvan1} the proof of Theorem \ref{lem: b1} is complete. $\blacksquare$

\begin{remark} \label{rem: proof}
Following the proof of Theorem \ref{lem: b1} (see (\ref{eqn: compe})), it is easy to see that
Assumption (A9) $n^2\mathcal{M}_{b,n,m}^2\rho_{m,n}^2  ||\text{Ed}_{z}(\Lambda) -\text{Ed}_{w}(\Delta)||_F^{-2}  \to 0\ \forall b\in (c^*,1-c^*)$ on the misclassification rate due to clustering is required for achieving consistency of $\tilde{\tilde{\tau}}_{m,n}$. The rate of $\mathcal{M}_{b,n,m}$ varies for different clustering procedures.  For Clustering Algorithm I presented in Section \ref{sec: DSBM}, the  rate of $\mathcal{M}_{b,n,m}^2$ is given in Theorem \ref{thm: mismiscluster} and hence (A9) reduces to (A1).  Details are given before stating (A1). Two variants together with Assumption (A9) and their 
corresponding misclassification error rates have also been presented and discussed in Section \ref{sec:discuss}. 

Note that Assumption (A9) is needed when used in conjunction with the every time point algorithm. However, the assumption can be weakened 
if we only cluster the nodes once before and after the change pont. Note that we assume that the change-point lies in the interval 
$ (c^*,1-c^*)$, which implies that we can cluster nodes using all time points before $c$ and obtain $z$ and similarly cluster nodes using
all time points after ($1-c^*$ to obtain $w$.  Then, $A_2(b) = 0$ and as a consequence 
$E^*(A(b)) \geq -C(\tau-b) (\frac{K^2}{n}\rho_{m,n}+\mathcal{M}_{b,n,m}^2\rho_{m,n}^2 )$ holds which is a sharper lower bound for $E^*(A(b))$ than the one provided
in (\ref{eqn: compa}). Analogously, for $w$ we get $E^*(D(b)) \geq -C(\tau-b) (\frac{K^2}{n}\rho_{m,n}+\mathcal{M}_{b,n,m}^2\rho_{m,n}^2 )$,
This provides $E^* (L_1-L_2) \geq  -C(\tau-b)\frac{K^2}{n}\rho_{m,n} -C(\tau-b)\mathcal{M}_{b,n,m}^2\rho_{m,n}^2  + C(\tau-b)||\text{Ed}_{z}(\Lambda) -\text{Ed}_{w}(\Delta) ||_F^2$  and a weaker version  (A9*) is needed [$\mathcal{M}_{b,n,m}^2\rho_{m,n}^2  ||\text{Ed}_{z}(\Lambda) -\text{Ed}_{w}(\Delta)||_F^{-2}  \to 0\ \forall \ b\in (c^*,1-c^*)$ (instead of (A9))] along with SNR-DSBM to establish (\ref{eqn: lsecon1f}).
\end{remark}

\subsection{Proof of Theorem \ref{thm: c3}} \label{subsec: c3}
Next, we focus on establishing the convergence rate for $\tilde{\tilde{\Lambda}}$, while analogous arguments are applicable for $\tilde{\tilde{\Delta}}$.  
\vskip 2pt
\noindent Without loss of generality, assume $\tilde{\tilde{\tau}}_{m,n} > \tau_{m,n}$.  For some clustering function $f$ and $b \in (c^{*},1-c^{*})$, recall that $\tilde{\lambda}_{uv,f,(b,n),m} = \frac{1}{nb}\sum_{t=1}^{nb} \frac{1}{s_{u,f}s_{v,f}} \sum_{\stackrel{f(i)=u}{f(j)=v}} A_{ij,(t,n)}$. 

\noindent For some $C>0$, we have
\begin{eqnarray}
||\text{Ed}_{\tilde{\tilde{z}}}(\tilde{\tilde{\Lambda}}) - \text{Ed}_{z}(\Lambda)||_F^2 &=& \sum_{i,j=1}^{m}(\tilde{\lambda}_{\tilde{\tilde{z}}(i)\tilde{\tilde{z}}(j),\tilde{\tilde{z}},(\tilde{\tilde{\tau}}_{m,n},n),m} - \lambda_{z(i)z(j)})^2 \nonumber \\
& \leq & 2\sum_{i,j=1}^{m}(\tilde{\lambda}_{\tilde{\tilde{z}}(i)\tilde{\tilde{z}}(j),\tilde{\tilde{z}},(\tilde{\tilde{\tau}}_{m,n},n),m} -\tilde{\lambda}_{\tilde{\tilde{z}}(i)\tilde{\tilde{z}}(j),\tilde{\tilde{z}},(\tau_{m,n},n),m})^2 \nonumber \\
&& + 2\sum_{i,j=1}^{m} (\tilde{\lambda}_{\tilde{\tilde{z}}(i)\tilde{\tilde{z}}(j),\tilde{\tilde{z}},(\tau_{m,n},n),m} - \lambda_{z(i)z(j)})^2 \nonumber \\
& \leq & Cm^2(\hat{\hat{\tau}}_{m,n} - \tau_{m,n})^2 \nonumber \\
&& \hspace{1 cm}+ C\sum_{i,j=1}^{m} (\tilde{\lambda}_{\tilde{\tilde{z}}(i)\tilde{\tilde{z}}(j),\tilde{\tilde{z}},(\tau_{m,n},n),m} - E^{*}\tilde{\lambda}_{\tilde{\tilde{z}}(i)\tilde{\tilde{z}}(j),\tilde{\tilde{z}},(\tau_{m,n},n),m})^2 \nonumber \\
&& \hspace{3.5 cm} C\sum_{i,j=1}^{m} (E^{*}\tilde{\lambda}_{\tilde{\tilde{z}}(i)\tilde{\tilde{z}}(j),\tilde{\tilde{z}},(\tau_{m,n},n),m}-\lambda_{z(i)z(j)})^2 \nonumber \\
&=& \mathbb{T}_1 + \mathbb{T}_2 + \mathbb{T}_3\ \ \text{(say)}. \nonumber 
\end{eqnarray}
Note that by Theorem \ref{lem: b1} 
we have $m^{-2}\mathbb{T}_1 = O_{\text{P}}(I(n>1)n^{-2}||\text{Ed}_{z}(\Lambda) -\text{Ed}_{w}(\Delta) ||_F^{-4}\rho_{m,n}^2)$. 

\noindent Let $P^{*}(\cdot) = P(\cdot|\tilde{\tilde{z}},\tilde{\tilde{w}})$.  By the sub-Gaussian property of Bernoulli random variables and since for some positive sequence $\{\tilde{C}_{m,n}\}$, $\hat{\mathcal{S}}_{m,n} \geq \tilde{C}_{m,n}\ \forall \ n$ with probability $1$, we get
\begin{eqnarray}
P^{*}(m^{-2}\mathbb{T}_2 \geq t) &=& P^{*}\left(\frac{1}{m^2}\sum_{i,j=1}^{m} (\tilde{\lambda}_{\tilde{\tilde{z}}(i)\tilde{\tilde{z}}(j),\tilde{\tilde{z}},(\tau_{m,n},n),m} - E^{*}\tilde{\lambda}_{\tilde{\tilde{z}}(i)\tilde{\tilde{z}}(j),\tilde{\tilde{z}},(\tau_{m,n},n),m})^2 \geq t \right)  \nonumber \\
& \leq & \sum_{i,j=1}^{m} P^{*} \left(|\tilde{\lambda}_{\tilde{\tilde{z}}(i)\tilde{\tilde{z}}(j),\tilde{\tilde{z}},(\tau_{m,n},n),m} - E^{*}\tilde{\lambda}_{\tilde{\tilde{z}}(i)\tilde{\tilde{z}}(j),\tilde{\tilde{z}},(\tau_{m,n},n),m}| \geq C\sqrt{t} \right) \nonumber \\
& \leq & m^2 C_1 e^{-C_2 n \hat{\mathcal{S}}_{m,n}^2 t} \leq m^2 C_1 e^{-C_2 n \tilde{C}_{m,n}^2 t}. \nonumber
\end{eqnarray}
Therefore, $P\left(m^{-2}\mathbb{T}_2 \geq t \right)  \leq  m^2 C_1 e^{-C_2 n \tilde{C}_{m,n} t} \to 0$ for $t = \frac{\log m}{n \tilde{C}_{m,n}^2} \rho_{m,n}$.  Hence, $m^{-2} \mathbb{T}_2 = O_{\text{P}} \left( \frac{\log m}{n \tilde{C}_{m,n}^2} \rho_{m,n} \right)$. 
Finally, 
\begin{eqnarray}
m^{-2}\mathbb{T}_3 &=& \frac{1}{m^2}\sum_{i,j=1}^{m} (E^{*}\tilde{\lambda}_{\tilde{\tilde{z}}(i)\tilde{\tilde{z}}(j),\tilde{\tilde{z}},(\tau_{m,n},n),m}-\lambda_{z(i)z(j)})^2 \nonumber \\
&=& \frac{1}{m^2}\sum_{i,j=1}^{m} \bigg( \frac{1}{s_{\tilde{\tilde{z}}(i),\tilde{\tilde{z}}(j)}s_{\tilde{\tilde{z}}(j),\tilde{\tilde{z}}}} \sum_{a,b=1}^{K} S((\tilde{\tilde{z}}(i),\tilde{\tilde{z}}(j),\tilde{\tilde{z}}),(a,b,z)) (\lambda_{ab} -\lambda_{z(i)z(j)})\bigg)^2 \nonumber \\
& \leq &  C\mathcal{M}_{\tilde{\tilde{\tau}}_{m,n},n,m}^2  = O_{\text{P}} \left( \left(\frac{Km}{n\nu_{m,n}^2}\right)^2\right).  \nonumber 
\end{eqnarray}
Thus, combining the convergence rate of $\mathbb{T}_1$, $\mathbb{T}_2$ and $\mathbb{T}_3$ derived above, establishes the convergence rate of $\text{Ed}_z(\tilde{\tilde{\Lambda}})$ when $\tilde{\tilde{\tau}}_{m,n} > \tau_{m,n}$. Similar arguments work for $\tilde{\tilde{\tau}}_{m,n} \leq \tau_{m,n}$. 

\noindent This completes the proof of Theorem \ref{thm: c3}.  $\blacksquare$

\subsection{More on Remark \ref{rem: otherclustering}} \label{sec:discuss}

A key step in using the ``all time point clustering" algorithm involves clustering. A specific clustering procedure proposed in \cite{B2017} 
was used to identify the communities and for locating the change-point in Section \ref{sec: DSBM}. Nevertheless, other clustering algorithms proposed in the literature  [\cite{P2017,RCY2011}] could be employed. For any given clustering algorithms,  (a) and (b) in Remark \ref{rem: otherclustering} hold. 

However, for (A9) to hold, the corresponding misclassification rate in the dense regime needs to satisfy 
$\mathcal{M}_{b,n,m} n  ||\text{Ed}_{z}(\Lambda) -\text{Ed}_{w}(\Delta)||_F^{-1} \to 0$ for consistency of the estimators. Next, we elaborate on these alternative clustering algorithms.

\vskip 5pt
 \noindent \textbf{Clustering Algorithm II}.   Instead of doing a spectral decomposition of the average adjacency matrices $B_1$ and $B_2$, 
the spectral decompositon is applied to their corresponding Laplcian matrices.
 \vskip 5pt
 \noindent An appropriate modification of the Proof of Theorem $2.1$ in \cite{RCY2011} implies
\begin{eqnarray} \label{eqn: c2rem}
\mathcal{M}_{b,n,m} = O_{\text{P}} \left(\frac{P_{m,n}}{\xi_{K_{m,n}}^4} \left(\frac{(\log m)^2}{nm} + m^2|\tau_{m,n}-b| ||\text{Ed}_{z}(\Lambda) -\text{Ed}_{w}(\Delta)||_F^{2} \right)\right)\,, 
\end{eqnarray} 
where $P_{m,n} = \max \{s_{u,z}, s_{u,w}: u=1,2,\ldots, K_{m,n}\}$ is the maximum community size and $\xi_{K_{m,n}}$ is the minimum between the $K_{m,n}$-th smallest  eigenvalue of the Laplacians of $B_1$ and $B_2$. A proof is given in Section \ref{subsec: remc2}.  Therefore, to satisfy $\mathcal{M}_{b,n,m} n  ||\text{Ed}_{z}(\Lambda) -\text{Ed}_{w}(\Delta)||_F^{-1} \to 0$ for the above spectral clustering, we need $\frac{n^{1/2} P_{m,n} (\log m)^2}{\xi_{K_{m,n}}^4 m K_{m,n}} = O(1)$ and $\frac{P_{m,n} m^3 n}{\xi_{K_{m,n}}^4} \to 0$. However, the latter condition seems excessively stringent in practical settings. For example, suppose $\Lambda = (p_1-q_1)I_{K_{m,n}} + q_1J_{K_{m,n}}$ and $\Delta = (p_2-q_2)I_{K_{m,n}} + q_2J_{K_{m,n}}$ where $I_{K_{m,n}}$ is the identity matrix of order ${K_{m,n}}$ and $J_{K_{m,n}}$ is the ${K_{m,n}} \times {K_{m,n}}$ matrix whose entries all equal $1$.   Further, suppose $0 < C < p_1, q_1, p_2, q_2 < 1-C<1$ and that the communities are of equal size. Then, 
$P_n = O(m/{K_{m,n}})$.   Moreover, \cite{RCY2011} established that $\xi_{{K_{m,n}}} = O({K_{m,n}}^{-1})$.  Hence, $\frac{n^{1/2} P_{m,n} (\log m)^2}{\xi_{K_{m,n}}^4 m K_{m,n}} = O(\sqrt{n}K_{m,n}^2 (\log m)^2) \to \infty$ and $\frac{P_{m,n} m^3 n}{\xi_{K_{m,n}}^4} = O(m^4 n {K_{m,n}}^3) \to \infty$. On the other hand, as we have seen in Example \ref{example: misclassnew},  (A1) is satisfied for this example.
\vskip 5pt

\noindent \textbf{Clustering Algorithm III}. In this case, the following modification of \cite{RCY2011}'s algorithm for community detection is employed as follows.  Define
\begin{eqnarray} \label{eqn: deglapmatrix}
D_{i,(t,n)} &=& \sum_{j=1}^{m} A_{ij,(t,n)}, \ \ \ 
D_{(t,n)} = \text{Diag}\{D_{i,(t,n)}: 1 \leq i \leq n\},\ \ \\  
L_{(t,n)} &=& D_{(t,n)}^{-1/2} A_{t,n} D_{(t,n)}^{-1/2},\ \ \ 
L_{\Lambda, (b,n)} = \frac{1}{nb}\sum_{t=1}^{nb}L_{(t,n)},\ \ L_{\Delta, (b,n)} = \frac{1}{n(1-b)}\sum_{t=nb+1}^{n} L_{(t,n)}. \nonumber
\end{eqnarray}
Note that $I - L_{(t,n)}$ is the Laplacian of $A_{t,n}$. 
Next, run the spectral clustering algorithm introduced in \cite{RCY2011} after replacing $L$ respectively by $L_{\Lambda, (b,n)}$ and $L_{\Delta, (b,n)}$ for estimating $z$,  $w$. 
\vskip 5pt
\noindent  In this case,
\begin{eqnarray} \label{eqn: remc3}
\mathcal{M}_{b,n,m} &=& O_{\text{P}}\bigg(\frac{P_{m,n}}{\xi_{K_{m,n}}^4} \bigg(\frac{(\log m\sqrt{n})^2}{\sqrt{n}m} + m^2|\tau_{m,n}-b| ||\text{Ed}_{z}(\Lambda) -\text{Ed}_{w}(\Delta)||_F^{2} \nonumber \\
&& \hspace{6 cm}  + |\tau_{m,n}-b| \frac{(\log m)^2}{m}\bigg)\bigg),
\end{eqnarray}
where $P_{m,n}$ and $\xi_{K_{m,n}}$ are as described after (\ref{eqn: c2rem}). A proof is given in Section \ref{subsec: remc3}. Therefore, to satisfy $\mathcal{M}_{b,n,m} n  ||\text{Ed}_{z}(\Lambda) -\text{Ed}_{w}(\Delta)||_F^{-1} \to 0$ for this variant of the spectral clustering algorithm, we require $\frac{n^{3/2} P_{m,n} (\log m)^2}{\xi_{K_{m,n}}^4 m K_{m,n}} = O(1)$, $\frac{n P_{m,n} (\log m\sqrt{n})^2}{\xi_{K_{m,n}}^4 m K_{m,n}} = O(1)$  and $\frac{P_{m,n} m^3 n}{\xi_{K_{m,n}}^4} \to 0$. However, these are much stronger conditions that the one required for Clustering Algorithm II.

The upshot of the previous discussion is that Clustering Algorithm I requires a milder assumption (A1) on the misclassification rate compared to  Clustering Algorithms II and III. This is the reason that the results established in Sections \ref{sec: DSBM} and \ref{sec: 2step} leverage the former algorithm.

\subsubsection{Justification of (\ref{eqn: c2rem})} \label{subsec: remc2}
Let $\mathcal{L}(A)$ denote the Laplacian of $A$. Also without loss of generality, assume $b >\tau$.  Using similar arguments as in Appendix B, C and D of \cite{RCY2011}, we can easily show that for some $C > 0$ and with probability tending $1$, 
\begin{eqnarray}
\mathcal{M}_{b,n,m} &\leq & C\frac{P_{m,n}}{\xi_{K_{m,n}}^4} ||(\mathcal{L}(\frac{1}{n}\sum_{t=1}^{nb}A_{t,n}))^2 - (\mathcal{L}(\text{Ed}_{z}(\Lambda)))^2||_F^2 \nonumber \\
& \leq & C\frac{P_{m,n}}{\xi_{K_{m,n}}^4}\bigg( ||(\mathcal{L}(\frac{1}{n}\sum_{t=1}^{n\tau}A_{t,n}))^2 - (\mathcal{L}(\text{Ed}_{z}(\Lambda)))^2||_F^2 \nonumber \\
&& \hspace{0.5 cm}+ ||(\mathcal{L}(\frac{1}{n}\sum_{t=n\tau+1}^{nb}A_{t,n}))^2 - (\mathcal{L}(\text{Ed}_{z}(\Delta)))^2||_F^2 \nonumber \\
&& \hspace{2 cm}+  |\tau-b|\ ||(\mathcal{L}(\text{Ed}_{z}(\Delta)))^2 - (\mathcal{L}(\text{Ed}_{z}(\Lambda)))^2 ||_F^2 \bigg)\nonumber \\
&=& C\frac{P_{m,n}}{\xi_{K_{m,n}}^4} (A_1 +A_2+A_3),\ \text{say}. 
\end{eqnarray}
Then,using similar arguments as in Lemma $A.1$ of \cite{RCY2011}, we  obtain
\begin{eqnarray}
A_1, A_2 &=& O_{\text{P}}(\frac{(\log m)^2}{mn}). \nonumber
\end{eqnarray}
Define,
\begin{eqnarray}
\mathcal{D}_{i,\Lambda} = \sum_{j=1}^{m}\lambda_{z(i)z(j)},\ \ \mathcal{D}_{\Lambda} = \text{Diag}\{\mathcal{D}_{i,\Lambda}:\ 1 \leq i \leq m\}, \nonumber \\
\mathcal{D}_{i,\Delta} = \sum_{j=1}^{m}\delta_{w(i)w(j)},\ \ \mathcal{D}_{\Delta} = \text{Diag}\{\mathcal{D}_{i,\Delta}:\ 1 \leq i \leq m\}. \nonumber
\end{eqnarray}
Then, 
\begin{eqnarray}
A_3 &\leq & Cm||\mathcal{L}(\text{Ed}_{z}(\Lambda)) - \mathcal{L}(\text{Ed}_{w}(\Delta))||_F^2 \nonumber \\
&\leq & Cm ||\mathcal{D}_{\Lambda}^{-1/2}\text{Ed}_{z}(\Lambda) \mathcal{D}_{\Lambda}^{-1/2}  - \mathcal{D}_{\Delta}^{-1/2} \text{Ed}_{z}(\Delta)\mathcal{D}_{\Delta}^{1/2} ||_F^2\nonumber \\
& \leq & Cm \bigg[ ||\text{Ed}_{z}(\Lambda) - \text{Ed}_{z}(\Delta)||_F^2 ||\mathcal{D}_{\Lambda}^{1/2}||_F^4 + 2 ||\mathcal{D}_{\Lambda}^{-1/2} - \mathcal{D}_{\Delta}^{-1/2}||_F^2||\text{Ed}_{z}(\Delta)||_F^2||\mathcal{D}_{\Delta}^{-1/2}||_F^2 \bigg] \nonumber \\
& \leq & Cm (||\text{Ed}_{z}(\Lambda) - \text{Ed}_{z}(\Delta)||_F^2 + \frac{C}{m} ||\text{Ed}_{z}(\Lambda) - \text{Ed}_{z}(\Delta)||_F^2 m^2) \nonumber \\
& \leq & Cm^2||\text{Ed}_{z}(\Lambda) - \text{Ed}_{z}(\Delta)||_F^2. \nonumber
\end{eqnarray}
Hence, 
\begin{eqnarray}
\mathcal{M}_{b,n,m} = O_{\text{P}}\left(\frac{P_{m,n}}{\xi_{K_{m,n}}^4} \left( \frac{(\log m)^2}{mn}  +|\tau-b| m^2 ||\text{Ed}_{z}(\Lambda) - \text{Ed}_{z}(\Delta)||_F^2\right) \right). \nonumber
\end{eqnarray}
This completes the  justification of (\ref{eqn: c2rem}). 

\subsubsection{Justification of (\ref{eqn: remc3})} \label{subsec: remc3}
Using similar arguments to those presented  in Section \ref{subsec: remc2},  with probability tending to $1$, we have
\begin{eqnarray}
\mathcal{M}_{b,n,m} &\leq & C \frac{P_{m,n}}{\xi_{K_{m,n}}^4} ||(L_{\Lambda,(b,n)})^2-(\mathcal{L}(\text{Ed}_{z}(\Lambda)))^2 ||_F^2 \nonumber \\
& \leq & C \frac{P_{m,n}}{\xi_{K_{m,n}}^4} \bigg[||(L_{\Lambda,(\tau,n)})^2-(\mathcal{L}(\text{Ed}_{z}(\Lambda)))^2 ||_F^2 + \frac{1}{n}\sum_{t=n\tau +1}^{nb} ||(\mathcal{L}(A_{t,n}))^2-(\mathcal{L}(\text{Ed}_{w}(\Delta)))^2 ||_F^2  \nonumber \\
& & \hspace{2 cm} +|\tau-b| ||(\mathcal{L}(\text{Ed}_{z}(\Delta)))^2 - (\mathcal{L}(\text{Ed}_{z}(\Lambda)))^2  ||_F^2\bigg] \nonumber \\
&\leq & C \frac{P_{m,n}}{\xi_{K_{m,n}}^4} (A_1+A_2 + |\tau - b| m^2 ||\text{Ed}_{z}(\Lambda) - \text{Ed}_{z}(\Delta)||_F^2), \ \text{say}. \nonumber 
\end{eqnarray}
Then, by Theorem $2.1$ in \cite{RCY2011}, we have
\begin{eqnarray}
A_1 = O_{\text{P}}(\frac{(\log m\sqrt{n})^2}{m\sqrt{n}})\ \ \text{and}\ \ A_2 = O_{\text{P}}(|\tau-b|\frac{(\log m)^2}{m}). \nonumber
\end{eqnarray}
Hence,
\begin{eqnarray} \nonumber
\mathcal{M}_{b,n,m} = O_{\text{P}}\left(\frac{P_{m,n}}{\xi_{K_{m,n}}^4} \left(\frac{(\log m\sqrt{n})^2}{\sqrt{n}m} + m^2|\tau-b| ||\text{Ed}_{z}(\Lambda) -\text{Ed}_{w}(\Delta)||_F^{2}  + |\tau-b| \frac{(\log m)^2}{m}\right)\right). 
\end{eqnarray}
This completes the justification of (\ref{eqn: remc3}). 

\subsection{Proof of Theorem \ref{thm: 2step}} \label{subsec: 2step}
To prove Theorem \ref{thm: 2step}, note that the proof of Theorem \ref{lem: b1} in Section \ref{subsec: b1} goes through once we use
$\mathcal{M}_{b,n,m} =0$, $z_1=z_2=z$,  $w_1=w_2=w$ and $K=m$.  In this case, (\ref{eqn: compe}) and (\ref{eqn: Varcal1}) implies 
\begin{eqnarray}
E^{*}(L_1-L_2) & \geq & - C(\tau-b)\frac{m^2 \rho_{m,n}}{n} + C(\tau-b)||\text{Ed}_{z}(\Lambda) - \text{Ed}_{w}(\Delta)||_F^2, \nonumber \\
V^{*}(L_1-L_2)  &\leq &  C(\tau-b) \rho_{m,n} ||\text{Ed}_{z}(\Lambda) - \text{Ed}_{w}(\Delta)||_F^2. 
\nonumber
\end{eqnarray}
Therefore, by SNR-ER and Lemma \ref{lem: wvan1}, Theorem \ref{thm: 2step} follows. $\blacksquare$

\subsection{More on Remark \ref{rem: further}}

In this section, we focus on the dense network regime. A similar discussion is applicable for the sparse regime as well. As pointed out in Remark \ref{rem: further}, one may wonder regarding settings where SNR-DSBM holds, but neither (A1) nor SNR-ER do. The following Examples \ref{example: v10new1} and \ref{example: v10new2} introduce such settings in the context of changes in the connection probabilities and in the community structures, respectively.

\begin{example} \label{example: v10new1}
\textbf{(Change in connection probabilities)} Consider a DSBM where
\begin{eqnarray}
z=w\hspace{1 cm}\text{and}\hspace{1 cm} \Lambda = \Delta - \frac{1}{\sqrt{n}}. \label{eqn: v10new1}
\end{eqnarray}
  In this case $||\text{Ed}_{z}(\Lambda)-\text{Ed}_{w}(\Delta)||_F^2 = \frac{m^2}{n}$. Therefore, SNR-ER does not hold. However, SNR-DSBM holds if $K=o(m)$. Cases (a)-(c) presented below  provide settings where (A1) does not hold, but SNR-DSBM does. 
\vskip 2pt
\noindent (a) Consider the setting in Example \ref{example: misclassnew} with $K = Cm^{0.5-\delta}$ (that is $K=o(m)$) and $n = Cm^{4\delta}$ for some $C>0$ and $\delta < 1/6$. It can easily be seen that (A1) does not hold, but SNR-DSBM does. 
\vskip 2pt
\noindent (b) Suppose all assumptions in Example \ref{example: misclassnew2} hold, $K$ is finite and $m = Cn^{2\delta}$. In this case,  (A1) does not hold, but SNR-DSBM does.
\vskip 2pt
\noindent (c) Finally, consider the setup in Example \ref{example: misclass1new} with $K = o(m)$,  $m_{\min} = Cm^{\delta}$, $n = m^{\lambda}$ for some $\lambda >0$,  $\delta \in [0,1]$ and $-\lambda/2 \leq 2\delta - 1 < \lambda/2$. The same conclusion on (A1) failing to hold, while SNR-DSBM holding is
reached. 
\vskip 2pt
\noindent Therefore, in each of the (a)-(c) cases, together with (\ref{eqn: v10new1}) do not satisfy (A1) and SNR-ER,  whereas  SNR-DSBM holds.
\end{example}

\begin{example} \label{example: v10new2}
\textbf{(Change in communities)} Consider a DSBM where for $0<p<1$, 
\begin{eqnarray}
K=2, && z(i) = \begin{cases}
1\ \ \text{if $i$ is odd} \\
2\ \ \text{if $i$ is even},
\end{cases}
\ \ w(i) = \begin{cases}
1\ \ \text{if $1 \leq i \leq [m/2]$} \\
2\ \ \text{if $[m/2]<i\leq m$},
\end{cases} \\
&&\Lambda = \Delta= \left(\begin{array}{cc}
p & p-\frac{1}{\sqrt{n}} \\ 
p-\frac{1}{\sqrt{n}} & p
\end{array} \right).  \nonumber 
\end{eqnarray}
This gives $||\text{Ed}_{z}(\Lambda)-\text{Ed}_{w}(\Delta)||_F^2 = \frac{m^2}{n}$. Hence, SNR-ER does not hold, but SNR-DSBM does. Also suppose $m = Cn^{\delta}$ for some $C>0$ and $\delta \in [1,1.5)$. In this case (A1) is not satisfied.
\end{example}

The methods discussed in Sections \ref{sec: DSBM} and \ref{sec: 2step} fail to detect the change-point under the above presented settings.
Therefore, alternative strategies not based on clustering and hence assumption (A1) need to be investigated.

One possibility for the case of a  single change-point being present was discussed in Remark \ref{rem: nc}. 
\begin{example} \label{example: v10new3}
As the true change-point $\tau_{m,n} \in (c^*,1-c^*)$, we can use $\tau^*_{m,n}$ to estimate $\tau_{m,n}$ and its consistency follows from SNR-DSBM and (A1*) $\frac{m}{\sqrt{n}\nu_{m,n}^2} = O(1)$ which is much weaker than (A1). As we have seen before, (A1) and SNR-ER do not hold in Examples \ref{example: v10new2} and \ref{example: v10new3} whereas SNR-DSBM is satisfied. 
Based on the discussion in Remark \ref{rem: nc}, it is easy to see that  (A1*)  holds for these examples. Therefore, for the settings posited in Examples \ref{example: v10new2} and \ref{example: v10new3}, $\tau^*_{m,n}$ estimates $\tau_{m,n}$ consistently. Nevertheless, as mentioned in Remark \ref{rem: nc}, this strategy is not easy to extend to a setting involving multiple change-points.
\end{example}

Another setting that does not require clustering is presented next and builds on the model discussed in \cite{G2015rate}. 

\begin{example} \label{example: v10new4}
Consider a DSBM with $K=2$ communities. Further, let  $B_{1z}$ and $B_{1w}$ be the blocks where node $1$ belongs to under $z$ and $w$, respectively, and let  $\Lambda = \left(\begin{array}{cc}
a_1 & d_1 \\ 
d_1 & a_1
\end{array} \right)$, $\Delta =  \left(\begin{array}{cc}
a_2 & d_2 \\ 
d_2 & a_2
\end{array} \right)$ with $0<c<a_1,a_2,d_1,d_2<1-c<1$, $a_1>d_1, a_2>d_2$,  $a_1-d_1 = a_2-d_2$ and the true change-point $\tau \in (c^*,1-c^*)$. Recall $\hat{p}_{ij,(b,n)}$ and $\hat{q}_{ij,(b,n)}$  from (\ref{eqn: estimatea1}).  Let $\gamma_j = \hat{p}_{11,(b,n)}-\hat{p}_{1j,(b,n)}$ and $\delta_j = \hat{q}_{11,(b,n)}-\hat{q}_{1j,(b,n)}$.  One can use the following algorithm to detect communities. Chose $B, B^*>0$ and $\delta \in (0,1)$ such that $\frac{B}{\sqrt{n^\delta}} \leq \frac{c^*}{1-c^*} (a_1-d_1)$.
\vskip2pt
\noindent 1. If $\gamma_j \leq \frac{B}{\sqrt{n^\delta}}$ and $\delta_j \leq \frac{B}{\sqrt{n^\delta}}$, then put node $j$  in $B_{1z}\cap B_{1w}$. 
\vskip 2pt
\noindent 2. If $\gamma_j \leq \frac{B}{\sqrt{n^\delta}}$ and $\delta_j > \frac{B}{\sqrt{n^\delta}}$, then  put node $j$  in $B_{1z}\cap B_{1w}^{c}$. 
\vskip 2pt
\noindent 3. If $\gamma_j > \frac{B}{\sqrt{n^\delta}}$ and $\delta_j \leq \frac{B}{\sqrt{n^\delta}}$, then put node $j$  in $B_{1z}^{c}\cap B_{1w}$. 
\vskip 2pt
\noindent 4. If $\gamma_j > \frac{B}{\sqrt{n^\delta}}$ and $\delta_j > \frac{B}{\sqrt{n^\delta}}$, then we need further investigation.
\vskip 2pt 
(4a) If $\frac{\gamma_j}{\delta_j} \leq 1-\frac{B^*}{\sqrt{n^\delta}}$, then put node $j$ in $B_{1z}\cap B_{1w}^{c}$.
\vskip 2pt
(4b) If $\frac{\gamma_j}{\delta_j} > 1+\frac{B^*}{\sqrt{n^\delta}}$, then put node $j$ in $B_{1z}^{c}\cap B_{1w}$.
\vskip 2pt
(4c) If $\frac{\gamma_j}{\delta_j} \in (1-\frac{B^*}{\sqrt{n^\delta}}, 1+ \frac{B^*}{\sqrt{n^\delta}})$, then put node $j$ in $B_{1z}^{c}\cap B_{1w}^{c}$.
\vskip2pt
\noindent In this algorithm, it is easy to see that $\text{P}(\text{no node is misclassifed}) \to 1$. Therefore, an alternative condition  (A9) is satisfied 
(see details about it in Section \ref{subsec: examplev10new4}) 
and $\tilde{\tilde{\tau}}_n$ estimates $\tau_n$ consistently. 
\end{example}
However, the setting in Example \ref{example: v10new4} is very specific involving two parameters only for each connection probability matrix),
which in turn allows one to use statistics based on the degree connectivity of each node and thus avoid using a clustering algorithm. 
Nevertheless, a generally applicable strategy is currently lacking for the regime where SNR-DSBM holds, but neither SNR-ER or (A1) do.
This constitutes an interesting direction of further research.

\subsubsection{justification of Example \ref{example: v10new4}} \label{subsec: examplev10new4}
It is easy to see that the following results (a)-(d) hold under the assumptions in Example \ref{example: v10new4}. 
\vskip 5pt
\noindent (a) $\gamma_j, \delta_j = O_{\text{P}}(\frac{1}{\sqrt{n}})$ when $j \in B_{1z}\cap B_{1w}$.
\vskip 5pt
\noindent (b) $\gamma_j - \frac{b-\tau}{b} (a_2-d_2),\ \  \delta_j - (a_2-d_2),\ \  \frac{\gamma_j}{\delta_j} - \frac{b-\tau}{b} = O_{\text{P}}(\frac{1}{\sqrt{n}})$ when $j \in B_{1z} \cap B_{1w}^{c}$. 
\vskip 5pt
\noindent (c) $\gamma_j - \frac{\tau}{b}(a_1-d_1), \ \ \delta_j = O_{\text{P}}(\frac{1}{\sqrt{n}})$ when $j \in B_{1z}^{c}\cap B_{1w}$.
\vskip 5pt
\noindent (d) $\gamma_j - (a_1-d_1),\ \ \delta_j - (a_1-d_1),\ \ \frac{\gamma_j}{\delta_j} -1 = O_{\text{P}}(\frac{1}{\sqrt{n}})$ when $j \in B_{1z}^c \cap B_{1w}^{c}$. 
\vskip 5pt
\noindent Using the above results, we have
\vskip 5pt
\noindent (a) $P(j\ \text{is classified in}\ B_{1z} \cap B_{2w}\ |\  j \in B_{1z} \cap B_{2w})\ \leq \ P(\gamma_j < \frac{B}{\sqrt{n^{\delta}}}, \delta_j < \frac{B}{\sqrt{n^{\delta}}}\ |\  j \in B_{1z} \cap B_{2w}) \to 1$.
\vskip 5pt
\noindent (b) $P(j\ \text{is classified in}\ B_{1z}^{c} \cap B_{2w}\ |\  j \in B_{1z}^{c} \cap B_{2w}) \leq P(\gamma_j > \frac{B}{\sqrt{n^{\delta}}}, \delta_j < \frac{B}{\sqrt{n^{\delta}}}\ |\  j \in B_{1z}^{c} \cap B_{2w}) +  P(\frac{\gamma_j}{\delta_j} > 1 + \frac{B^{*}}{\sqrt{n^\delta}}\ |\  j \in B_{1z}^{c} \cap B_{2w}) \to 1$.
\vskip 5pt 
\noindent (c) $P(j\ \text{is classified in}\ B_{1z} \cap B_{2w}^{c}\ |\  j \in B_{1z} \cap B_{2w}^{c}) \leq P(\gamma_j < \frac{B}{\sqrt{n^{\delta}}}, \delta_j > \frac{B}{\sqrt{n^{\delta}}}\ |\  j \in B_{1z} \cap B_{2w}^{c}) +  P(\frac{\gamma_j}{\delta_j} < 1 - \frac{B^{*}}{\sqrt{n^\delta}}\ |\  j \in B_{1z} \cap B_{2w}^{c}) \to 1$.
\vskip 5pt
\noindent (d) $P(j\ \text{is classified in}\ B_{1z}^{c} \cap B_{2w}^{c}\ |\  j \in B_{1z}^{c} \cap B_{2w}^{c}) \leq   P(\frac{\gamma_j}{\delta_j} \in  (1 - \frac{B^{*}}{\sqrt{n^\delta}}, 1 + \frac{B^{*}}{\sqrt{n^\delta}})\ |\  j \in B_{1z}^{c} \cap B_{2w}^{c}) \to 1$.
\vskip 5pt
\noindent These all together implies $P(\text{no node is misclassified}) \to 1$.

\subsection{Proof of Theorem \ref{lem: b2lse}} \label{subsec: b2lse}
Next, we prove Theorem \ref{lem: b2lse} for $\tilde{\tilde{\tau}}_{m,n}$.Note that the proof for $\hat{\tau}_{m,n}$ is much simpler, once we use $z_1=z_2=z$, $w_1=w_2=w$, $K=m$ and $\mathcal{M}_{b,n,m} =0$ in the following proof. 
\vskip 2pt
\noindent Suppose $||\text{Ed}_z(\lambda) -\text{Ed}_w(\Delta)||_F \to \infty$. Then by Theorem \ref{lem: b1}, it is easy to see that $P(\tilde{\tilde{\tau}}_{m,n} = \tau_{m,n}) \to 1$. 
\vskip 3pt
\noindent Lemma \ref{lem: wvandis1} from \cite{Wellner1996empirical} proves useful for establishing the asymptotic distribution of the change-point
 estimate, when $||\text{Ed}_{z}(\Lambda) -\text{Ed}_{w}(\Delta)||_F \to c \geq 0$.

\noindent Next, suppose $||\text{Ed}_z(\lambda) -\text{Ed}_w(\Delta)||_F \to c \geq 0$. Take $h = n|\tau -b|||\text{Ed}_z(\lambda) -\text{Ed}_w(\Delta)||_F ^2$.  

\noindent Recall the definitions of $A(b)$ and $D(b)$ from (\ref{eqn: abd1}). Using expectations in (\ref{eqn: compa}) and (\ref{eqn: compd}), it is easy to see that by SNR-DSBM, (A1) and as $||\text{Ed}_z(\lambda) -\text{Ed}_w(\Delta)||_F \to c \geq 0$, we have
\begin{eqnarray}
E\sup_{h \in \mathcal{C}}|nA(b)|,  E\sup_{h \in \mathcal{C}}|nD(b)| \leq  C\frac{K^2}{n||\text{Ed}_z(\lambda) -\text{Ed}_w(\Delta)||_F^2} + C \frac{n \mathbb{M}_{b,n,m}^2}{|\text{Ed}_z(\lambda) -\text{Ed}_w(\Delta)||_F^2} \to 0 \nonumber 
\end{eqnarray}
for some compact set $\mathcal{C} \subset \mathbb{R}$. 
\vskip 2pt
This establishes that if $||\text{Ed}_z(\lambda) -\text{Ed}_w(\Delta)||_F \to c \geq 0$ and SNR-DSBM and (A1) hold, then 
\begin{eqnarray} \label{eqn: distnt1t2}
\sup_{h \in \mathcal{C}}|nA(b)|,\ \  \sup_{h \in \mathcal{C}}|nD(b)| \stackrel{\text{P}}{\to} 0. 
\end{eqnarray}
\noindent Next, recall the definition of $B(b)$ from (\ref{eqn: abd1}).  Using similar arguments in Section \ref{subsec: b1}, it is easy to show that
\begin{eqnarray}
B(b) &=& \sum_{t=nb+1}^{n\tau} \sum_{i,j=1}^{m} (2A_{ijt} -2\lambda_{z(i)z(j)})(\lambda_{z(i)z(j)}-\delta_{w(i)w(j)}) \nonumber \\
&& \hspace{2 cm} + n|\tau-b| ||\text{Ed}_{z}(\Lambda) -\text{Ed}_{w}(\Delta)||_F^2 + R(b) \nonumber 
\end{eqnarray}
where $R(b) \leq C|\tau-b|n^2 \mathcal{M}_{b,n,m}^2$. Therefore, by (A1)
\begin{eqnarray} \label{eqn: t32t33}
\sup_{h \in \mathcal{C}}|R(b)| \stackrel{\text{P}}{\to} 0. 
\end{eqnarray}
\vskip 2pt 
\noindent Suppose $||\text{Ed}_{z}(\Lambda) -\text{Ed}_{w}(\Delta)||_F \to 0$.  Applying the Central Limit Theorem, it is easy to see that
\begin{eqnarray} \label{eqn: t31}
\sup_{h \in \mathcal{C}}|B(b)-R(b) + |h| + 2\gamma  B_h| \stackrel{\text{P}}{\to} 0. 
\end{eqnarray}
Thus, by  (\ref{eqn: distnt1t2})-(\ref{eqn: t31}) and Lemma \ref{lem: wvandis1}, 
\begin{eqnarray}
n ||\text{Ed}_{z}(\Lambda) - \text{Ed}_{w}(\Delta)||_F^2 (\tilde{\tilde{\tau}}_{m,n}-\tau_{m,n}) \stackrel{\mathcal{D}}{\to} \arg \min_{h \in \mathbb{R}} (|h|+2\gamma B_h) & \stackrel{\mathcal{D} } {=} & \arg \max_{h \in \mathbb{R}} (-0.5|h|+\gamma B_h) \nonumber \\
& \stackrel{\mathcal{D} } {=} & \gamma^2 \arg \max_{h \in \mathbb{R}} (-0.5|h|+ B_h).  \nonumber
\end{eqnarray}
This proves Part(b) of Theorem \ref{lem: b2lse}. 
\vskip 3pt
\noindent Suppose $||\text{Ed}_{z}(\Lambda) -\text{Ed}_{w}(\Delta)||_F \to c>0$.   Then,
\begin{eqnarray}
B(b) - R(b) &=& \sum_{i,j=1}^{m} \sum_{t=nb+1}^{n\tau} \bigg[- (A_{ij,(t,n)} - {\lambda}_{z(i)z(j)})^2 + (A_{ij,(t,n)} - \delta_{w(i)w(j)})^2\bigg] \nonumber \\
&=& \sum_{i,j \in \mathcal{K}_n} \sum_{t=nb+1}^{n\tau} \bigg[ - (A_{ij,(t,n)} - \lambda_{z(i)z(j)})^2 + (A_{ij,(t,n)} - \delta_{w(i)w(j)})^2\bigg] \nonumber \\
&& + \sum_{i,j\in \mathcal{K}_0} \sum_{t=nb+1}^{n\tau} \bigg[ -(A_{ij,(t,n)} -\lambda_{z(i)z(j)})^2 + (A_{ij,(t,n)} - \delta_{w(i)w(j)})^2\bigg] \nonumber \\
&=& T_{a} + T_{b} \ \ (\text{say}). \label{eqn: t311312}
\end{eqnarray}
By (A6) and (A7) and if $||\text{Ed}_{z}(\Lambda) -\text{Ed}_{w}(\Delta)||_F \to c>0$, we obtain
\begin{eqnarray}
\sup_{h \in \mathcal{C}} |T_{b} - A^{*}(h)| \stackrel{\text{P}}{\to} 0 \label{eqn: t312}
\end{eqnarray}
where each $h \in \mathbb{Z}$, 
$A^{*}(c^2(h+1)) - A^{*}(c^2h) = \sum_{k \in \mathcal{K}_0} \bigg[(Z_{ij, {h}} -a_{ij,1}^{*})^2 -(Z_{ij,{h}} -a_{ij,2}^{*})^2 \bigg]$
and $\{Z_{ij, {h}}\}$ are independently distributed with $Z_{ij, {h}} \stackrel{d}{=} A_{ij,1}^{*}I({h} < 0) + A_{ij,2}^{*}I({h} \geq  0)$ for all $(i,j) \in \mathcal{K}_0$. 

 \noindent Next, 
$T_{a} = 2\sum_{t=nb+1}^{n\tau}\sum_{i,j \in \mathcal{K}_n} (A_{ij,(t,n)}-\lambda_{z(i)z(j)})(\delta_{w(i)w(j)}-\lambda_{z(i)z(j)}) +|h|$. 
An application of the Central Limit Theorem together with (A4) and (A5) yields
\begin{eqnarray} \label{eqn: t311}
\sup_{h \in \mathcal{C}}|T_{a} - D^{*}(h)-C^{*}(h)| \stackrel{\text{P}}{\to} 0. 
\end{eqnarray}
where  for each ${h} \in \mathbb{Z}$, $D^{*} (c^2 (h+1))-D^{*}(c^2 h) = 0.5 {\rm{Sign}}(-h) c_1^2$ and $C^{*}(c^2(h+1)) - C^{*}(c^2 h) = \tilde{\gamma}_{_{\text{LSE}}} W_{{h}},\ \ W_{{h}} \stackrel{\text{i.i.d.}}{\sim} \mathcal{N}(0,1)$.

\noindent Therefore, by (\ref{eqn: distnt1t2}), (\ref{eqn: t32t33}),  (\ref{eqn: t312}), (\ref{eqn: t311}) and Lemma \ref{lem: wvandis1}, Part (c) of Theorem \ref{lem: b2lse} is established.

\noindent This completes the proof of Theorem \ref{lem: b2lse}.  $\blacksquare$

\subsection{Proof of Theorem \ref{thm: adapdsbm}} \label{subsec: adapdsbm}

Suppose $h>0$. Then,

\begin{eqnarray}
\tilde{L}^{*}(\hat{\tau}_{m,n} + h/n, \hat{z},\hat{w},\hat{\hat{\Lambda}},\hat{\hat{\Delta}}) &=& \frac{1}{n} \sum_{i,j=1}^{m} \bigg[\sum_{t=1}^{n\hat{\tau}_{m,n} +h} (A_{ij,(t,n),\text{DSBM}} - \hat{\hat{\lambda}}_{\hat{z}(i)\hat{z}(j)})^2 \nonumber \\
&& \hspace{1.5 cm}+ \sum_{t=n\hat{\tau}_{m,n} + h +1}^{n} (A_{ij,(t,n),\text{DSBM}}-\hat{\hat{\delta}}_{\hat{w}(i)\hat{w}(j)})^2 \bigg]. \nonumber 
\end{eqnarray}
We then have
\begin{eqnarray}
&& \tilde{L}^{*}(\hat{\tau}_{m,n} + h/n, \hat{z},\hat{w},\hat{\hat{\Lambda}},\hat{\hat{\Delta}}) - \tilde{L}^{*}(\hat{\tau}_{m,n}, \hat{z},\hat{w},\hat{\hat{\Lambda}},\hat{\hat{\Delta}}) \nonumber \\
& = & \frac{1}{n} \sum_{i,j=1}^{m} \sum_{t=n\hat{\tau}_{m,n}+1}^{n\hat{\tau}_{m,n}+h} \bigg[(A_{ij,(t,n),\text{DSBM}} - \hat{\hat{\lambda}}_{\hat{z}(i)\hat{z}(j)})^2 -  (A_{ij,(t,n),\text{DSBM}}-\hat{\hat{\delta}}_{\hat{w}(i)\hat{w}(j)})^2\bigg] \nonumber \\
& = & \frac{1}{n} \sum_{i,j=1}^{m} \sum_{t=n\hat{\tau}_{m,n}+1}^{n\hat{\tau}_{m,n}+h} \bigg[(\hat{\hat{\delta}}_{\hat{w}(i)\hat{w}(j)} - \hat{\hat{\lambda}}_{\hat{z}(i)\hat{z}(j)})^2 \nonumber \\
&& \hspace{5 cm}+ 2 (A_{ij,(t,n),\text{DSBM}}-\hat{\hat{\delta}}_{\hat{w}(i)\hat{w}(j)})(\hat{\hat{\delta}}_{\hat{w}(i)\hat{w}(j)} - \hat{\hat{\lambda}}_{\hat{z}(i)\hat{z}(j)})\bigg]. \nonumber
\end{eqnarray}
\noindent Let $E^{**}(\cdot) =E(\cdot| \hat{z},\hat{w})$,  $V^{**} = V(\cdot|\hat{z},\hat{w})$ and $\text{Cov}^{**}(\cdot) = \text{Cov}(\cdot|\hat{z},\hat{w})$.  \vskip 2pt
\noindent Therefore,
\begin{eqnarray}
 E^{**}(\tilde{L}^{*}(\hat{\tau}_n + h/n, \hat{z},\hat{w},\hat{\hat{\Lambda}},\hat{\hat{\Delta}}) - \tilde{L}^{*}(\hat{\tau}_n, \hat{z},\hat{w},\hat{\hat{\Lambda}},\hat{\hat{\Delta}})) 
&=& \frac{h}{n} ||\text{Ed}_{\hat{z}}(\hat{\hat{\Lambda}}) - \text{Ed}_{\hat{w}}(\hat{\hat{\Delta}})||_F^2. \nonumber 
\end{eqnarray}
Note that all entries of $\hat{\hat{\Lambda}}$ and $\hat{\hat{\Delta}}$ are bounded away from $0$ and $1$, since $\log m = o(\sqrt{n})$. Therefore, 
\begin{eqnarray}
&& V^{**}(\tilde{L}^{*}(\hat{\tau}_{m,n} + h/n, \hat{z},\hat{w},\hat{\hat{\Lambda}},\hat{\hat{\Delta}}) - \tilde{L}^{*}(\hat{\tau}_{m,n}, \hat{z},\hat{w},\hat{\hat{\Lambda}},\hat{\hat{\Delta}}))  \nonumber \\
&=& \frac{h}{n^2} \sum_{i,j=1}^{m} (\hat{\hat{\lambda}}_{\hat{z}(i)\hat{z}(j)}-\hat{\hat{\delta}}_{\hat{w}(i)\hat{w}(j)})^2 \hat{\hat{\delta}}_{\hat{w}(i)\hat{w}(j)}(1-\hat{\hat{\delta}}_{\hat{w}(i)\hat{w}(j)})  \nonumber \\
& \leq & \frac{h}{n^2} ||\text{Ed}_{\hat{z}}(\hat{\hat{\Lambda}}) - \text{Ed}_{\hat{w}}(\hat{\hat{\Delta}})||_F^2. \nonumber
\end{eqnarray}
Hence, by Lemma \ref{lem: wvan1} and similar arguments to those made at the beginning of  Section \ref{subsec: b1}, we have
\begin{eqnarray}
||\text{Ed}_{\hat{z}}(\hat{\hat{\Lambda}}) - \text{Ed}_{\hat{w}}(\hat{\hat{\Delta}})||_F^2 h = O_{\text{P}}(1).\nonumber 
\end{eqnarray}
Then, by Lemma \ref{lem: adaplem}(a), 
\begin{eqnarray}
||\text{Ed}_{z}(\Lambda) - \text{Ed}_{w}(\Delta)||_F^2 h = O_{\text{P}}(1).\nonumber 
\end{eqnarray}
This implies Theorem \ref{thm: adapdsbm}(a). 

Next, we establish Theorem \ref{thm: adapdsbm}(b).   Note that
\begin{eqnarray}
&& n(\tilde{L}^{*}(\hat{\tau}_{m,n} + h||\text{Ed}_{\hat{z}}(\hat{\hat{\Lambda}})-\text{Ed}_{\hat{w}}(\hat{\hat{\Delta}})||_F^{-2}/n, \hat{z},\hat{w},\hat{\hat{\Lambda}},\hat{\hat{\Delta}}) - \tilde{L}^{*}(\hat{\tau}_{m,n}, \hat{z},\hat{w},\hat{\hat{\Lambda}},\hat{\hat{\Delta}}))    
\nonumber \\
&=& 
 -|h| -2 \sum_{t=n\hat{\tau}_{m,n}+1}^{n\hat{\tau}_{m,n} + h} \sum_{i,j=1}^{m} (\hat{\hat{\lambda}}_{\hat{z}(i)\hat{z}(j)} - \hat{\hat{\delta}}_{\hat{w}(i)\hat{w}(j)}) ({A}_{ij,(t,n),\text{DSBM}}-\hat{\hat{\delta}}_{\hat{w}(i)\hat{w}(j)}) \nonumber \\
 && \hspace{8 cm}+ o_{\text{P}}(1), 
\end{eqnarray}
Further, note that given $\{A_{t,n}\}$,  $\{\sum_{i,j=1}^{m} (\hat{\hat{\lambda}}_{\hat{z}(i)\hat{z}(j)} - \hat{\hat{\delta}}_{\hat{w}(i)\hat{w}(j)}) ({A}_{ij,(t,n),\text{DSBM}}-\hat{\hat{\delta}}_{\hat{w}(i)\hat{w}(j)})\}$ is a collection of independent random variables.  By  Lemma \ref{lem: adaplem}(b), we have 
\begin{align*}
&& E \sum_{t=n\hat{\tau}_{m,n}+1}^{n\hat{\tau}_{m,n}+h} \sum_{i,j=1}^{m} (\hat{\hat{\lambda}}_{\hat{z}(i)\hat{z}(j)} - \hat{\hat{\delta}}_{\hat{w}(i)\hat{w}(j)}) ({A}_{ij,(t,n),\text{DSBM}}-\hat{\hat{\delta}}_{\hat{w}(i)\hat{w}(j)})\nonumber \\
& =&  E\bigg[\sum_{t=n\hat{\tau}_{m,n}+1}^{n\hat{\tau}_{m,n}+h} \sum_{i,j=1}^{m} (\hat{\hat{\lambda}}_{\hat{z}(i)\hat{z}(j)} - \hat{\hat{\delta}}_{\hat{w}(i)\hat{w}(j)}) E^{**}({A}_{ij,(t,n),\text{DSBM}}-\hat{\hat{\delta}}_{\hat{w}(i)\hat{w}(j)})\bigg] = 0, 
\end{align*}
\begin{align*}
 && \text{V}\bigg(\sum_{t=n\hat{\tau}_{m,n}+1}^{n\hat{\tau}_{m,n}+h} \sum_{i,j=1}^{m} (\hat{\hat{\lambda}}_{\hat{z}(i)\hat{z}(j)} - \hat{\hat{\delta}}_{\hat{w}(i)\hat{w}(j)}) ({A}_{ij,(t,n),\text{DSBM}}-\hat{\hat{\delta}}_{\hat{w}(i)\hat{w}(j)})\bigg) \nonumber \\
&=& hE\bigg(||\text{Ed}_{\hat{z}}(\hat{\hat{\Lambda}}) - \text{Ed}_{\hat{w}}(\hat{\hat{\Delta}})||_F^{-2} \sum_{i,j=1}^{m} (\hat{\hat{\Lambda}}_{\hat{z}(i)\hat{z}(i)} -\hat{\hat{\Delta}}_{\hat{w}(i)\hat{w}(i)})^2  \hat{\hat{\delta}}_{\hat{w}(i)\hat{w}(i)}(1-\hat{\hat{\delta}}_{\hat{w}(i)\hat{w}(i)}) \bigg) \to h \gamma^2, 
\end{align*}
\begin{eqnarray}
&& \frac{E\bigg[\sum_{t=n\hat{\tau}_{m,n}+1}^{n\hat{\tau}_{m,n}+h} \sum_{i,j=1}^{m} (\hat{\hat{\lambda}}_{\hat{z}(i)\hat{z}(j)} - \hat{\hat{\delta}}_{\hat{w}(i)\hat{w}(j)}) ({A}_{ij,(t,n),\text{DSBM}}-\hat{\hat{\delta}}_{\hat{w}(i)\hat{w}(j)})\bigg]^3}{\bigg[\text{V}\bigg(\sum_{t=n\hat{\tau}_{m,n}+1}^{n\hat{\tau}_{m,n}+h} \sum_{i,j=1}^{m} (\hat{\hat{\lambda}}_{\hat{z}(i)\hat{z}(j)} - \hat{\hat{\delta}}_{\hat{w}(i)\hat{w}(j)}) ({A}_{ij,(t,n),\text{DSBM}}-\hat{\hat{\delta}}_{\hat{w}(i)\hat{w}(j)})\bigg)\bigg]^{3/2}} \nonumber \\
 & \leq &  C E\left(\frac{\sum_{i,j=1}^{m}|\hat{\hat{\lambda}}_{\hat{z}(i)\hat{z}(j)} - \hat{\hat{\delta}}_{\hat{w}(i)\hat{w}(j)}|^3}{||\text{Ed}_{\hat{z}}(\hat{\hat{\Lambda}})-\text{Ed}_{\hat{w}}(\hat{\hat{\Delta}})||_F^{2}}\right)  \nonumber \\
 &\leq  & C (E||\text{Ed}_{\hat{z}}(\hat{\hat{\Lambda}})-\text{Ed}_{\hat{w}}(\hat{\hat{\Delta}})||_F^{2})^{1/2} \nonumber \\
&\leq & C\left( E(||\text{Ed}_{z}(\Lambda) - \text{Ed}_{\hat{z}}{\hat{\hat{\Lambda}}}||_F^2) + E(||\text{Ed}_{w}(\Delta)- \text{Ed}_{\hat{w}}({\hat{\hat{\Delta}}})||_F^2)+ ||\text{Ed}_{z}(\Lambda)-\text{Ed}_{w}(\Delta)||_F^{2} \right) \to 0. \nonumber 
\end{eqnarray}
Hence, an application of Lyapunov's Central Limit Theorem together with (A1)-(A4) and SNR**-DSBM-ADAP yields
\begin{eqnarray}
n(\tilde{L}^{*}(\hat{\tau}_{m,n} + h||\text{Ed}_{\hat{z}}(\hat{\hat{\Lambda}})-\text{Ed}_{\hat{w}}(\hat{\hat{\Delta}})||_F^{-2}/n, \hat{z},\hat{w},\hat{\hat{\Lambda}},\hat{\hat{\Delta}}) - \tilde{L}^{*}(\hat{\tau}_{m,n}, \hat{z},\hat{w},\hat{\hat{\Lambda}},\hat{\hat{\Delta}})) \nonumber \\
\Rightarrow -|h| + \gamma B_h \label{eqn: adaplseb1}
\end{eqnarray}
Similar arguments are applicable for the case of $h<0$. 
\vskip 2pt
\noindent Finally, (\ref{eqn: adaplseb1})  in conjunction with Lemma \ref{lem: wvandis1} establish Theorem \ref{thm: adapdsbm}(b). 

\noindent An analogous argument to that in the proof of Theorem \ref{lem: b2lse}(c) together with similar approximations as in the proof of Theorem \ref{thm: adapdsbm}(b)  establish Theorem \ref{thm: adapdsbm}(c) and hence they are omitted.
\noindent Hence, Theorem \ref{thm: adapdsbm} is established. $\blacksquare$

\subsection{Assumptions for the asymptotic distribution of change-point estimators} \label{subsec: assumption}

Next, we provide precise statements of Assumptions (A3)-(A7) required for establishing the asymptotic distribution of the change point estimators in Theorem \ref{lem: b2lse}. 
A brief comment on these assumptions is given after stating them. We refer to \cite{BBM2017}  for more in depth explanation.
\vskip 2pt
For Regime II,  we define 
\begin{eqnarray}
\gamma^2 &=& \lim \frac{\sum_{i,j=1}^{m}(\lambda_{z(i)z(j)}-\delta_{w(i)w(j)})^2\lambda_{z(i)z(j)}(1-\lambda_{z(i)z(j)})}{\sum_{i,j=1}^{m}(\lambda_{z(i)z(j)}-\delta_{w(i)w(j)})^2} \nonumber \\
&=& \lim \frac{\sum_{i,j=1}^{m}(\lambda_{z(i)z(j)}-\delta_{w(i)w(j)})^2\delta_{w(i)w(j)}(1-\delta_{w(i)w(j)})}{\sum_{i,j=1}^{m}(\lambda_{z(i)z(j)}-\delta_{w(i)w(j)})^2}, \nonumber 
\end{eqnarray}
and assume  that
\vskip 2pt
\noindent \textbf{(A3)} $\gamma^2$ exists.
\vskip 2pt
\noindent In Regime II, the asymptotic variance of the change-point estimator is proportional to $\gamma^2$. Hence, we require (A3) for its existence and (A2) for the non-degeneracy of the asymptotic distribution. 

In Regime III, we consider the following set of edges
\begin{eqnarray}
\mathcal{K}_n = \{(i,j): 1 \leq i,j \leq m,\ \  
|\lambda_{z(i)z(j)} - \delta_{w(i)w(j)}| \to 0\}.
\end{eqnarray} 
Define 
\begin{eqnarray}
c_1^2 &=& \lim \sum_{i,j \in \mathcal{K}_n} (\lambda_{z(i)z(j)}-\delta_{w(i)w(j)})^2  \ \ \text{and}\ \label{eqn: c1defn} \\
\tilde{\gamma}^2  &=& \lim \sum_{i,j \in \mathcal{K}_n}(\lambda_{z(i)z(j)}-\delta_{w(i)w(j)})^2\lambda_{z(i)z(j)}(1-\lambda_{z(i)z(j)}) \nonumber \\
&=& \lim \sum_{i,j \in \mathcal{K}_n}(\lambda_{z(i)z(j)}-\delta_{w(i)w(j)})^2\delta_{w(i)w(j)}(1-\delta_{w(i)w(j)}). \nonumber  
\end{eqnarray}
\noindent Consider the following assumptions. \\
\noindent \textbf{(A4)} $c_1$ and $\tilde{\gamma}$ exist.  \\
\noindent \textbf{(A5)} $\sup_{ij \in \mathcal{K}_n} | \lambda_{z(i)z(j)} - \delta_{w(i)w(j)}| \to 0$. \\
\noindent \textbf{(A6)} $\mathcal{K}_0 = \mathcal{K}_n^c$ does not vary with $n$. \\
\noindent \textbf{(A7)} For some $\tau^{*} \in (c^{*},1-c^{*})$, $\tau_n \to \tau^{*}$ as $n \to \infty$. Suppose $\lambda_{z(i)z(j)} \to a_{ij,1}^{*}$ and $\delta_{w(i)w(j)} \to a_{ij,2}^{*}$ for all $(i,j) \in \mathcal{K}_0$.  

In Regime III, we need to treat edges in $\mathcal{K}_n$ and $\mathcal{K}_0 = \mathcal{K}_n^c$ separately. Note that in Regime II, $\mathcal{K}_n = \{(i,j):\ 1\leq i,j \leq m\}$  is the set of all edges.  Hence, we can treat $\mathcal{K}_n$ in a similar way as in Regime II and hence we need (A4) in Regime III analogous to (A3) in Regime II. (A5) is a technical assumption and is required for establishing asymptotic normality on $\mathcal{K}_n$.  Moreover,  $\mathcal{K}_0$ is a finite set. (A6) guarantees that $\mathcal{K}_0$ does not vary with $n$. Consider the collection of independent Bernoulli random variables $\{A_{ij,l}^{*}: (i,j) \in \mathcal{K}_0, l=1,2\}$ with $E(A_{ij,l}^{*}) = a_{ij,l}^{*}$.  (A7) ensures that  $A_{ij,(\lfloor nf \rfloor, n)} \stackrel{\mathcal{D}}{\to} A_{ij,1}^{*}I(f<\tau^{*}) + A_{ij,2}^{*}I(f > \tau^{*})\ \forall (i,j) \in \mathcal{K}_0$.

\begin{remark} \label{rem: sparse}
Note that (A2) is a crucial assumption for establishing the asymptotic distribution of the change-point estimator. It indicates that the resulting random graphi's topology. However, another regime of interest is that where the expected degree of each node grows slower than the total number of nodes in the graph, which gives rise to a {\em sparse regime}.  A number of technical results both from probabilistic and statistical viewpoints have been considered in the
recent literature - see, for examples \cite{SB2015} and \cite{LLV2017}. Note that results strongly diverge in their conclusions under these two regimes. For example,  \cite{O2009} showed that the inhomogeneous Erd\H{o}s-R\'{e}nyi model satisfies
\begin{eqnarray}
||L(A) - L(EA)|| = O\left(\sqrt{\frac{\log m}{d_0}} \right) \nonumber 
\end{eqnarray}
with high probability, where $m$ is the total number of nodes in the graph, $d_0 = \min_{i} \sum_{j=1}^{m} EA_{ij}$, $A$ is the observed adjacency matrix, $L(\cdot)$ is the Laplacian and $||\cdot||$ is the operator norm. Therefore, if the expected  degrees  are growing slower than $\log m$, $L(A)$ will not be concentrated around $L(EA)$.  \cite{LLV2017} established a different concentration inequality for the case $d_0 = o(\log m)$ after appropriate regularization on the Laplacian and the edge probability matrix. \cite{SB2015} also established the convergence rate of the  eigenvectors of the Laplacian  for SBM
 with two communities, which deviates from existing results for dense random graphs.  The upshot is that results for the sparse regime are
markedly different than those for the dense one. 

It is worth noting that the convergence rate results established in Sections \ref{sec: DSBM} and \ref{sec: 2step} hold also for the sparse setting; however,
establishing the  asymptotic distribution of the change-point estimate in a sparse setting, together with issues of adaptive inference will require further work.
\end{remark}

\bibliographystyle{natbib}

\end{document}